%% file: main.tex
\documentclass[12pt]{amsart}
\usepackage[utf8]{inputenc} 
\usepackage{amsmath}
\usepackage{amssymb, enumerate, verbatim}
\usepackage{enumitem}
\usepackage{amsfonts}
\usepackage{amsthm}
\usepackage[a4paper]{geometry}
\usepackage{microtype}
\usepackage{graphicx}
\usepackage{tabularx}
\usepackage{array}
\usepackage{caption}
\usepackage{subcaption}
\usepackage{lipsum}
\usepackage{float} 
\usepackage[normalem]{ulem}
\usepackage[english]{babel}
\usepackage[cmtip,all]{xy}
\usepackage{tikz-cd}
\usepackage{mathrsfs}
\usepackage{hyperref}
\usepackage{xcolor}
\usepackage[textsize=small]{todonotes}
\DeclareRobustCommand{\SkipTocEntry}[5]{}

\theoremstyle{plain}
\newtheorem{Thm}{Theorem}[section]
\newtheorem*{unThm}{Theorem}

\newtheorem{Lem}[Thm]{Lemma}

\newtheorem{Cor}[Thm]{Corollary}
\newtheorem*{unCor}{Corollary}
\newtheorem{Prop}[Thm]{Proposition}

\newtheorem*{unRem}{Remark}

\theoremstyle{definition}
\newtheorem{Rem}[Thm]{Remark}

\newtheorem{Def}[Thm]{Definition}
\newtheorem*{unDef}{Definition}
\newtheorem{Not}[Thm]{Notation}

\newtheorem{Ex}[Thm]{Example}

\newenvironment{claim}[1]{\par\noindent\underline{Claim:}\space#1}{}
\newenvironment{claimproof}[1]{\par\noindent{Proof:}\space#1}

\def \Rr {\ensuremath{\mathbb R }}
\def \Qq {\ensuremath{\mathbb Q }}
\def \Zz {\ensuremath{\mathbb Z }}
\def \Nn {\ensuremath{\mathbb N }}
\def \Ff {\ensuremath{\mathbb F }}

\def \Kk {\ensuremath{\mathbb{K}}}
\def \Ll {\ensuremath{\mathbb{L}}}

\def \- {\ensuremath{\text{-}}}
\def \a {\ensuremath{\alpha}}

\def \d {\ensuremath{\delta}}

\def \e {\ensuremath{\varepsilon}}

\def \Fev {\text{For every}}
\def \fev {\text{for every}}

\def \g {\ensuremath{\gamma}}
\def \G {\ensuremath{\Gamma}}

\def \iff {\text{if and only if}}
\def \l {\ensuremath{\lambda}}

\def \p {\ensuremath{\phi}}

\def \pslr {\ensuremath{\mathrm{PSL}_2(\mathbb{R})}}
\def \pslf {\ensuremath{\mathrm{PSL}_2(\mathbb{F})}}
\def \pslk {\ensuremath{\mathrm{PSL}_2(\mathbb{K})}}
\def \pslnk {\ensuremath{\mathrm{PSL}_n(\mathbb{K})}}

\def \projspn {\ensuremath{\mathbb{P}^{n-1}}}

\def \projspnr {\ensuremath{\mathbb{P}^{n-1}(\Rr)}}
\def \projspr {\ensuremath{\mathbb{P}^{1}(\Rr)}}
\def \projspnf {\ensuremath{\mathbb{P}^{n-1}(\Ff)}}
\def \projspf {\ensuremath{\mathbb{P}^{1}(\Ff)}}
\def \projspnk {\ensuremath{\mathbb{P}^{n-1}(\Kk)}}
\def \projspk {\ensuremath{\mathbb{P}^{1}(\Kk)}}

\def \r {\ensuremath{\rho}}
\def \s {\ensuremath{\sigma}}

\def \slnr {\ensuremath{\mathrm{SL}_n(\mathbb{R})}}
\def \slnf {\ensuremath{\mathrm{SL}_n(\mathbb{F})}}
\def \slnk {\ensuremath{\mathrm{SL}_n(\mathbb{K})}}

\def \st {\text{such that}}

\def \sym {\ensuremath{\mathcal{P}^1(n)}}
\def \symr {\ensuremath{\mathcal{P}^1(n,\mathbb{R})}}

\def \symk {\ensuremath{\mathcal{P}^1(n,\mathbb{K})}}
\def \symf {\ensuremath{\mathcal{P}^1(n,\mathbb{F})}}
\def \symkr {\ensuremath{\mathcal{P}^1(n,\mathbb{K}_\rho)}}
\def \symfr {\ensuremath{\mathcal{P}^1(n,\mathbb{F}_\rho)}}
\def \tes {\text{there exists}}

\def \te {\text{there exist}}

\newcommand{\anar}[1]{\ensuremath{#1_\mathrm{R}^{\mathrm{an}}}}
\newcommand{\ana}[1]{\ensuremath{#1^{\mathrm{an}}}}

\newcommand{\com}[1]{}

\newcommand{\func}[3]{\ensuremath{#1\colon #2 \rightarrow #3}}

\newcommand{\Homt}[2]{\ensuremath{\mathrm{Hom}( #1 , #2)}}
\newcommand{\Hom}[3]{\ensuremath{\mathrm{Hom}_{#1}( #2 , #3)}}

\newcommand{\norm}[1]{\ensuremath{\left\lVert#1\right\rVert}}

\newcommand{\psc}[2]{\ensuremath{\left< \, #1 \ , \ #2 \, \right>}}

\newcommand{\rsp}[1]{\ensuremath{#1_{\mathrm{cl}}^{\mathrm{RSp}}}}

\newcommand{\set}[1]{\ensuremath{\{ \, #1 \, \}}}

\newcommand{\setfrac}[1]{\ensuremath{\left\{ \, #1 \, \right\}}}
\newcommand{\setrel}[2]{\ensuremath{\{ \, #1 \, | \, #2 \, \}}}
\newcommand{\setrelb}[2]{\ensuremath{\, \big\{ #1 \, \big| \, #2  \, \big\}}}
\newcommand{\setrelfrac}[2]{\ensuremath{ \, \left\{ #1 \, \middle| \, #2 \, \right\}}}
\newcommand{\specn}[1]{\ensuremath{\mathrm{Spec}( #1 )}}

\newcommand{\speccl}[1]{\ensuremath{#1_{\mathrm{cl}}^{\mathrm{RSp}}}}
\newcommand{\spec}[1]{\ensuremath{#1^{\mathrm{RSp}}}}
\newcommand{\specarch}[1]{\ensuremath{#1_{\mathrm{Arch}}^{\mathrm{RSp}}}}
\newcommand{\specclf}[1]{\ensuremath{( #1 )_{\mathrm{cl}}^{\mathrm{RSp}}}}
\newcommand{\specclff}[1]{\ensuremath{\left( #1 \right)_{\mathrm{cl}}^{\mathrm{RSp}}}}
\newcommand{\specf}[1]{\ensuremath{( #1 )^{\mathrm{RSp}}}}
\newcommand{\specff}[1]{\ensuremath{\left( #1 \right)^{\mathrm{RSp}}}}
\newcommand{\specarchf}[1]{\ensuremath{( #1 )_{\mathrm{Arch}}^{\mathrm{RSp}}}}
\newcommand{\specarchff}[1]{\ensuremath{\left( #1 \right)_{\mathrm{Arch}}^{\mathrm{RSp}}}}

\def\DefMap#1#2#3#4#5{\begin{matrix}#1 \colon&#2&\longrightarrow &#3;
\\ &#4 &\longmapsto &#5.
\end{matrix}}

\def\defmap#1#2#3#4#5{\begin{matrix}#1\colon&#2&\longrightarrow &#3;
\\& #4 &\longmapsto &#5,
\end{matrix}}

\def\map#1#2#3#4{\begin{matrix}#1&\longrightarrow &#2;
\\#3 &\longmapsto &#4,
\end{matrix}}

\title[Universal geometric space over character varieties]{Real spectrum compactifications of universal geometric spaces over character varieties}
\author{Victor Jaeck}
\date{\today} 
\address{Department of Mathematics, ETH Z\"{u}rich, Switzerland}
\email{victor.jaeck@math.ethz.ch}

\def\subjclassname{\textup{2020} Mathematics Subject Classification}
\expandafter\let\csname subjclassname@1991\endcsname=\subjclassname
\subjclass{
22E40, 
14P10, 
12J15, 
12J25, 
14G22. 
}

\begin{document}

\begin{abstract}
    We construct universal geometric spaces over the real spectrum compactification $\Xi^{\mathrm{RSp}}$ of the character variety $\Xi$ of a finitely generated group $\Gamma$ in $\mathrm{SL}_n$, providing geometric interpretations of boundary points.
    For an algebraic set $Y(\Rr)$ on which $\mathrm{SL}_n(\Rr)$ acts by algebraic automorphisms (such as $\mathbb{P}^{n-1}(\Rr)$ or an algebraic cover of the symmetric space of $\mathrm{SL}_n(\Rr)$), the projection map $\Xi \times Y \rightarrow \Xi$ extends to a $\Gamma$-equivariant continuous surjection $(\Xi \times Y)^{\mathrm{RSp}} \rightarrow \Xi^{\mathrm{RSp}}$. The fibers of this extended map are homeomorphic to the Archimedean spectrum of $Y(\mathbb{F})$ for some real closed field $\mathbb{F}$, which is a locally compact subset of $Y^{\mathrm{RSp}}$. The Archimedean spectrum is naturally homeomorphic to the real analytification, and we use this identification to compute the image of the fibers in their Berkovich analytification. For $Y=\mathbb{P}^1$, the image is a real subtree.
\end{abstract}
\maketitle
\tableofcontents
\section{Introduction}

    Let \G \ be a finitely generated group and \Hom{\mathrm{\mathrm{red}}}{\G} {\slnr} the space of reductive representations of \G \ in \slnr, equipped with the topology of pointwise convergence. The character variety $\Xi(\G,\slnr)$ is the topological quotient by~\slnr -postconjugation of \Hom{\mathrm{red}}{\G} {\slnr}. 
    This space provides a geometric point of view for studying representations of discrete groups in~$\slnr$ and unifies several key themes. It underpins Goldman and Mirzakhani’s work on the symplectic and hyperbolic geometry of moduli spaces \cite{Gthesymplecticnatureoffundamentalgroupsofsurfaces,Msimplegeodesicsandweilpetersson}; it is central to Thurston's and Culler--Morgan--Shalen's theory linking character varieties to the topology of manifolds and geometric group theory \cite{Tonthegeometryanddynamicsofdiffeomorphismsofsurfaces, CSvar, MSval}; it connects to non-abelian Hodge theory \cite{Hthe,Dtwi,Cfla,Shig}; and it bridges to modern mathematical physics \cite{FCaquantumTeichmullerspace} and the Langlands program \cite{BDqua,DPlanglandsduality}.
    
    In general, the space of representations is not compact, so there exist sequences of representations in the character variety that do not have a limit. To investigate their asymptotic behavior, we examine the real spectrum compactification $\rsp{\Xi(\G,\slnr)}$ of~$\Xi(\G,\slnr)$~\cite{Bthe,BIPPthereal}.
    Boundary elements
    \[
        (\r ,\Ff) \in \rsp{\partial\Xi(\G,\slnr)}
    \]
    correspond to classes of reductive representations of \G \ in \slnf, where \Ff \ is a well chosen non-Archimedean real closed field \cite{Bthe,BIPPthereal}; see Section \ref{Section algebraic model} for the definitions.
    While this describes the boundary points algebraically, a central challenge raised by Wienhard at the ICM \cite{Wani} is to provide a geometric interpretation of the boundary points of this compactification.
    This article takes a step towards this goal by analyzing an algebraic subset of minimal vectors
    \[
        \mathcal{M}_\G(\Rr) \subset \mathrm{Hom}_{\mathrm{red}}(\G,\slnr),
    \]
    which is a cover of $\Xi(\G,\pslr)$. More precisely $\mathcal{M}_\G(\Rr) / \mathrm{SO}(n,\Rr)$ is semialgebraically isomorphic to $\Xi(\G,\pslr)$.
    We examine the geometric actions induced by $\mathcal{M}_\G(\Rr)$ on $\projspnr$, the $(n-1)$-dimensional projective space (or on an algebraic cover $\widehat{\symr}$ of the symmetric space \symr \ of \slnr, see Subsection \ref{Subsection: Displacement of an action and its associated geometric space} and Subsection \ref{Subsection: The projection map and the Archimedean spectrum of the symmetric space}). 
    The algebraic group \slnr \ acts on $\projspnr$ by algebraic automorphisms so that any element of $\mathcal{M}_\G(\Rr)$ induces a \G -action on $\projspnr$. We call 
    \[
        \mathcal{M}_\G(\Rr) \times \projspnr
    \]
    the \emph{universal projective space} over $\mathcal{M}_\G(\Rr)$. It is an algebraic set that parametrizes all the \G -actions induced by elements of $\mathcal{M}_\G(\Rr)$ on \projspnr, as illustrated in Figure~\ref{Figure: universal geomtric space over MGAMMA} (see Subsection \ref{Subsection: The universal projective space}). 
    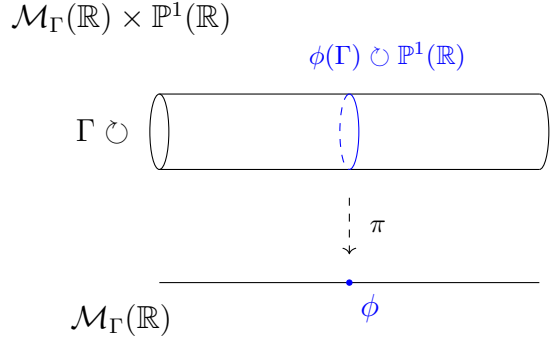
\begin{figure}[H]
    \centering
    \begin{tikzpicture}[scale=0.5]

        \draw (-5,2) -- (5,2);
        \draw (-5,0) -- (5,0);

        \draw (-5.25,1) arc[start angle=-180,end angle=180,x radius=0.25,y radius=1];

        \draw[blue] (0,0) arc[start angle=-90,end angle=90,x radius=0.25,y radius=1];
        \draw[dashed,blue] (0,2) arc[start angle=90,end angle=270,x radius=0.25,y radius=1];

        \draw (5,0) arc[start angle=-90,end angle=90,x radius=0.25,y radius=1];

        \draw (-5,-3) -- (5,-3);
        \filldraw[blue] (0,-3) circle(2pt) node[anchor = north west] {$\phi$};

        \draw (-6,4) node {$\mathcal{M}_\G(\Rr)\times \mathbb{P}^1(\mathbb{R})$};
        \draw (-6,-4) node {$\mathcal{M}_\G(\Rr)$};

        \draw[blue] (1,3) node {\footnotesize $\phi(\Gamma) \circlearrowright \mathbb{P}^1(\mathbb{R})$}; 

        \draw (-5.5,1) node[left] { $\Gamma \circlearrowright$}; 

        \draw[dashed,->] (0,-0.75) -- (0,-2.25); 
        \draw (0.75,-1.5) node {\footnotesize $\pi$}; 

    \end{tikzpicture}
    \caption{Schematic picture of the universal projective line ($n=2$) over $\mathcal{M}_\G(\Rr)$, with the fiber over $\phi$ in blue.} \label{Figure: universal geomtric space over MGAMMA}
    \end{figure}
    A natural question is whether such a geometric space exists over $\rsp{\mathcal{M}_\G(\Rr)}$ and whether it contains information about the \G -actions induced by elements in the boundary of $\rsp{\mathcal{M}_\G(\Rr)}$. 
    We provide this universal geometric space by considering 
    \[
        \rsp{(\mathcal{M}_\G(\Rr) \times \projspnr)}
    \]
    and analyzing its geometric and dynamical properties. The \emph{Archimedean spectrum} $\specarch{\projspnf}$ of $\projspnf$ is a dense and locally compact open subset of $\rsp{\projspnr}$ that contains $\projspnf$ with desirable geometric properties. A detailed study of this space is given in Section~\ref{Section relation between the real spectrum of minimal vectors and the Archimedean spectrum of their associated symmetric space}. Our main result, illustrated in Figure \ref{Figure: universal geometric space over the compactification}, can be summarized as follows:

    \begin{unThm}[Theorem \ref{Thm: universal projective space over the real spectrum}]
        The projection map $\func{\pi}{\mathcal{M}_\G(\Rr) \times \projspnr}{ \mathcal{M}_\G(\Rr)}$ extends to a continuous map
        \[
            \speccl{\pi}: \rsp{(\mathcal{M}_\G(\Rr) \times \projspnr)} \rightarrow \rsp{\mathcal{M}_\G(\Rr)},
        \]
        which is surjective, \G -equivariant, with the fiber over $(\r,\Ff)\in \rsp{\mathcal{M}_\G(\Rr)}$ homeomorphic to $\specarch{\projspnf}$.
    \end{unThm}

    Using this theorem, we analyze $\rsp{\mathcal{M}_\G(\Rr)}$ through the study of its fibers in $\rsp{(\mathcal{M}_\G(\Rr) \times \projspnr)}$.
    We understand the \G -actions on \specarch{\projspnf} induced by representations in $\rsp{\mathcal{M}_\G(\Rr)}$ by comparing them to \G -actions on \projspnr \ coming from representations in $\mathcal{M}_\G(\Rr)$. 

    \begin{unCor}[Corollary \ref{Cor: convergence of the fibers in the universal projective space}] 
        Let $(\r_n,\Kk_n) \subset \rsp{\mathcal{M}_\G(\Rr)}$ be a sequence converging to $(\r,\Kk) \in \rsp{\mathcal{M}_\G(\Rr)}$. For every element of the fiber $(\r,A,\Ff) \in (\speccl{\pi})^{-1}(\r,\Kk)$, \tes \ a sequence $(\r_n,A_n,\Ff_n) \in (\speccl{\pi})^{-1}(\r_n,\Kk_n)$ such that $(\r_n,A_n,\Kk_n)$ converges to $(\r,A,\Ff)$.
    \end{unCor}

    Both results hold if $\projspnr$ is replaced by the algebraic cover $\widehat{\symr}$ of the symmetric space \symr \ of \slnr; see Subsection \ref{Subsection: Displacement of an action and its associated geometric space} and \ref{Subsection: The projection map and the Archimedean spectrum of the symmetric space} for the precise statements.

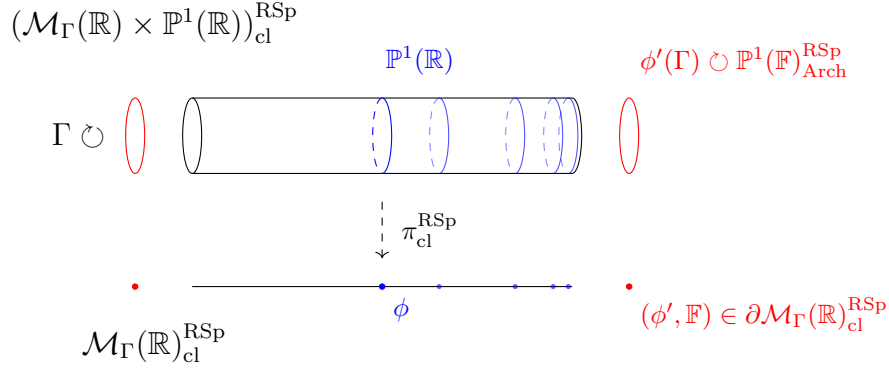
\begin{figure}[ht]
    \centering
    \begin{tikzpicture}[scale=0.5]

        \draw (-5,2) -- (5,2);
        \draw (-5,0) -- (5,0);

        \draw (-5.25,1) arc[start angle=-180,end angle=180,x radius=0.25,y radius=1];

        \draw[blue] (0,0) arc[start angle=-90,end angle=90,x radius=0.25,y radius=1];
        \draw[dashed,blue] (0,2) arc[start angle=90,end angle=270,x radius=0.25,y radius=1];

        \draw (5,0) arc[start angle=-90,end angle=90,x radius=0.25,y radius=1];

        \foreach \x in {1.5, 3.5, 4.5, 4.90} {
            \pgfmathsetmacro\scale{0.9 - 0.15*\x}
            \draw[blue, opacity=0.7] (\x,0) arc[start angle=-90,end angle=90,x radius=0.25,y radius=1];
            \draw[dashed,blue, opacity=0.5] (\x,2) arc[start angle=90,end angle=270,x radius=0.25,y radius=1];
        }

        \draw[red] (-6.5,1) ellipse (0.25cm and 1cm);
        \draw[red] (6.5,1) ellipse (0.25cm and 1cm);

        \draw (-5,-3) -- (5,-3);
        \filldraw[red] (-6.5,-3) circle(2pt);
        \filldraw[red] (6.5,-3) circle(2pt) node[anchor = north west] {\footnotesize $(\phi',\mathbb{F}) \in \partial \rsp{\mathcal{M}_\G(\Rr)}$};
        \filldraw[blue] (0,-3) circle(2pt) node[anchor = north west] {\footnotesize $\phi$};

        \foreach \x in {1.5, 3.5, 4.5, 4.90} {
            \filldraw[blue, opacity=0.6] (\x,-3) circle(1.5pt);
        }

        \draw (-6,4) node {$\rsp{\left(\mathcal{M}_\G(\Rr)\times \mathbb{P}^1(\mathbb{R})\right)}$};
        \draw (-6,-4.5) node {$\rsp{\mathcal{M}_\G(\Rr)}$};

        \draw[blue] (1,3) node {\footnotesize $\projspr$};

        \draw (-7,1) node[left] { $\Gamma \circlearrowright$};

        \draw[red] (6.5,3) node[anchor=west] {\footnotesize $\phi'(\G) \circlearrowright \specarch{\projspf}$};

        \draw[dashed,->] (0,-0.75) -- (0,-2.25); 
        \draw (1.25,-1.5) node {\footnotesize $\speccl{\pi}$}; 

    \end{tikzpicture}
    \caption{Schematic extension of the universal projective line over $\mathcal{M}_\G(\Rr)$ to the real spectrum, illustrating the convergence of interior fibers (in blue) to the fiber over the boundary point (in red).} \label{Figure: universal geometric space over the compactification}
\end{figure}

    These results establish a connection between the study of degenerations of representations in character varieties and techniques from analytic geometry.
    Building on the natural homeomorphism between $\specarch{\projspnf}$ and the real analytification of $\projspnf$ \cite{JSYrea}, we construct a proper continuous map
    \[   
        \func{\psi}{\specarchf{\projspnf}}{\ana{\projspnf}},    
    \]
    where $\ana{\projspnf}$ is the Berkovich analytification of $\projspnf$---a connected, locally compact space containing \projspnf \ and endowed with an analytic structure. 
    Theorem \ref{Thm: image of the archimedean spectrum in the anlaytification} provides a description of the image of $\psi$; in the special case of \projspf, this yields the following result:
    \begin{unCor}
        Let \Ff \ be a non-Archimedean real closed field endowed with an absolute value. 
        The image under $\psi$ of $\specarch{\projspf}\backslash \projspf$ inside $\ana{\projspf} \backslash \projspf$ is a closed, \pslf -invariant real subtree. 
    \end{unCor}
    The natural \G-action on \specarch{\projspnf}, induced by an element of $\rsp{\mathcal{M}_\G(\Rr)}$, provides an analytic perspective to the study of $\rsp{\mathcal{M}_\G(\Rr)}$ with the aim of identifying dynamical properties of boundary representations. The following subsections outline these constructions and the main methods used in this article.
    
    \input{Introduction}

    \addtocontents{toc}{\SkipTocEntry}
    \subsection*{Acknowledgments}
    I would like to thank Raphael Appenzeller, Claus Scheiderer, and Johannes Schmitt for their helpful discussions and Konstantin Andritsch, Segev Gonen Cohen, Gabriel Ribeiro, and Samir Canning for their comments on this text. I am also grateful to Gilles Courtois, Antonin Guilloux, and Vlerë Mehmeti for inviting me to speak at a workshop in Paris, which led to valuable exchanges that influenced this work.

\section{Character varieties,  minimal vectors, and symmetric spaces} \label{Section algebraic model}

    This section provides definitions of minimal vectors $\mathcal{M}_\G(\Rr)$ of representations of a finitely generated group~\G \ in \slnr , their quotient the character varieties $\Xi(\G,\slnr)$, and of a cover $\widehat{\symr}$ of the symmetric space associated to \slnr . We show that $\mathcal{M}_\G(\Rr)$ and $\widehat{\symr}$ are algebraic sets while $\Xi$ is a semialgebraic set. In particular, this allows their study in the following sections using tools from real algebraic geometry. 
    \input{Algebraic_sets}

    \input{Character_varieties_and_minimal_vectors}
    \input{Symmetric_space_associated_to_SLNR}

\section{The real spectrum compactification and the Archimedean spectrum}\label{Section: Real spectrum and Archimedean spectrum}
    We want to understand the degeneration of minimal vectors of representations. To do this, it is interesting to understand the manner in which representations go to infinity. To understand this behavior, we study the real spectrum compactification and the Archimedean spectrum of minimal vectors. The former has the advantage of providing a natural compactification for real points of an algebraic set with good topological properties. The second provides a natural framework for studying the \Ff -extension of the algebraic set when \Ff \ is a non-Archimedean real closed field, and provides a better understanding of the geometry and dynamics of its \Ff -points.
    \input{Real_spectrum_of_a_ring}

    \input{Real_spectrum_and_archimedean_spetrum_of_algebraic_sets}

\section{Universal geometric spaces over the real spectrum of $\mathcal{M}_\G(\Rr)$} \label{Section relation between the real spectrum of minimal vectors and the Archimedean spectrum of their associated symmetric space}

    In this section, we study the real spectrum compactification of $\mathcal{M}_\G(\Rr)$, the algebraic set of minimal vectors of a finitely generated group \G \ in \slnr, its interaction with compact algebraic sets (as the projective space $\projspn$), and the symmetric space $\widehat{\sym}$. 
    A key focus is on providing a geometric understanding of the convergence of sequences in the real spectrum. Specifically, we demonstrate that representations on the boundary of $\rsp{\mathcal{M}_\G(\Rr)}$ induce actions of \G \ on $\specarch{\projspn(\Kk)}$ and \specarch{\widehat{\symk}}, where \Kk \ is a well-chosen non-Archimedean real closed field.
    This motivates the relationship we exhibit in Section \ref{Section: The Archimedean spectrum and the real analytification of an algebraic set} between the Archimedean spectrum and the Berkovich analytification.

    \input{The_universal_projective_space}
    \input{displacement_function}
    \input{The_universal_symmetric_space}

\section{The Archimedean spectrum and the analytification of an algebraic set}\label{Section: The Archimedean spectrum and the real analytification of an algebraic set}

    The Berkovich analytification offers a framework for applying complex analysis techniques to the study of algebraic sets defined over ultrametric fields. Similarly, the real analytification provides its real counterpart for studying algebraic sets defined over non-Archimedean real fields using analytic methods. 
    In this section, we explore the identification between the Archimedean spectrum and the real analytification of an algebraic set. We further examine its relationship with the Berkovich analytification. Specifically, we compute the image of the Archimedean spectrum within the Berkovich analytification and derive actions on real trees.

    \input{Berkovich_and_real_analytification}
    \input{Archimedean_spectrum_and_Berkovich_analytification}

    \input{Application_to_the_universal_projective_space}

\bibliography{ReferenceT} 
\bibliographystyle{alpha}  

\end{document}

%% file: Introduction.tex
\subsection{Real spectrum compactification}
    Introduced by Coste and Roy in \cite{CClespectrereeletlatopologie, CRlatopologieduspectrereel}, the real spectrum compactification applies to semialgebraic sets and preserves the structure determined by the equalities and inequalities that define them. In other words, the semialgebraic properties that characterize interior points extend naturally to points at infinity.
    To simplify technical details, we present the definitions for algebraic sets and refer to Section~\ref{Subsection The real spectrum of a semialgebraic set} for further details. 
    A field $\Ff$ is a \emph{real field} if it is endowed with an order compatible with its field operations. Moreover, \Ff \ is \emph{real closed} if it has no proper algebraic ordered field extension, see Subsection \ref{Subsection: preliminaries in real algebraic geometry}. As an example, one may consider $\Ff = \Rr$ in what follows. Let $\Ll \subset \Kk$ be real closed fields and $V\subset \Ll^n$ an algebraic set. 
    Denote by $\Kk[V]$ the \emph{coordinate ring} of the \Kk-extension $V(\Kk)$ of $V$, and, by abuse of notation, we use the same symbol for the \Kk-extension of $V$ and the set of \Kk-points $V(\Kk)$. We endow $V(\Kk)$ with its \emph{Euclidean topology}, see Subsection \ref{Subsection: preliminaries in real algebraic geometry}.
    \begin{unDef}[{\cite[Proposition 7.1.2]{BCRrea}}]
        Let $\Ll \subset \Kk$ be real closed fields. The \emph{real spectrum} of the \Kk -points of an algebraic set $V\subset \Ll^n$ is 
        \[
            \spec{V(\Kk)}:=\setrelfrac{(\r ,\Ff)}
            {\begin{tabular}{@{}c@{}}
                \Ff \text{ a real closed field,} \\
                \func{\r}{\Kk[V]}{\Ff} \text{ a ring morphism}
                \end{tabular}}/\sim,
        \]
        where $\sim$ is the smallest equivalence relation such that $(\r_1, \Ff_1) \sim (\r_2, \Ff_2)$ if there is an ordered field homomorphism \func{\varphi}{\Ff_1}{\Ff_2} for which $\r_2 = \varphi \circ \r_1$.
    \end{unDef}
    The real spectrum is endowed with a natural topology called the \emph{spectral topology}, see Subsection \ref{Subsection The real spectrum of a semialgebraic set}. It offers a real counterpart to the spectrum in algebraic geometry and serves as a central object of study in real algebraic geometry \cite{BCRrea}.
    The real spectrum defines a functor from the category of real algebraic sets to the category of compact topological spaces, which associates to every algebraic map \func{\pi}{V(\Kk)}{W(\Kk)} between algebraic sets a continuous map 
    \[
        \func{\spec{\pi}}{\spec{V(\Kk)}}{\spec{W(\Kk)}}.
    \]
    With its spectral topology \spec{V(\Kk)} is not Hausdorff in general. However, the subspace of closed points $\rsp{V(\Kk)}\subset \spec{V(\Kk)}$ is a compact Hausdorff space, which we refer to as the \emph{real spectrum compactification} of $V(\Kk)$.
    Note that given an algebraic map \func{\pi}{V(\Kk)}{W(\Kk)} between algebraic sets $V(\Kk)$ and $W(\Kk)$, the induced continuous map $\spec{\pi}$ does not necessarily map closed points to closed points.
    \begin{unThm}[{Theorem \ref{Thm: proper algebraic map and closed points}}]
        Let $V\subset \Rr^n$ be an algebraic set and $W \subset \Rr^m$ a closed algebraic set. If $\func{\pi}{V}{ W}$ is a proper algebraic map, then \speccl{\pi}, the restriction to the closed points of the induced map $\spec{\pi}$, takes its values in $\rsp{W}$. That is, we have a continuous map 
        \[
            \speccl{\pi}\colon \rsp{V}\rightarrow \rsp{W}.
        \]
    \end{unThm}
    Brumfiel deduces from \cite{CSvar} that $\Xi(\G,\pslr)$ is a semialgebraic set, and Burger, Iozzi, Parreau, and Pozzetti extend this result to $\Xi(\G,\slnr)$ via Richardson–Slowdowy theory \cite{RSmin}, as described in \cite{Bthe, BIPPthereal}. 
    These papers examine real spectrum compactifications with a focus on their topology and their related objects at infinity, with the aim to better understand the degenerations of elements in the representation space.
    The work \cite{BIPPthereal} initiates a program dedicated at studying $\rsp{\mathrm{Hom}(\G,\slnr)}$ and $\rsp{\Xi(\G,\slnr)}$, with the goal of analyzing representations of finitely generated groups in higher rank reductive Lie groups. 
    In particular, they characterize elements of $\rsp{\mathrm{Hom}(\G,\slnr)}$ as equivalence classes of representations
    \[
        \func{\phi}{\G}{\slnf},
    \]
    for suitable real closed fields \Ff \ \cite[Proposition 6.3]{BIPPthereal}.
    Such a representation induces a \G -action by isometries on an affine building (a geometric space with a rich combinatorial structure \cite[Theorem 8.1]{Ageneralizedaffinebuildings}) that has no fixed point \cite[Proposition 6.4]{BIPPthereal}. This construction has implications for the study of Anosov, maximal, Hitchin, and $\Theta$-positive representations \cite[Examples 6.19]{BIPPthereal}.
    This method of studying boundary points of $\rsp{\mathrm{Hom}(\G,\slnr)}$ naturally leads to the examination of non-Archimedean geometry. For a non-Archimedean real closed field \Ff , the projective space \projspnf \ is totally disconnected, not locally compact, and not open in $\rsp{\projspnf}$. To develop a deeper understanding of this non-Archimedean geometry, we use the Archimedean spectrum.

\subsection{Non-Archimedean geometry and Archimedean spectrum}
    Let $\mathbb{L}$ be an ordered field. An ordered algebraic extension $\Ff$ is the \emph{real closure} of $\mathbb{L}$ if \Ff \ is a real closed field and the ordering on $\mathbb{L}$ extends to the ordering of~\Ff . The real closure, denoted by $\overline{\Ll}^r$, always exists and is unique up to order-preserving isomorphism over \Ll , see Subsection \ref{Subsection: preliminaries in real algebraic geometry}.
    In addition, if $R_1 \subset R_2$ are two subrings of an ordered field, then the subring $R_2$ is \emph{Archimedean} over $R_1$ if every element of $R_2$ is bounded above by some element of $R_1$. 
    \begin{unDef}
        Let $\Ll \subset \Kk$ be real closed fields and $V\subset \Ll^n$ an algebraic set. The \emph{Archimedean spectrum} of $V(\Kk)$ is the following subspace of \spec{V(\Kk)}
        \[
            \specarch{V(\Kk)}:=\setrelfrac{(\r,\Ff) \in \spec{V(\Kk)}}
            {\begin{tabular}{@{}c@{}}
             $\Ff =  \overline{\mathrm{Frac}(\r(\Kk[V])}^r$, \\
                \Ff \text{ is Archimedean over \Kk}
                \end{tabular}},
        \] 
    where $\mathrm{Frac}(\r(\Kk[V]))$ is the fraction field of $\r(\Kk[V])$.
    \end{unDef}
    The Archimedean spectrum defines a functor from the category of algebraic sets to the category of topological spaces, see Subsection \ref{Subsection: Archimedean spectrum}.
    An element $b$ in a real closed field \Kk \ is a \emph{big element} if for every $a\in \Kk$ there exists $n\in \Nn$ with $a \leq b^n$.
    \begin{unThm}[{Theorem \ref{Thm: The Archimedean spectrum is a countable union of open compact sets}}]
        Let \Ll \ be a real closed field, $\Ll \subset \Kk$ a real closed field with a big element $b$, and $V\subset \Ll^n$ an algebraic set. The Archimedean spectrum of $V(\Kk)$ is an open subset of \rsp{V(\Kk)} which is a countable union of compact subsets of \rsp{V(\Kk)}. In particular, the space \specarch{V(\Kk)} is \s-compact and locally compact. 
    \end{unThm}
    
    The Archimedean spectrum thus provides a well-behaved topological space for studying group actions on non-Archimedean algebraic sets. These favorable properties are further supported by its identification with real analytifications, as discussed in \cite{JSYrea}.    

\subsection{Universal geometric space over the real spectrum compactification}
    To better understand the advantages of universal geometric spaces over $\rsp{\mathcal{M}_\G(\Rr)}$, consider the two following constructions (where $\widehat{\symr}$ is an algebraic cover of the symmetric space $\symr$ associated to \slnr).
    First, we define a topological space that contains $\mathcal{M}_\G(\Rr)\times \widehat{\symr}$: 
    \[
        E_1 \subset \rsp{\left(\mathcal{M}_\G(\Rr)\times \widehat{\symr}\right)}
    \]
    which is a countable union of specifically chosen open subsets, see Subsection \ref{Subsection: Displacement of an action and its associated geometric space}. These open sets are constructed to yield a \s-compact space equipped with a natural \G-action by homeomorphisms, induced by the \slnr \ action on $\widehat{\symr}$ by algebraic isomorphisms. 
    Second, define 
    \[
        E_2:=\rsp{\left(\mathcal{M}_\G(\Rr) \times \projspnr\right)}.
    \]
    The algebraic group \slnr \ acts on \projspnr \ by algebraic isomorphisms. 
    So, \G \ acts on the universal projective space $\mathcal{M}_\G(\Rr) \times \projspnr$ over $\mathcal{M}_\G(\Rr)$ via
    \[
       \g.\left(\p,x\right)=\left(\p, \p(\g)x \right) \quad \forall \g \in \G. 
    \]
    Both spaces share important structural properties. For either construction, we consider the natural projection maps
    \begin{align*}
        \func{\pi_1}{\mathcal{M}_\G(\Rr) \times \widehat{\symr}}{\mathcal{M}_\G(\Rr)}, \\
        \func{\pi_2}{\mathcal{M}_\G(\Rr) \times \projspnr}{\mathcal{M}_\G(\Rr)},
    \end{align*}
    which are algebraic, surjective, \G -equivariant for the trivial action on $\mathcal{M}_\G(\Rr)$, and open. 
    For each element $\p\in \mathcal{M}_\G(\Rr)$, \G \ acts naturally via \p \ on the fiber $\pi_i^{-1}(\p)$ which is homeomorphic to either $\widehat{\symr}$ for $i=1$ or \projspnr \ for $i=2$, see Subsection \ref{Subsection: The universal projective space}.
    To capture the degeneration of \G -actions on these geometric spaces induced by representations in $\mathcal{M}_\G(\Rr)$, we extend these universal geometric spaces over $\rsp{\mathcal{M}_\G(\Rr)}$. Our candidate for such an extension is $E_i$.
    However 
    \[
        \rsp{(\mathcal{M}_\G(\Rr) \times \projspnr)} \not\cong \rsp{\mathcal{M}_\G(\Rr)} \times \rsp{\projspnr}
    \]
    in general (and similarly when replacing \projspnr \ by $\widehat{\symr}$), making the study of the fibers of \spec{\pi_i} less straightforward.
    We refine the study of the real spectrum compactification of the product of algebraic sets. 
    We use the functorial properties of the real spectrum and topological properties of $E_i$ (for example, the compactness of \projspnr) to deduce that $\spec{\pi_i}|_{E_i}$ takes values in \rsp{\mathcal{M}_\G(\Rr)}. It leads to the central result of this thesis: 
    a characterization of $\partial\rsp{\mathcal{M}_\G(\Rr)}$ as \G-actions on the Archimedean spectrum, which in turn are interpreted as actions on real analytifications.
    \begin{unThm}[Theorem \ref{Thm: universal projective space over the real spectrum} and Theorem \ref{Thm: universal symmetric space over the real spectrum}]
        For $i=1,2$, the projection map $\pi_i$ induces a continuous map
        \[
            \spec{\pi_i}|_{E_i}: E_i \rightarrow \rsp{\mathcal{M}_\G(\Rr)},
        \]
        which is surjective, \G -equivariant, with the fiber over $(\r,\Ff)\in \rsp{\mathcal{M}_\G(\Rr)}$ homeomorphic to $\specarch{\widehat{\symf}}$ for $i=1$ and $\specarch{\projspnf}$ for $i=2$.
    \end{unThm} 
    To better understand the arrangement of the fibers of $\spec{\pi_i}|_{E_i}$ over the space $\rsp{\mathcal{M}_\G(\Rr)}$, a finer analysis is required. 
    By \cite{CRlatopologieduspectrereel}, \spec{\pi_i} is an open map, and fibers over $(\r,\Ff)\in\spec{\mathcal{M}_\G(\Rr)}$ are homeomorphic to $\spec{\widehat{\symr}}$ for $i=1$ and $\spec{\projspnf}$ for $i=2$.
    In our setting, we refine this description.
    \begin{unThm}[{Theorem \ref{Thm: the projection map is open}}]
        If $\pi_i$ is the projection map defined above, then
        \[
            \func{\spec{\pi_i}|_{E_i}}{E_i}{\rsp{\mathcal{M}_\G(\Rr)}}
        \]
        is open.
    \end{unThm}

    We deduce structural behavior of the fibers in $E_i$, which strengthens the accessibility results presented in \cite[Subsection 3.2]{BIPPthereal} for our specific context.
    
    \begin{unCor}[{Corollary \ref{Cor: convergence of the fibers in the universal projective space} and Corollary \ref{Cor: convergence of the fibers in the universal projective space}}]
        Let $(\r_n,\Kk_n) \subset \rsp{\mathcal{M}_\G(\Rr)}$ be a sequence converging to $(\r,\Kk) \in \rsp{\mathcal{M}_\G(\Rr)}$. For any $(\r,A,\Ff) \in (\speccl{\pi})^{-1}(\r,\Kk)$, \tes \ a sequence $(\r_n,A_n,\Ff_n)$ in the preimage $ (\speccl{\pi})^{-1}(\r_n,\Kk_n)$ such that $(\r_n,A_n,\Kk_n)$ converges to $(\r,A,\Ff)$ in the spectral topology.
    \end{unCor}

    It is therefore useful to gain a better understanding of the \G-actions on the Archimedean spectrum. To this end, we use an identification of the Archimedean spectrum with analytic spaces.

\subsection{Archimedean spectrum and Berkovich analytification}
    Let \Kk \ be a non-Archimedean field \Kk \ with an absolute value $|\cdot|$. The \Kk -points of an algebraic set $V$ are then totally disconnected, making the study of its analytical properties nonstandard. To address this, Berkovich introduces the \emph{Berkovich analytification} \ana{V(\Kk)} in \cite{Bspe}, which contains $V(\Kk)$ and has a richer topological structure.
    
    A \emph{multiplicative seminorm} on $V(\Kk)$ is a map \func{\eta}{\Kk[V]}{\Rr_{\geq 0}} \st 
    \begin{enumerate}
        \item $\eta(f)=|f|_{\Kk}$ \fev \ $f\in \Kk$,
        \item  $\eta(fg)=\eta(f)\eta(g)$ \fev \ $f,g\in \Kk[V]$, 
        \item $\eta(f+g)\leq \max{\set{\eta(f),\eta(g)}}$ \fev \ $f,g\in \Kk[V]$. 
    \end{enumerate} 
    \begin{unDef}[{\cite[Definition 1.5.1]{Bspe}}]
        The \emph{Berkovich} \emph{analytification} of $V(\Kk)$ is the set
        \[
            \ana{V(\Kk)}:= \setrel{\func{\eta}{\Kk[V]}{\Rr_{\geq 0}}}{\eta \text{ is a multiplicative seminorm}},
        \]
            with the coarsest topology that makes the evaluation maps on elements of $\Kk[V]$ continuous.
    \end{unDef}
    This space has good topological properties. For example, $\ana{V(\Kk)}$ is locally compact, locally contractible, and contains $V(\Kk)$ as an open and dense subset, see Subsection \ref{subsection: Real analytification}. 
    Moreover, $V(\Kk)$ is connected in the Zariski topology if and only if \ana{V(\Kk)} is connected. This makes the Berkovich analytification a powerful tool for studying non-Archimedean geometry using analytic techniques \cite{BRpot}.
    In this work, we investigate the relationship between the Archimedean spectrum and the Berkovich analytification. 
    Specifically, using the good topological properties of the Berkovich analytification and its canonical measure \cite[Chapter 10]{BRpot}, we aim to study group actions on the Archimedean spectrum using tools from dynamical systems.
    \begin{unThm}[{\cite[Theorem 3.17 and Lemma 3.9]{JSYrea}}]
        Let \Kk \ be a real closed field with a non-trivial absolute value and $V(\Kk)$ an algebraic set. There exists a canonical map from~\specarch{V(\Kk)} to $\ana{V(\Kk)}$ which is continuous and proper. 
    \end{unThm}
    To establish this result, the authors of \cite{JSYrea} introduce the real analytification, a real counterpart of Berkovich analytification, and show that it is homeomorphic to the Archimedean spectrum.
    We give an explicit description of the image of the Archimedean spectrum inside the Berkovich analytification:
    \begin{unThm}[{Theorem \ref{Thm: image of the archimedean spectrum in the anlaytification}}]
        Let $\Ll \subset \Kk$ be real closed fields with non-trivial absolute values and $V \subset \Ll^n$ an algebraic set.
        The image of $\specarch{V(\Kk)}$ in $\ana{V(\Kk)}$ is 
        \[
            \setrelfrac{\eta \in \ana{V(\Kk)}}{ \eta \left(f^2_1 + \cdots + f^2_q\right) = \max_i \eta\left(f^2_i\right) \quad \forall f_1,\ldots, f_q\in \Kk[V]}.
        \]
    \end{unThm}
    In the case $V(\Kk)=\projspk$, the space \ana{\projspk} is uniquely path connected \cite{BRpot}. 

    \begin{unThm}[{Corollary \ref{Cor: image of projective space is uniquely path connected} and \ref{Cor: image of the projective space is a real tree}}]
        Let \Kk \ be a non-Archimedean real closed field with a non-trivial absolute value.
        The image of $\specarch{\projspk}$ in ${\ana{\projspk}}$ is uniquely path connected and the image of $\specarch{\projspk}\backslash \projspk$ in $\ana{\projspk} \backslash \projspk$ is a closed real subtree of $\ana{\projspk}$ on which \pslk \ acts by isometries. 
    \end{unThm}
    \begin{unRem}
        It would be interesting to extend this result to higher rank settings. One indication that this is possible is that \ana{\projspnk} admits a skeleton $X$, which is canonically identified with the Bruhat--Tits building of \pslnk . Moreover, there exists a strong retraction from \ana{\projspnk} onto $X$ \cite{Tgeometrietoroidaleetgeometrieanalytique}. A natural question is whether the image of  \specarch{\projspnk} under this retraction is a convex subset of the affine building $X$.
    \end{unRem}

\subsection{Structure of the paper}

    Section \ref{Section algebraic model} introduces the necessary background in real algebraic geometry and establishes the relevant notation. Using Richardson--Slodowy theory, we describe $\mathcal{M}_\G(\Rr)\subset \mathrm{Hom}_{\mathrm{red}}(\G,\slnr)$ the space of minimal vectors as an algebraic set that covers the character variety. Finally, we construct $\widehat{\mathcal{P}^1(n,\Rr)}$, an algebraic cover of $\mathcal{P}^1(n,\Rr)$, which we endow with a semialgebraic multiplicative pseudo-metric.

    Section \ref{Section: Real spectrum and Archimedean spectrum} reviews some theory of the real spectrum of commutative rings with unity and of algebraic sets. We study its closed points which define the real spectrum compactification and prove functorial properties. Then, we examine the Archimedean spectrum of algebraic sets and prove that it is an open dense subset of the real spectrum compactification which is locally compact and \s -compact. 
    Finally, we compute the Archimedean spectrum of the \Kk -points of any algebraic set for \Kk \ Archimedean, and of the affine line over any real closed field.

    Section \ref{Section relation between the real spectrum of minimal vectors and the Archimedean spectrum of their associated symmetric space} constructs universal geometric spaces over $\rsp{\mathcal{M}_\G(\Rr)}$ and gives properties of the projection map onto $\rsp{\mathcal{M}_\G(\Rr)}$. The fibers of the projection map at the level of the real spectrum are well organized over $\rsp{\mathcal{M}_\G(\Rr)}$ and coincide with the Archimedean spectrum, yielding new accessibility results in the real spectrum. Finally, we construct a suitable subset of \rsp{(\mathcal{M}_\G(\Rr) \times \widehat{\symr})} which contains $\mathcal{M}_\G(\Rr) \times \widehat{\symr}$ and serves as a universal symmetric space over $\rsp{\mathcal{M}_\G(\Rr)}$. 

    Section \ref{Section: The Archimedean spectrum and the real analytification of an algebraic set} reviews the theory of Berkovich and real analytification and connects them to the theory of the Archimedean spectrum. We use the canonical homeomorphism between the Archimedean spectrum and the real analytification to construct a continuous and proper map from the Archimedean spectrum to the Berkovich analytification. We characterize its image completely and examine the specific case of the Archimedean spectrum of the projective line. In this setting, the image in the Berkovich analytification is a $\mathrm{PSL}_2(\Ff)$-invariant real subtree.

%% file: Algebraic_sets.tex
\subsection{Preliminaries in real algebraic geometry} \label{Subsection: preliminaries in real algebraic geometry}
Our approach to the theory of character varieties is based on real algebraic geometry. In particular, our main objects of study are algebraic sets defined over real closed fields. To set up this framework, we first introduce key definitions and notation, following \cite[Section 1 and 2]{BCRrea} and \cite[Section 2]{BIPPthereal}.

\begin{Def}
    A field $\mathbb{K}$ is \emph{ordered} if there exists a total order $\leq$ compatible with its field operations. Formally, $\leq$ satisfies: for every $a, b, c \in \mathbb{K}$
    \[
        \text{if } a \leq b, \text{ then } a + c \leq b + c \text{ and if } 0 \leq a, b \text{ then } 0 \leq ab.
    \]
\end{Def}
A \emph{real field} is a field that can be ordered.
A stronger notion is that of \emph{real closed field} $\mathbb{K}$, which is a real field that has no proper algebraic ordered field extension. Equivalently, this means that every positive element has a square root, and every polynomial of odd degree has a root in $\mathbb{K}$ \cite[Theorem 1.2.2]{BCRrea}. 
The following theorem sheds a little more light on the nature of real closed fields. It gives an equivalent definition, which is highlighted by the Transfer principle ---a result of the Tarski--Seidenberg principle. 

\begin{Thm}[Transfer principle {\cite[Proposition 5.2.3]{BCRrea}}] \label{Thm transfer principle}
    Let $\Kk$ be a real closed field, $\Psi$ a formula in the first-order language of ordered rings with parameters in~$\Kk$ without a free variable, and~$\Ff$ a real closed extension of~$\Kk$. Then $\Psi$ holds true in~$\Kk$ \iff \ it holds true in \Ff .
\end{Thm}

Real closed fields possess the same elementary theory as the reals. They give a natural framework to extend results from \Rr \ to more general fields which are real~closed. 

\begin{Ex} \label{Example real closed fields}
    The field $\overline{\Qq}^r$ of real algebraic numbers and the field \Rr \ of real numbers are real closed fields \cite[Example 1.3.6]{BCRrea}. Another example of a real field is given by the \emph{real Puiseux series}
    \[
        \setrelfrac{\sum_{k=-\infty}^{k_0} c_k x^{\frac{k}{m}}}{k_0, m \in \mathbb{Z}, \, m > 0 , \, c_k \in \Rr, \, c_{k_0}\neq 0},
    \]
    which, if it is endowed with the order such that $\sum_{k=-\infty}^{k_0} c_k x^{\frac{k}{m}} > 0$ if $c_{k_0} > 0$, is real closed.
\end{Ex}

Let $\mathbb{L}$ be an ordered field. An ordered algebraic extension $\Ff$ is the \emph{real closure} of $\mathbb{L}$ if \Ff \ is a real closed field and the ordering on $\mathbb{L}$ extends to the ordering of~\Ff . The real closure of $\mathbb{L}$ always exists and is unique up to a unique order preserving isomorphism over \Ll. That is, if \( \Ff_1 \) and \( \Ff_2 \) are real closures of \Ll, there exists a unique order preserving isomorphism \( \Ff_1 \to \Ff_2 \) that is the identity on~\Ll. Denote the real closure of \Ll \ by:
\[
    {\overline{\mathbb{L}}}^r.
\]

In real algebraic geometry, the fundamental objects of study are real algebraic and semialgebraic sets. Intuitively, algebraic sets consist of points satisfying polynomial equations, while semialgebraic sets allow polynomial inequalities as well.
\begin{Def}
        Let $\mathbb{K}$ be a real closed field. A set $V\subset \Kk^n$ is an \emph{algebraic set} defined over \Kk \ if there exists $B \subset \mathbb{K}[x_1, \dots, x_n]$ such that
        \[
            V := \left\{ v \in \Kk^n \mid f(v) = 0 \quad \forall f\in B \right\}.
        \]
        Then, $S\subset \Kk^n$ is a \emph{semialgebraic set} defined over \Kk \ if there exists polynomials $f_i,g_j \in \mathbb{K}[x_1, \dots, x_n]$ such that
        \[
            S := \bigcup_{finite}\bigcap_{finite} \left\{ s \in \Kk^n \mid f_i(s) = 0 \right\} \cap \left\{ s \in \Kk^n \mid g_j(s) > 0 \right\}.
        \]
    Denote by $I(S):= \setrel{f\in \mathbb{K}[x_1,\dots,x_n]}{f(s)=0\ \forall s\in S}$ the ideal of polynomials vanishing on $S$.
\end{Def}

For the remainder of this section, let $\mathbb{K}$ be a real closed field and $\mathbb{F}$ a real closed extension of $\mathbb{K}$.

\begin{Def}\label{Def: coordinate ring}
    Let $S\subset \Kk^n$ be a semialgebraic set.
    The \emph{coordinate ring} of the \Ff-extension of $S$ is
    \[
        \Ff[S]:=\Ff[x_1,\ldots,x_n]/I(S).
    \]
\end{Def}

Consider a semialgebraic set $S$ defined over $\mathbb{K}$. The \emph{$\mathbb{F}$-points} of $S$ are the solutions in $\mathbb{F}^n$ of the polynomials defining $S$
\[
    S(\Ff):=\left\{ s \in \Ff^n \mid f(v) = 0 \quad \forall s\in B \right\},
\]
where $B \subset \Kk[x_1,\ldots,x_n]$ is a set of polynomials defining $S$. It can be shown that the set $S(\Ff)$ is independent of the choice of $B$ \cite[Proposition 5.1.1]{BCRrea}.
By abuse of notation, we use the same symbol for the \Ff-points and the $\Ff$-extension of $S$.
To define a topology on $S(\mathbb{F})$, introduce the norm $N: \mathbb{F}^n \to \mathbb{F}_{\geq 0}$, given by
\[
    N(s) := \sqrt{\sum_{i=1}^n s_i^2} \quad \forall s=(s_1,\ldots , s_n)\in \Ff^n.
\]
The \emph{open ball centered at $s \in \Ff^n$ with radius $r\in \Ff_{\geq 0}$} is then defined as
\[
    B(s, r) := \setrel{ y \in \Ff^n}{N(s - y) < r}.
\]
These open balls form a basis for the \emph{Euclidean topology} on $\mathbb{F}^n$.
Finally, to study the links between algebraic sets we will use maps between semialgebraic sets with good algebraic properties.
\begin{Def}
    A map $\func{\pi}{S_1}{S_2}$ between semialgebraic sets $S_1 \subset \Kk^n$, $S_2 \subset \Kk^m$ is \emph{semialgebraic} if its graph is a semialgebraic subset of \( \Kk^n \times \Kk^m \).
\end{Def}

\begin{Prop}[{\cite[Proposition 5.3.1]{BCRrea}}]
    Let $S_1 \subset \Kk^n$, $S_2 \subset \Kk^m$ be semialgebraic sets and $\func{\pi}{S_1}{S_2}$ a semialgebraic map with graph $X$. The \Ff-extension $X(\Ff)$ is the graph of a semialgebraic map $\pi_\Ff$ called the \emph{$\Ff$-extension} of $\pi$.
\end{Prop}

We defined our primary objects of study: semialgebraic sets and semialgebraic maps over real closed fields. In Section \ref{Section relation between the real spectrum of minimal vectors and the Archimedean spectrum of their associated symmetric space}, we employ semialgebraic maps to construct universal geometric spaces over character varieties, thereby providing a geometric description of the real spectrum compactification of character varieties.
In the next subsection, we construct some algebraic and semialgebraic examples that are of particular interest to us.

%% file: Character_varieties_and_minimal_vectors.tex
\subsection{Character varieties and minimal vectors $\mathcal{M}_\G$} \label{Subsection minimal vectors and character varities}

        To better understand character varieties, we study the theory of minimal vectors of representations of a finitely generated group~\G \ in \slnr. 
        Minimal vectors define an algebraic set, which in turn allows for the definition of its real spectrum compactification in Subsection \ref{Subsection The real spectrum of a semialgebraic set}.
        Moreover, character varieties are a quotient by a compact algebraic group of minimal vectors providing a semialgebraic model for character varieties.
        This introduction to minimal vectors is based on~\cite{RSmin} (on ideas of Kempf--Ness for the complex reductive case). See also~\cite{BLrea} and \cite[Section 6, Section 7]{BIPPthereal} for more general treatments. 
        
        Let \G \ be a finitely generated group with finite generating set $F$ with $s$ elements. The algebraic group \slnr \ acts by conjugation on the real vector space $M_{n \times n}(\Rr)^{F}$, which is endowed with the $\mathrm{SO}_n(\Rr)$-invariant scalar product 
        \[
            \psc{(A_1,\ldots,A_s)}{(B_1,\ldots,B_s)} := \sum_{i=1}^{s}\mathrm{tr}\left(A_i^{T} B_i\right).
        \]
        Denote by $\norm{\cdot}$ its associated norm. The set of \emph{minimal vectors} of $M_{n\times n}(\Rr)^F$ for the \slnr -action by conjugation is
        \[
            \mathcal{M}(\Rr):=\setrelb{v\in M_{n\times n}(\Rr)^F}{\norm{g.v}\geq \norm{v} \ \fev \ g\in \mathrm{SL}_{n}(\Rr)}. 
        \]
        This defines a closed subset of $M_{n\times n}(\Rr)^F$, and we show that it defines an algebraic set. Consider the involution $\sigma: g \mapsto (g^T)^{-1}$ on \slnr \ which sends~$g$ to the inverse of its transpose. 
        Then, $\mathrm{SO}_n(\Rr)$ is the subgroup of fixed points of~$\sigma$ and 
        \[
            sym^0_n(\Rr) := \setrelfrac{A\in M_{n\times n}(\Rr)}{\mathrm{tr}(A)=0, A=A^T}
        \]
        is the $-1$ eigenspace of the Cartan involution $d_{\mathrm{Id}}\sigma$ on the Lie algebra of \slnr. 
        This eigenspace also acts on $M_{n\times n}(\Rr)^F$ by
        \begin{align*}
            Z. (A_1,\ldots,A_s) &= \frac{d}{dt}_{|_{t=0}}\exp{(tZ)}(A_1,\ldots,A_s)\exp{(-tZ)} \\
            & =\left( [Z,A_1], \ldots , [Z,A_s]\right).
        \end{align*}
        In this context, the following result from the Richardson--Slowdowy theory of minimal vectors holds.
        \begin{Thm}[{\cite[Theorem 4.3]{RSmin}}]
            An element $v\in M_{n\times n}(\Rr)^F$ is in $\mathcal{M}(\Rr)$ if and only if $\psc{Z.v}{v}= 0$ \fev \ $Z \in sym^0_n(\Rr)$.
        \end{Thm}
        Using the description of the action of $sym^0_n(\Rr)$ on $M_{n\times n}(\Rr)^F$ and, in the last equality, that $[A,A^T]\in sym_n^0(\Rr)$ for all $A \in M_{n \times n}(\Rr)$, we have
        \begin{align*}
            \mathcal{M}(\Rr) & = \setrelfrac{(A_1,\ldots,A_s)\in M_{n\times n}(\Rr)^F}{\mathrm{tr}\left( \sum_{i=1}^s [Z,A_i]^TA_i \right) = 0 \ \forall  Z\in sym^0_n(\Rr)} \\
            & = \setrelfrac{(A_1,\ldots,A_s)\in M_{n\times n}(\Rr)^F}{\mathrm{tr}\left( \sum_{i=1}^s \left[A_i,A_i^T\right]Z \right) = 0 \ \forall Z\in sym^0_n(\Rr)} \\
            & = \setrelfrac{(A_1,\ldots,A_s)\in M_{n\times n}(\Rr)^F}{\sum_{i=1}^s\left[A_i,A_i^T\right]=0}.
        \end{align*}
        In particular, $\mathcal{M}$ is an algebraic set defined over $\overline{\Qq}^r$. 
        
        Next, we relate minimal vectors to representations of \G \ in \slnr. 
        We study the representation space within the real vector space $M_{n \times n}(\Rr)^{F}$ using
        \[
            \defmap{ev}{\mathrm{Hom}(\G,\mathrm{SL}_{n}(\Rr))}{\mathrm{SL}_{n}(\Rr)^{F} \subset M_{n \times n}(\Rr)^{F}}{\r}{(\r(\g))_{\g\in F}}
        \]
        the evaluation of morphisms on the generating set $F$. The evaluation map is injective and its image $R^{F}(\G,\mathrm{SL}_{n}(\Rr))$ is a closed algebraic subset of $M_{n \times n}(\Rr)^F$. 
        The group \slnr \ acts by conjugation on both $\mathrm{Hom}(\G,\mathrm{SL}_{n}(\Rr))$ and $M_{n \times n}(\Rr)^F$. Moreover, the evaluation map is \slnr -equivariant with respect to these actions and so its image is \slnr -invariant.
        From \cite[Theorem 30]{Scharactervarieties} (following an argument in \cite{JMdeformationspaces}), the restriction to \emph{reductive homomorphisms}, that is a direct sum of irreducible representations, has image 
        \begin{equation}\label{Equation: reductive representation are the closed orbits}
            R^{F}_{\mathrm{red}}(\G,\mathrm{SL}_{n}(\Rr))=\setrelb{v\in R^{F}(\G,\mathrm{SL}_{n}(\Rr))}{\mathrm{SL}_{n}(\Rr).v \text{ is closed}}.
        \end{equation}
        
        The link with minimal vectors comes from the following result.
    
        \begin{Thm}[{\cite[Theorem 4.4]{RSmin}}]
            If $v\in M_{n \times n}(\Rr)^{F}$, then the intersection $\mathrm{SL}_{n}(\Rr) . v \cap \mathcal{M}(\Rr)$ is not empty if and only if $ \mathrm{SL}_{n}(\Rr) . v $ is closed.
        \end{Thm}

        Hence, the set
        \[
            \mathcal{M}_\G(\Rr):=\mathcal{M}(\Rr)\cap R_{\mathrm{red}}^{F}(\G,\mathrm{SL}_{n}(\Rr))
        \]
        satisfies $\mathcal{M}(\Rr)\cap R_{\mathrm{red}}^{F}(\G,\mathrm{SL}_{n}(\Rr))=\mathcal{M}(\Rr)\cap R^{F}(\G,\mathrm{SL}_{n}(\Rr))$. Thus, $\mathcal{M}_\G(\Rr)$ is a closed algebraic subset of $M_{n \times n}(\Rr)^{F}$.
        Using results from the theory of quotients by compact Lie groups \cite[Section 7]{RSmin} (using \cite{Ssmoothfunctions}), the quotient $\mathcal{M}_\G(\Rr)/\mathrm{SO}_n(\Rr)$ is homeomorphic to a closed semialgebraic set. 
    
        \begin{Thm}[{\cite[Theorem 7.7]{RSmin}}]
            The inclusion $ \mathcal{M}_{\G}(\Rr) \subseteq R_{\mathrm{red}}^{F}(\G,\mathrm{SL}_{n}(\Rr))$ induces a homeomorphism between $\mathcal{M}_\G(\Rr)/\mathrm{SO}_n(\Rr)$ and the topological quotient $R_{\mathrm{red}}^{F}(\G,\mathrm{SL}_{n}(\Rr))/\mathrm{SL}_{n}(\Rr)$.
        \end{Thm}

        The \emph{character variety} $\Xi(\G,\slnr)$ of the finitely generated group \G \ and the algebraic group $\slnr$ is the quotient of reductive homomorphisms $\mathrm{Hom}_{\mathrm{red}}(\G, \slnr)$ via postconjugation by $\slnr$ endowed with the compact open topology.
        Thus, the previous theorem and Equation \ref{Equation: reductive representation are the closed orbits} above prove the following result.

        \begin{Thm}[{\cite[Section 7.1]{RSmin}}]
            The character variety of a finitely generated group \G \ in \slnr \ is a semialgebraic set which is homeomorphic to $\mathcal{M}_\G(\Rr)/\mathrm{SO}_n(\Rr)$.
        \end{Thm}
        
        \begin{Rem}
            In this text we restrict ourselves to the study of minimal vectors of \G \ in \slnr \ that we study as real algebraic sets. For a better understanding of how to use the continuous surjection from minimal vectors to the character variety, we refer to subsection 7.2 of \cite{BIPPthereal} and the subsections that follow.
        \end{Rem}

        It is worth noting that the space of minimal vectors of \G \ in \slnr \ remains an algebraic set in the more general case when~\slnr \ is replaced by $G(\Rr)$ a connected semisimple algebraic group defined over the reals, see \cite[Subsection 7.2]{BIPPthereal}.           
        One aspect of our study is to define the real spectrum compactification of~$\mathcal{M}_\G$, which relies on $\mathcal{M}_\G$ to be an algebraic set. Before defining this compactification, we first analyze the symmetric space associated to~\slnr \ which provides tools to study the geometric and dynamical properties of $\mathcal{M}_\G$.

%% file: Symmetric_space_associated_to_SLNR.tex
\subsection{The symmetric space associated to \slnr}\label{Subsection: Symmetric space associated to slnr}
        This section recalls the basic theory of the symmetric space associated to \slnr \ and introduces a cover of it, which is an algebraic set. We review the construction of a continuous semialgebraic multiplicative norm on both spaces using the Cartan projection defined on \slnr . This allows us to further study the space of minimal vectors and their dynamics in Section \ref{Section relation between the real spectrum of minimal vectors and the Archimedean spectrum of their associated symmetric space}. This introduction is inspired by \cite[Section 5]{BIPPthereal}. 
        
        The symmetric space associated to \slnr \ is
        \[
            \symr := \setrel{A\in M_{n\times n}(\Rr)}{\mathrm{det}(A)=1, A \text{ is symmetric and positive definite}}.
        \]
        From this definition, \sym \ is a semialgebraic set and \slnr \ acts transitively on \symr \ via
        \[
            g A := g Ag^T
        \]
        for all $g \in \slnr$ and all $A \in \symr$. Moreover, by Sylvester's criterion (see for example \cite[Theorem 7.2.5]{HJmatrix}), a symmetric matrix is positive definite if and only if all its principal minors are positive. So we define a cover of \symr \ as        
        \[
            \widehat{\mathcal{P}^1(n,\Rr)} := \setrelfrac{(A,t) \in M_{n\times n}(\Rr) \times \Rr^{n-1}}{
            \begin{tabular}{@{}c@{}}
                $A$ is symmetric, $\mathrm{det}(A)=1$, \\
                $\mathrm{det}\left(A\left[j\right]\right)t^2_j = 1$ for $1\leq j \leq n-1$
                \end{tabular}}.
        \]
        where $A[j]$ denotes the $j$-th leading principal minor of the matrix $A$. 
        The group \slnr \ acts on $\widehat{\mathcal{P}^1(n,\Rr)}$ via 
        \[
            g.\left(A,t_1,\ldots,t_{n-1}\right) := \left(g A g^T,t'_1,\ldots,t'_{n-1}\right), 
        \]
        \fev \ $g \in \slnr$ and $(A,t_1,\ldots,t_{n-1}) \in \widehat{\symr}$ where 
        \[
            t'_j:=\left(\frac{\mathrm{det}\left(A[j]\right)}{\mathrm{det}\left(gAg^T[j]\right)}\right)^{1/2}t_j
        \] 
        \fev \ $1\leq j \leq n-1$.
        \begin{Rem}
            Note that \symr \ is the quotient of $\widehat{\symr}$ by the action of $(\Zz/2\Zz)^{n-1}$ defined by 
            $(z_1,\ldots,z_{n-1})(A,t_1,\ldots,t_{n-1})=(A,\e_1 t_1,\ldots,\e_{n-1}t_{n-1})$, where $\e_i=1$ if $z_i=0$ or $\e_i=-1$ if $z_i=1$. Moreover, the quotient map is \slnr -equivariant for the actions described above.
        \end{Rem}
        
        We now use the Cartan decomposition of \slnr \ to construct a multiplicative distance on \symr , which we promote to $\widehat{\symr}$.
        As in the previous section, consider $\mathrm{SO}_n(\Rr)$ the maximal compact subgroup of \slnr \ associated with the involution $\sigma: g \mapsto (g^T)^{-1}$. Consider also $S(\Rr)$ the maximal split torus of diagonal matrices in \slnr \ and its \emph{closed multiplicative Weyl chamber}
        \[
            \overline{C^+(\Rr)}:= \setrelfrac{
                \begin{pmatrix}
                    \lambda_1 &  &  \\
                     & \ddots &  \\
                     &  & \lambda_n
                \end{pmatrix}\in \slnr}{\l_1 \geq \cdots \geq \l_n >0}.
        \]
        Then $\slnr = \mathrm{SO}_n(\Rr)\overline{C^+(\Rr)}\mathrm{SO}_n(\Rr)$, that is, each $g \in \slnr$ can be written as 
        \[
            g=k_1c(g)k_2, \text{ where } k_1,k_2\in \mathrm{SO}_n(\Rr), c(g)\in \overline{C^+(\Rr)}.
        \]
        Moreover, the element $c(g)$ in the decomposition is unique (see \cite[Chapter IX, Theorem 1.1]{Hdifferential} or \cite[Theorem 7.39]{Klie}). This decomposition is called the \emph{Cartan decomposition} of \slnr. In addition, the map $\slnr \rightarrow \overline{C^+(\Rr)}$ that sends $g$ to $c(g)$ is semialgebraically continuous \cite[Proposition 4.4]{BIPPthereal}.
        A consequence of this decomposition is the existence of a \emph{Cartan projection}, which is semialgebraically continuous by \cite[Corollary 5.1]{BIPPthereal}:

        \begin{Lem}\label{simultaneous diagonalisation} 
            \Fev \ $A, B \in \symr$ the \slnr -orbit of~$(A,B)$ intersects $\set{\mathrm{Id}} \times \overline{C^+(\Rr)}.\set{\mathrm{Id}}$ in exactly one point $(\mathrm{Id}, \d (A, B).\mathrm{Id})$.
            Moreover, the Cartan projection 
            \[
                \defmap{\d}{\symr \times \symr}{\overline{C^+(\Rr)}}{(A,B)}{\d (A, B)}
            \]
            is well-defined, invariant under the \slnr -action, and semialgebraically continuous.
        \end{Lem}

        To apply results from real algebraic geometry, we are interested in a semialgebraically continuous multiplicative distance. To do this, consider the semialgebraically continuous map $\func{N}{S(\Rr)}{\Rr_{> 0}}$ that sends $\mathrm{diag}(\l_1,\ldots,\l_n)$ to $\max_{i\neq j}\l_i \l_j^{-1}$. It is a \emph{multiplicative norm} $\func{N}{S(\Rr)}{\Rr_{> 0}}$. That is, a map which is invariant for the Weyl group action, and verifies:
        \begin{itemize}
            \item $N(gh)\leq N(g)N(h)$ \fev \ $g,h\in S(\Rr)$,
            \item $N(g)\geq 1$ \fev \ $g\in S(\Rr)$ with equality \iff \ $g=\mathrm{Id}$,
            \item $N(g^n)=N(g)^{|n|}$ \fev \ $g\in S(\Rr)$ and $n \in \Zz$.
        \end{itemize}
        We now introduce the Cartan's multiplicative distance which allows us, in Section \ref{Section relation between the real spectrum of minimal vectors and the Archimedean spectrum of their associated symmetric space}, to define a universal symmetric space over the real spectrum compactification of $\mathcal{M}_\G(\Rr)$.
    
        \begin{Prop}[{\cite[Proposition 5.5]{BIPPthereal}}]\label{Properties Cartan distance}
            The Cartan's multiplicative distance is defined as
            \[
                \begin{matrix}
                d_\d: &\symr \times \symr&\longrightarrow &\overline{C^+(\Rr)}&\longrightarrow&\Rr_{\geq 1};\\
                &(A,B) & \longmapsto & \d(A,B) & \longmapsto & N(\d (A,B)) = \frac{\l_1}{\l_n},
                \end{matrix}
            \]
            where $\d(A,B)=\mathrm{diag}(\l_1,\ldots,\l_n)$.
            It is \slnr -invariant, semialgebraically continuous and verifies
            \begin{enumerate}
                \item \label{cartan projection is submultiplicative} $d_\d(x, z) \leq d_\d(x, y)d_\d(y, z)$ \fev \ $x, y, z \in \symr$,
                \item \label{cartan projection is positive definite} $d_\d(x, y) = 1$ \iff \ $x = y$.
            \end{enumerate}
        \end{Prop}        

        The multiplicative distance defined on \symr \ extends to a multiplicative pseudo-distance on the cover~$\widehat{\symr}$ via 
        \[  
            \DefMap{\widehat{d_\d}}{\widehat{\symr} \times \widehat{\symr}}{\Rr_{\geq 1}}{((A,t_1,\ldots,t_{n-1}), (B,t_1',\ldots,t_{n-1}'))}{d_\d(A,B)}
        \] 
        We use the objects presented above in the following sections to introduce a subspace of the real spectrum compactification of $\mathcal{M}_\G(\Rr)\times \widehat{\symr}$, invariant under the action of \G . This subspace encodes the behavior of minimal representations and their induced \G -actions on their associated symmetric space in the real spectrum. To this end, in the next section we introduce the real spectrum and the Archimedean spectrum of an algebraic set.

%% file: Real_spectrum_of_a_ring.tex
\subsection{The real and Archimedean spectrum of rings}\label{Subsection The real spectrum of a semialgebraic set}
        The real spectrum applies to commutative rings with unity and provides a natural functor from commutative rings to compact spaces. 
        In this subsection, we first present this compactification and the accompanying notions essential for its study in our text. Our presentation follows \cite[Chapter 7]{BCRrea}, \cite{Bthe, Btree}, and adapts \cite[Section 2]{BIPPthereal} to our context. Second, we define the concept of Archimedicity between rings, which allows us to characterize the closed points of the real spectrum and to introduce the Archimedean spectrum of a \Kk-algebra---both of which form notable subspaces of the real spectrum with interesting topological properties. Throughout this subsection, $A$ denotes a commutative ring with unity.
        \begin{Def}
            The \emph{real spectrum} $\spec{A}$ of a commutative ring $A$ with unity is the set of \emph{prime cones} of $A$. That is, the subsets $\a \subset A$ such that
            \begin{itemize}
                \item $-1 \notin \a$,
                \item $\a + \a \subset \a$ and $\a \cdot \a \subset \a$,
                \item $\a \cup (-\a) = A$,
                \item $\a \cap (-\a)$ is a prime ideal in $A$.
            \end{itemize}
        \end{Def} 
        Another characterization of the points of the real spectrum uses real algebraic geometry. In particular ring morphisms to real closed fields. 
        
        \begin{Prop}[{\cite[Proposition 7.1.2]{BCRrea}}]\label{Proposition: various defintion of the real spectrum}
            The following data are equivalent:
            \begin{enumerate}   
                \item \label{def: real spectrum as cones} a prime cone $\a \subset A$,
                \item \label{def: real spectrum as ideals} a pair $(p, \leq_p)$ consisting of a prime ideal $p$ and an ordering $\leq_p$ on the field of fractions of $A/p$,
                \item \label{Def: real spectrum as ring isomorphisms} an equivalence class of pairs $(\r, \Ff_\r)$ where \func{\r}{A}{\Ff_\r} is a ring homomorphism to a real closed field $\Ff_\r$ which is the real closure of the field of fractions of $\r(A)$
                and $(\r_1, \Ff_{\r_1})$, $(\r_2, \Ff_{\r_2})$ are equivalent if there exists an ordered field isomorphism \func{\varphi}{\Ff_{\r_1}}{\Ff_{\r_2}} such that $\r_2 = \varphi \circ \r_1$.
                \item \label{Def: real spectrum as ring homomorphisms} an equivalence class of pairs $(\r, \Ff)$ where \func{\r}{A}{\Ff} is a ring homomorphism to a real closed field \Ff, for the smallest equivalence relation such that $(\r_1, \Ff_1)$ and $(\r_2, \Ff_2)$ are equivalent if there is an ordered field homomorphism \func{\varphi}{\Ff_1}{\Ff_2} such that $\r_2 = \varphi \circ \r_1$.

            \end{enumerate}
        \end{Prop}
        One goes from \ref{def: real spectrum as cones} to \ref{def: real spectrum as ideals} by considering the prime ideal $p = \a \cap (-\a)$ and the unique order $\leq_\a$ on $\mathrm{Frac} (A/p)$ whose set of positive elements is given by
        \[
            \setrelfrac{\frac{\overline{a}}{\overline{b}}}{ab \in \a, b\notin p},
        \]
        where $\func{\overline{\cdot}}{A}{A/p}$ denotes the reduction modulo $p$ and~$\mathrm{Frac}(A/p)$ the fraction field of~$A/p$.
        One goes from \ref{def: real spectrum as ideals} to \ref{Def: real spectrum as ring isomorphisms} and \ref{Def: real spectrum as ring homomorphisms} by composing the reduction modulo $p$ with the inclusion of $\mathrm{Frac}(A/p)$ into its real closure \Ff \ with respect to the ordering $\leq_p$. 
        Finally, one goes from~\ref{Def: real spectrum as ring homomorphisms} to \ref{def: real spectrum as cones} by considering the prime cone 
        \[
            \a := \setrel{a\in A}{\r(a)\geq 0},
        \]
        where $(\r,\Ff)$ is the given ring homomorphism to a real closed field \Ff .
        
        Employing the notation from the fourth item of Proposition \ref{Proposition: various defintion of the real spectrum}, the \emph{spectral topology} on the real spectrum is defined using a basis of open sets 
        \[
            \tilde{U}(a_1,\ldots , a_p) := \setrel{(\r,\Ff) \in \spec{A}}{\r(a_k) > 0 \quad \forall k \in \set{1,\ldots, p}},
        \]
        where $a_k$ are elements in $A$ \fev \ $k\in\set{1,\ldots,p}$. With this topology, the real spectrum of a ring and its closed points have good properties.
        
        \begin{Thm}[{\cite[Proposition 7.1.25 (ii)]{BCRrea}}]\label{Thm: Functoriality of the real specturm}
            The real spectrum of $A$ is compact and its subset of closed points~\speccl{A} is Hausdorff and compact. 
        \end{Thm} 

        So, to each commutative ring with a unit, we have associated a compact topological space.
        Furthermore, if $\func{\pi}{A}{B}$ is a ring homomorphism between the commutative rings $A$ and $B$, then the \emph{lift} of $\pi$ to the real spectrum
        \[
            \defmap{\spec{\pi}}{\spec{B}}{\spec{A}}{(\r,\Ff)}{\left(\r \circ \pi, \Ff\right)}
        \]
        is a continuous map \cite[Proposition 7.1.7]{BCRrea}. Thus, the following result gives a first motivation to study the real spectrum of rings.     

        \begin{Thm}[{\cite[Proposition 7.1.7]{BCRrea}}]\label{Thm: functoriality of the real spectrum}
            $\mathrm{Spec}_\mathrm{R}$ is a contravariant functor from the category of commutative rings with unity to the category of compact topological spaces. 
        \end{Thm}   

        \begin{Ex}[{\cite[Example 2.24 (3)]{BIPPthereal}, \cite[Example 7.1.4]{BCRrea} for $\Kk =\Rr$}] \label{Ex: real spectrum of the line}

        Let $A = \Kk[x]$, where $\Kk\subset \Rr$ is real closed and $x$ is a variable. 
        \Fev \ $u \in \Rr$, set
        \begin{align*}
            \a_u &:= \setrelfrac{f \in \Kk[x]}{f(u) \geq 0}, \\
            \a_{u^+} &:= \setrelfrac{f \in \Kk[x]}{\exists \e > 0, \forall v \in ]u, u + \e[, \, f(v) \geq 0}, \\
            \a_{u^-} &:= \setrelfrac{f \in \Kk[x]}{\exists \e > 0, \forall v \in ]u - \e, u[, \, f(v) \geq 0},
        \end{align*}
        which are prime cones of $A$. They verify $\a_{u^\pm} \subset \a_u$, with equality if and only if $u \notin \Kk$.  
        The following prime cones complete the description of the real spectrum: 
        \begin{align*}
            \a_{+\infty} &:= \setrelfrac{f \in \Kk[x]}{\exists m \in \Kk, \forall v \in ]m, + \infty[, \, f(v) \geq 0}, \\
            \a_{-\infty} &:= \setrelfrac{f \in \Kk[x]}{\exists m \in \Kk, \forall v \in ]-\infty, m[, \, f(v) \geq 0}.
        \end{align*}
        
        By factoring polynomials, the topology of \spec{\Kk[x]} has a basis of open subsets consisting of the intervals 
        \begin{align*}
            \tilde{U}(x - s, -x + t) &= [s^+, t^-] = \{ \alpha_u \mid s < u < t \} \cup \{ \alpha_{u^-} \mid s < u \leq t \} \\
            &\phantom{= [s^+, t^-] = \{ \alpha_u \mid s < u < t \}} \cup \{ \alpha_{u^+} \mid s \leq u < t \}, \\
            \tilde{U}(x - s) &= [s^+, +\infty] = \{ \alpha_u \mid u > s \} \cup \{ \alpha_{u^-} \mid u > s \} \\
            &\phantom{= [s^+, +\infty] = \{ \alpha_u \mid u > s \}} \cup \{ \alpha_{u^+} \mid u \geq s \} \cup \{ \alpha_{+\infty} \}, \\
            \tilde{U}(-x + t) &= [-\infty, t^-] = \{ \alpha_u \mid u < t \} \cup \{ \alpha_{u^-} \mid u \leq t \} \\
            &\phantom{= [-\infty, t^-] = \{ \alpha_u \mid u < t \}} \cup \{ \alpha_{u^+} \mid u < t \} \cup \{ \alpha_{-\infty} \},
        \end{align*}
        where $s < t$ are two elements of \Rr.
        Note that \spec{\Kk[x]} is not a Hausdorff space since $\a_u$ belongs to the closure of both $\a_{u^+}$ and $\a_{u^-}$. 
        \end{Ex}

        To define the Archimedean spectrum, we need to recall the concept of Archimedicity in the context of general real fields.

        \begin{Def}[{\cite[Definition 2.26]{BIPPthereal}}] \label{Def: Archimedean ring}
            Let $R_1 \subset R_2$ be subrings of an ordered field. The subring $R_2$ is \emph{Archimedean} over $R_1$ if every element of $R_2$ is bounded above by some element of $R_1$.
        \end{Def}
                
        Before defining the Archimedean spectrum, we note that this definition allows us to characterize the closed points of the real spectrum in terms of Archimedicity. 

        \begin{Prop}[{\cite[Proposition 2.27]{BIPPthereal}}]\label{Prop: closed points and Archimedicity}
            In the notation of the third item of Proposition \ref{Proposition: various defintion of the real spectrum}, if $(\r,\Ff_\r)$ is an element of $\spec{A}$, then $(\r,\Ff_\r)$ is closed in the spectral topology if and only if $\Ff_\r$ is Archimedean over~$\r(A)$.
        \end{Prop}

        Let now $A$ be a \Kk -algebra where \Kk \ is a real closed field.
        The Archimedean spectrum is, informally, a subset of the real spectrum whose real closed fields are Archimedean over the ground field.
        
        \begin{Rem} \label{Rem: K is a subfield of F}
            Let $A$ be a \Kk -algebra where \Kk \ is a real closed field and $(\r,\Ff)\in \speccl{A}$. The composition of the inclusion $\Kk \rightarrow A$ with \func{\r}{A}{\Ff} gives a field morphism $\Kk \rightarrow \Ff$. Thus $\Kk$ is a subfield of \Ff . Since every positive element of \Kk \ is a square, the order on \Ff \ extends the order on \Kk . 
        \end{Rem}
        
        \begin{Def}
            Given $\Kk$ a real closed field and $A$ a \Kk-algebra, the \emph{Archimedean spectrum} of $A$ is the set
            \[
                \specarch{A}:=\setrelb{(\r,\Ff_\r) \in \spec{A}}{\Ff_\r \text{ is Archimedean over \Kk}},
            \] 
            endowed with the subspace topology from the spectral topology, where $\Ff_\r$ is the real closure of the field of fractions of $\r(A)$.
        \end{Def}

        In particular, from Proposition \ref{Prop: closed points and Archimedicity}, $\specarch{A} \subset \speccl{A}$. In addition, the Archimedean spectrum is, as for the real spectrum, a functor.

        \begin{Prop}
            The functor $(-)_{\mathrm{Arch}}^{\mathrm{RSp}}$ is contravariant from the category of \Kk -algebras to the category of topological spaces. 
        \end{Prop}

        \begin{proof}
            Consider $A,B$ two \Kk-algebras and \func{\pi}{A}{B} a \Kk-algebra morphism. From Theorem \ref{Thm: functoriality of the real spectrum}, 
            \[
                \defmap{\spec{\pi}}{\spec{B}}{\spec{A}}{(\r,\Ff_\r)}{(\r \circ \pi,\Ff_\r)}
            \]
            is a continuous map. Consider \specarch{\pi} the restriction of \spec{\pi} to the Archimedean spectrum of $B$. For an element $(\r,\Ff_\r) \in \specarch{B}$, where $\Ff_\r$ is the real closure of the field of fractions of $\r(B)$, $\Ff_\r$ is Archimedean over \Kk. So in particular, the image of \specarch{\pi} is contained in \specarch{A}. Hence, $(-)_{\mathrm{Arch}}^{\mathrm{RSp}}$ is a contravariant functor from the category of \Kk -algebras to the category of topological spaces.
        \end{proof}

        We presented a contravariant functor $(-)^\mathrm{RSp}$ from commutative rings with a unit to compact topological spaces. Additionally, we characterized the closed points of these induced topological spaces and introduced the Archimedean spectrum. In the next section, we use this framework, along with the coordinate ring of algebraic sets, to examine the real spectrum compactification of algebraic sets.

%% file: Real_spectrum_and_archimedean_spetrum_of_algebraic_sets.tex
\subsection{Functoriality of the real spectrum compactification of algebraic sets} \label{Subsection: Archimedean spectrum}
  
        Using the coordinate ring of the \Kk -extension of an algebraic set~$V \subset \Ll^n$ for some real closed fields \Kk \ such that $\Ll \subset \Kk$, we show that $\speccl{(-)}$ is a functor sending proper algebraic maps to continuous maps---in the coming section, we apply this concept to the algebraic sets $\mathcal{M}_\G$ and $\widehat{\sym}$ defined over $\overline{\Qq}^r$. 
        For the rest of the subsection, let $\Ll,\Kk$ be real closed fields such that $\Ll \subset \Kk$. Consider an algebraic set $V\subset \Ll^n$ and recall that $V(\Kk)$ denotes the \Kk -extension of $V$. 
        Denote by
        \[
            \spec{V(\Kk)}:= \spec{\Kk[V]} 
        \]
        the real spectrum of $V(\Kk)$, where $\Kk[V]$ is the coordinate ring of $V(\Kk)$, see Definition \ref{Def: coordinate ring}.
        Moreover, we endowed the \Kk-points $V(\Kk)$ with the Euclidean topology coming from the norm \func{N}{\Kk^n}{\Kk_{\geq 0}}, as defined in Subsection \ref{Subsection: preliminaries in real algebraic geometry}. 

        \begin{Rem}\label{Rem: Euclidean topology is equivalent to the spectral topology}
            The Euclidean topology on $V(\Kk)$ is equivalent to the topology generated by the basis of open sets
            \[
                U(f_1, \ldots, f_p) := \setrel{v \in V(\Kk)}{f_1(v) > 0, \ldots, f_p(v) > 0},
            \]
            for $f_1, \ldots, f_p\in \Kk[V]$, see \cite[Subsection 2.1]{BCRrea}.
        \end{Rem}

    \begin{Lem}\label{Lem: polynomial bound for proper maps}
        Let $V\subset \Rr^n$, $W \subset \Rr^m$ be algebraic sets. If the coordinate ring of $V(\Rr)$ is $\Rr[V]=\Rr[x_{1},\ldots, x_n]/I\left(V\right)$ and $\func{\pi}{V(\Rr)}{ W(\Rr)}$ is a proper algebraic map, then there exist constants $c,d \in \Nn$ such that for every $v\in V(\Rr)$:
        \[
             |x_i(v)| \leq c\left(1+N(\pi(v))^2\right)^d. 
        \]
    \end{Lem}

    \begin{proof}
        Since $\pi$ is proper, for every $w \in W(\Rr)$ the fiber $\pi^{-1}(w) \subset V(\Rr)$ is compact in the Euclidean topology. The map $x_i: V \to \Rr$ is continuous, so that it attains its maximum on $\pi^{-1}(w)$ by the extreme value theorem. Hence, the map
        \[ 
            \defmap{\tilde{v}_i}{W(\Rr)}{\Rr}{w}{\max \setrelfrac{x_i(v)}{v \in \pi^{-1}(w)}}
        \]
        is well defined.
        Consider the graph $X_{\pi} := \setrel{(v,w) \in (V \times W)(\Rr)}{ \pi(v) = w}$ of $\pi$,
        which is algebraic because $V$, $W$, and $\pi$ are algebraic and $X_{\tilde{v}_i} := \setrel{(w,t) \in W(\Rr) \times \Rr}{\tilde{v}_i(w) = t}$ the graph of $\tilde{v}_i$.
        By the definition of $\tilde{v}_i$, it holds
        \begin{align*}
            \tilde{v}_i(w) = t \quad \iff \quad & \exists v \in V(\Rr), \; \pi(v)=w, \; x_i(v)=t, \\ & \text{and } \forall v' \in V(\Rr), \; \pi(v')=w \Rightarrow x_i(v') \leq t.
        \end{align*}
        Since all sets and maps involved are semialgebraic, the set
        \[
            X = \setrelfrac{(v,w,t)\in (V\times W)(\Rr) \times \Rr}
            {\begin{tabular}{@{}l@{}}
            $\pi(v) = w, \; x_i(v)=t$, \\
            \text{and } $\forall v' \in V, \; \pi(v')=w \Rightarrow x_i(v') \leq t$ 
            \end{tabular}}
        \]
        is semialgebraic. In particular, its projection on $W(\Rr)\times \Rr$ is a semialgebraic set \cite[Theorem 2.2.1]{BCRrea}, which identifies with $X_{\tilde{v}_i}$. Thus $\tilde{v}_i$ is semialgebraic.
        By \cite[Proposition 2.6.2]{BCRrea}, there exist constants $c > 0$ and $d \in \mathbb{N}$ such that
        \[
            |\tilde{v}_i(w)| \leq c\left(1+N(w)^2\right)^d \quad \text{for all } w \in W(\Rr).
        \]
        So in particular, $|x_i(v)| \leq c(1+N(\pi(v))^2)^d$ \fev \ $v\in V(\Rr)$.
    \end{proof}

    \begin{Thm} \label{Thm: proper algebraic map and closed points}
        Let $V\subset \Rr^n$, $W \subset \Rr^m$ be algebraic sets. If $\func{\pi}{V(\Rr)}{ W(\Rr)}$ is a proper algebraic map, then the image of \speccl{\pi}, the restriction to the closed points of the induced map $\spec{\pi}$, is $\speccl{W(\Rr)}$. That is, we have a continuous surjective map 
        \[
            \speccl{\pi}\colon \speccl{V(\Rr)}\twoheadrightarrow \speccl{W(\Rr)}.
        \]
    \end{Thm}
    
    \begin{proof}
        As in the notation of the third item of Proposition \ref{Proposition: various defintion of the real spectrum}, consider $(\r,\Ff_\r) \in  \speccl{V(\Rr)}$, where $\Ff_\r$ is the real closure of the field of fractions of $\r(\Rr[V])$. 
        We first prove that $\Ff_\r$ is Archimedean over $\r(\Rr[W])$. 
        Consider the coordinate rings
        \[
            \Rr\left[V\right] := \Rr[x_{1},\ldots, x_n]/I\left(V\right) \text{ and } \Rr[W] := \Rr[y_1,\ldots,y_m]/I(W),
        \]
        as in Definition \ref{Def: coordinate ring} and the element $v_\rho := (\r(x_1),\ldots, \r(x_n)) \in V(\Ff_\r)$ defined in Notation \ref{Notation element of VF}.
        By Lemma \ref{Lem: polynomial bound for proper maps}, there exist constants $c,d \in \Nn$ such that  
        \begin{align*}
          \fev \ v \in V(\Rr): \left|x_i(v)\right|&\leq c\left(1+N(\pi(v))^2\right)^d.
        \end{align*}
        So by the Transfer principle (Theorem \ref{Thm transfer principle}) the following inequality holds 
        \begin{align*}
            \left|x_{i}(v_\r)\right| &\leq c(1+N\left(\pi_{\Ff_\r}(v_\r))^2\right)^d, 
        \end{align*}
        where $\pi_{\Ff_\r}$ is the $\Ff_\r$-extension of $\pi$.
        Hence $x_{i}(v_\r)$ is bounded by some element of $\r(\Rr[W])$.
        Finally, since $(\r,\Ff_\r)$ is a closed point of the real spectrum, $\Ff_\r$ is Archimedean over $\r (\Rr[V])$ by Proposition~\ref{Prop: closed points and Archimedicity}.
        Hence $\Ff_\r$ is Archimedean over $\r(\Rr[W])$ so that 
        \[
            \mathrm{Im}\left(\speccl{\pi}\right)\subset \speccl{W(\Rr)}.
        \]
        Finally, \speccl{\pi} is surjective by \cite[Lemma 7.6]{BIPPthereal} so that
        \[
            \mathrm{Im}\left(\speccl{\pi}\right) = \speccl{W(\Rr)}.
        \]
    \end{proof}
        
         In addition, the Euclidean topology allows us to embed an algebraic set into a part of the real spectrum that preserves information about its geometry.

        \begin{Thm}[{\cite[Proposition 7.1.5]{BCRrea} and \cite[Proposition 7.1.5]{BCRrea}}]\label{Thm: map from the algebraic set to the archimedean spectrum}
        Let $\Ll \subset \Kk$ be real closed fields and $V\subset \Ll^n$ an algebraic set, then
        the evaluation map
        \[
            \defmap{ev}{V(\Kk)}{\specarch{V(\Kk)}}{(v_1,\ldots,v_n)}{(ev(v_1,\ldots,v_n),\Kk)}
        \]
        is a continuous injection from $V(\Kk)$, with its Euclidean topology,
        to \specarch{V(\Kk)} with its spectral topology. Moreover, $V(\Kk)$ is dense in \spec{V(\Kk)} and so in its Archimedean spectrum.
        \end{Thm}

        \begin{proof}
            By \cite[Proposition 7.1.5]{BCRrea}, $\func{ev}{V(\Kk)}{\specarch{V(\Kk)}}$ is continuous and injective. By \cite[Corollary 2.32]{BIPPthereal}, $V(\Kk)$ is dense in \spec{V(\Kk)}. 
        \end{proof}

        From Theorem \ref{Thm: map from the algebraic set to the archimedean spectrum}, we define the \emph{real spectrum compactification} of $V(\Kk)$ as the closure of the image of the evaluation map in the spectral topology:
        \[ 
             \rsp{V(\Kk)} := \overline{ev(V(\Kk))}.
        \]

        \begin{Prop}[{\cite[Proposition 2.33]{BIPPthereal}}]\label{Prop: map from the algebraic set to the closed points}
            Let $V\subset (\overline{\Qq}^r)^n$ be an algebraic set.
            The evaluation map
            \[
                \defmap{ev}{V(\Rr)}{\speccl{V\left(\overline{\Qq}^r\right)}}{(v_1,\ldots,v_n)}{(ev(v_1,\ldots,v_n),\Rr)}
            \]
            is a homeomorphism from $V(\Rr)$, with its Euclidean topology,
            onto its image with the spectral topology. 
            Moreover, $V(\Rr)$ is open and dense in \speccl{V(\overline{\Qq}^r)} and \speccl{V(\overline{\Qq}^r)} is metrizable.
        \end{Prop}
        
        \begin{Rem}
            We study the real spectrum compactification of an algebraic model of~$\mathcal{M}_\G$ and $\widehat{\sym}$ which are defined over $\overline{\Qq}^r$.         
            These models depend on the choice of coordinates. However, every choice of coordinates leads to models that are related by a canonical algebraic isomorphism. 
            Consequently, the real spectrum compactifications of both models are homeomorphic \cite[Proposition 7.2.8]{BCRrea} and the real spectrum compactifications~$\rsp{\mathcal{M}_\G(\Rr)}$ and~$\rsp{\widehat{\symr}}$ are canonical. 
        \end{Rem}
        Despite its name, the real spectrum compactification is not necessarily a natural compactification of $V(\Kk)$, where a \emph{natural compactification} of a topological space $X$ is a Hausdorff compact topological space $Y$ such that $X\subset Y$ and $X$ is open and dense in $Y$.

        \begin{Ex}
            Consider the affine line $\mathbb{A}^1 \subset \overline{\Qq}^r$. By Example \ref{Ex: real spectrum of the line}, there is an homeomorphism
            \[
                \rsp{\mathbb{A}^1\left(\overline{{\Qq}}^r\right)} = \speccl{\overline{\Qq}^r[x]} \cong \Rr \cup \set{\pm \infty}.
            \]
            By density of the transcendental numbers in \Rr , the topological space $\mathbb{A}^1(\overline{{\Qq}}^r)$ is not open in \speccl{\mathbb{A}^1(\overline{\Qq}^r)} such that the real spectrum compactification is not a natural compactification. 
        \end{Ex}

        We introduced the real spectrum compactification of an algebraic set and established the proper framework to demonstrate its functoriality. However, depending on the field over which the algebraic set is defined, this compactification is not always natural. To address this issue, we further investigate the topological properties of the Archimedean spectrum.

    \subsection{Local compactness of the Archimedean spectrum of algebraic sets}
        Under suitable conditions on \Kk, the Archimedean spectrum is a locally compact topological space open in \speccl{V(\Kk)}. 
        This makes \speccl{V(\Kk)} a natural compactification of \specarch{V(\Kk)}.
        To illustrate this, we compute the Archimedean spectrum of the \Kk -extension of $V$ when \Kk \ is Archimedean, as well as the Archimedean spectrum of the affine line over any real closed field.
        
        \begin{Def}\label{Def: big elements}
            The element $b\in \Kk$ is a \emph{big element}, if \fev \ $c\in \Kk$, \tes \ $k\in \Nn$ that verifies $c < b^k$. 
        \end{Def}

        For the remainder of the subsection, \Ll \ is a real closed field, $\Ll \subset \Kk$ a real closed field with a big element, $V \subset \Ll^n$ an algebraic set, and we show that \speccl{V(\Kk)} is a natural compactification of the topological space \specarch{V(\Kk)}.

        \begin{Rem}
            \Fev \ $(\r,\Ff_\r)\in \speccl{V(\Kk)}$, $\Ff_\r$ is a real closed field of finite transcendence degree over~$\Kk$, see Proposition \ref{Proposition: various defintion of the real spectrum}. In particular, if \Kk \ has a big element, then $\Ff_\r$ also has a big element \cite[Section 5]{Bthe}. 
            Hence, if \Kk \ is a real closed field that appears in the boundary of \speccl{V(\Rr)}, where $V\subset (\overline{\Qq}^r)^n$ is an algebraic set, then the real closed fields that appear in \speccl{V(\Kk)} are also real closed fields of finite transcendence degree over $\overline{\Qq}^r$ that contain a big element \cite[Section 5]{Bthe}. 
        \end{Rem}

        In this setting, we give a description of \specarch{V(\Kk)} as an open and dense subset of \speccl{V(\Kk)} which is a countable union of compact sets. 
        To do this, we need constructible sets.

        \begin{Def}[{\cite[Definition 7.1.10]{BCRrea}}]\label{Def:constructible sets}
            Let $A$ be a commutative ring with a unit. A \emph{constructible subset} of \spec{A} is a finite boolean combination of basic open subsets~$\tilde{U}(a_1, \ldots , a_n)$. That is, obtained from the basic open sets of the real spectrum topology by taking finite unions, finite intersections and complements.
        \end{Def}
        
        Constructible sets form the essential building blocks of compact sets in the real spectrum topology. Moreover, they offer a correspondence between semialgebraic subsets of an algebraic set $V$ and compact subsets of \rsp{V(\Kk)}. They are therefore essential, as the following result shows.

        \begin{Prop}[{\cite[Corollary 7.1.13, Proposition 7.2.2, and Theorem 7.2.3]{BCRrea}}] \label{Prop: constructible sets}
            Let $\Ll,\Kk$ be real closed fields such that $\Ll \subset \Kk$, $V\subset \Ll^n$ an algebraic set, and
            $S$ a semialgebraic subset of $V$.
            \begin{enumerate}
                \item Every constructible subset of \spec{V(\Kk)} is compact with respect to the spectral topology. Moreover, an open subset of \spec{V(\Kk)} is constructible if and only if it is compact.
                \item There exists a unique constructible set $\widetilde{S(\Kk)} \subset \spec{V(\Kk)}$, so that 
                \[
                    \widetilde{S(\Kk)} \cap V(\Kk) = S(\Kk).
                \]
                \item If $S(\Kk)$ is a boolean combination of 
                \[
                    U(f_i) = \setrel{v \in V(\Kk)}{f_i(v) > 0},
                \]
                where $f_i \in \Kk[V]$, then $\widetilde{S(\Kk)}$ is the same boolean combination of
                \[
                    \tilde{U}(f_i) = \setrel{(\r,\Ff_\r) \in \spec{V(\Kk)}}{\r(f_i) > 0}.
                \]
                \item The mapping $S(\Kk) \mapsto \widetilde{S(\Kk)}$ is an isomorphism from the boolean algebra of semialgebraic subsets of $V(\Kk)$ onto the boolean algebra of constructible subsets of \spec{V(\Kk)}.
                \item The semialgebraic set $S(\Kk)$ is open (respectively closed) in $V(\Kk)$ if and only if $\widetilde{S(\Kk)}$ is open (respectively closed) in \spec{V(\Kk)}. Hence, the isomorphism $S(\Kk) \mapsto \widetilde{S(\Kk)}$ induces a bijection from the family of open semialgebraic subsets of $V(\Kk)$ onto the family of compact open subsets of \spec{V(\Kk)}.
            \end{enumerate}
        \end{Prop}

        \begin{Rem}
             Using constructible sets, it is possible to extend the definitions of Section~\ref{Section: Real spectrum and Archimedean spectrum}, in particular the real spectrum compactification, to semialgebraic sets \cite[Subsection 7.2]{BCRrea}. 
        \end{Rem}  

        Constructible sets can exhibit surprising behavior. In the following example, the closed points of a constructible set are not the intersection of the closed points of the algebraic set with the constructible set.
        
        \begin{Ex}[{\cite[Example 2.34]{BIPPthereal}}]
            Let $\mathbb{A}^1 \subset \overline{\Qq}^r$ denotes the affine line, and let \Kk \ be a real closed subfield of \Rr. For all $u \in \Kk$ 
            \[
                \overline{\setfrac{\a_{u^+}}} = \setfrac{\a_u, \a_{u^+}} \text{ and } \overline{\setfrac{\a_{u^-}}} = \setfrac{\a_u, \a_{u^-}}.
            \]
            Moreover, for $u\in \Rr \backslash \Kk$ it holds $\a_{u^+}=\a_{u^-}=\a_u$, which are closed points of the real spectrum. Thus 
            \[
                \speccl{\mathbb{A}^1(\Kk)}=\setrelfrac{\a_u}{u \in \mathbb{R}} \cup \set{\a_{\pm \infty}}, 
            \]
            where $\mathbb{R}$ embeds as an open and dense subset.
            Semialgebraic subsets of $\mathbb{A}^1(\Kk)$ are finite unions of intervals and half-lines with endpoints in $\Kk$. For the semialgebraic set $S(\Kk) := (s, t] \cap \Kk$ with $s, t \in \Kk$, it holds  
            \begin{align*}
                &\widetilde{S(\Kk)} = (s, t] \cup \setrelfrac{\a_{u^\pm}}{u \in (s, t) \cap \Kk} \cup \setfrac{\a_{s^+}, \a_{t^-}}, \\
                &\widetilde{S(\Kk)} \cap \speccl{V(\Kk)} = S(\Kk).
            \end{align*}
            In particular, $\widetilde{S(\Kk)} \cap \speccl{\mathbb{A}^1(\Kk)}$ is not compact.
            However, the closed points of $\widetilde{S(\Kk)}$ are 
            \[
                \widetilde{S(\Kk)}_{\mathrm{cl}}  = (s, t] \cup \setfrac{\a_s{^+}},
            \]
            which is homeomorphic to a closed segment in $\mathbb{R}$. Note that $\a_{s^+}$ is closed in $\widetilde{S(\Kk)}$ but not in $\spec{\mathbb{A}^1(\Kk)}$.
        \end{Ex}
        
        To describe the Archimedean spectrum as a countable union of compact sets, we use the following notation. 

        \begin{Not}[{\cite[Remark 2.36]{BIPPthereal}}] \label{Notation element of VF}
            Let $\Ll,\Kk$ be real closed fields such that $\Ll \subset \Kk$, $V\subset \Ll^n$ an algebraic set, and $(\r,\Ff_\r) \in \spec{\Kk[V]}$, where $\Ff_\r$ is the real closure of the field of fractions of $\r(\Kk[x_1, \ldots , x_n])$.
            If $x_{\r} = (\r(x_1), \ldots , \r(x_n)) \in V({\Ff_\r})$, then 
            \[
                (\r,\Ff_\r) =(ev({x_{\r}}),\Ff_\r).
            \]
            Thus, for all $f\in \Kk[V]$, we note $f(x_\r) := \r(f)$.
        \end{Not}
    
        \begin{Thm}\label{Thm: The Archimedean spectrum is a countable union of open compact sets}
            Let \Ll \ be a real closed field, $\Ll \subset \Kk$ a real closed field with a big element $b$, and $V\subset \Ll^n$ an algebraic set. The Archimedean spectrum of $V(\Kk)$ is an open subset of \rsp{V(\Kk)} which is a countable union of compact subsets of \rsp{V(\Kk)}. In particular, the space \specarch{V(\Kk)} is \s-compact and locally compact. 
        \end{Thm}
    
        \begin{proof}
            Consider the coordinates $x_1, \ldots , x_n$ such that $\Kk[V] = \Kk[x_1, \ldots , x_n]/I(V)$, where $I(V)$ is the ideal of polynomials vanishing on $V$. Write $x=(x_1, \ldots , x_n)$ and let $f \in \Kk[V]$ be an element with coordinate decomposition 
            \[
                f(x):=\sum_{I \text{ multiindex}} c_{I}x^I,
            \]
            where if $I=(i_1,\ldots, i_n)$ for some $i_j\geq 0$, then $x^I:= x_1^{i_1}\cdots x_n^{i_n}$.
            By the Cauchy--Schwartz inequality---a consequence of the Transfer principle (Theorem \ref{Thm transfer principle}), it holds for every $v\in \Kk^n$
            \begin{align*}
                f(v)^2 = \left(\sum_I c_{I}v^I\right)^2 \leq \left(\sum c_{I}^2\right) \left(\sum v^{2I}\right). 
            \end{align*}
            Define $g(x):= \sum x_i^2$ so that $v_i^2 \leq g(v)$. With the notation $|I|:=\sum i_j$, it holds $v^{2I}\leq g(v)^{|I|}$ and  
            \[
                f(v)^2 \leq \left(\sum c_{I}^2\right) \left(\sum g(v)^{|I|}\right).
            \]
            Either $g(v)\leq 1$ and 
            \[
                f(v)^2 \leq \left(\sum c_{I}^2\right)|\mathrm{supp}(f)|,
            \]
            where $\mathrm{supp}(f) := \setrel{I \text{ multiindex}}{c_I\neq 0}$.
            Or $g(v)\geq 1$ and 
            \[
                f(v)^2 \leq \left(\sum c_{I}^2\right)|\mathrm{supp}(f)| g(v)^{d(f)},
            \]
            where $d(f):= \max \setrel{|I|}{I\in \mathrm{supp}(f)}.$
            In either case, since $b$ is a big element of $\Kk$, there exists $m\in \Nn$ such that  
            \[
                f(v)^2 \leq b^m \left(1+g(v)^{d(f)}\right) \quad \forall v\in \Kk^n.
            \]
            Let $(\r,\Ff_\r) \in \spec{V(\Kk)}$ and $x_\r$ be the point of $V(\Ff_\r)$ defined in Notation \ref{Notation element of VF}. The previous inequality holds in every real closed field. Thus, \fev \ $f\in \Kk[V]$, \te \ $m\in \Nn$ such that
            \[
                f(x_\r)^2 \leq b^{m} \left(1+g(x_\r)^{d(f)}\right).
            \]
            So, define for every~$k\in \Nn$, the open set $A_k(\Kk)$ and the closed set $B_k(\Kk)$ as
            \[
                A_k(\Kk) := \setrelfrac{v\in V(\Kk)}{\sum v_i^2 < b^k}, \
                B_k(\Kk) := \setrelfrac{v \in V(\Kk)}{\sum v_i^2 \leq b^k}. 
            \]  
    
            \begin{claim}
                The Archimedean spectrum of $V(\Kk)$ contains $V(\Kk)$ and is an open subset of \speccl{V(\Kk)} as the following equality holds 
                \[
                    \specarch{V(\Kk)} = \bigcup_{k\in \Nn} \widetilde{A_k(\Kk)} \cap \speccl{V(\Kk)}.
                \]
                Moreover, \specarch{V(\Kk)} is a countable union of compact sets as
                \[
                    \specarch{V(\Kk)} = \bigcup_{k\in \Nn} \widetilde{B_k(\Kk)} \cap \specclf{V(\Kk)}.
                \]
            \end{claim}
    
            \begin{claimproof}
                On the one hand, if $(\r,\Ff_\r)\in \specarch{V(\Kk)}$, then $\Ff_\r$ is Archimedean over $\Kk$. In particular, it is Archimedean over~$\r(\Kk[V])$ (see Remark \ref{Rem: K is a subfield of F}) so that $(\r,\Ff_\r) \in \speccl{V(\Kk)}$. Since $\Ff_\r$ is Archimedean over \Kk \ and $b$ is a big element of \Kk , $b$ is also a big element of $\Ff_\r$. Thus \tes \ a natural number~$k$ that verifies $g(x_\r)<b^k$ so that the following inclusions hold 
                \[
                    \specarchf{V(\Kk)} \subset \bigcup_{k\in \Nn} \widetilde{A_k(\Kk)} \cap \speccl{V(\Kk)} \subset \bigcup_{k\in \Nn} \widetilde{B_k(\Kk)} \cap \speccl{V(\Kk)}.
                \]
                On the other hand, consider $(\r,\Ff_\r) \in \widetilde{B_k(\Kk)} \cap \speccl{V(\Kk)}$ for some fixed $k\in \Nn$. 
                By Proposition \ref{Prop: closed points and Archimedicity}, $\Ff_\r$ is Archimedean over~$\r(\Kk[V])$ and we show that $\r(\Kk[V])$ is Archimedean over \Kk. 
                Consider $f \in \Kk[V]$ and $m_1\in \Nn$ such that 
                \[
                    f(x_\r)^2 < b^{m_1}\left(1+g(x_\r)^{d(f)}\right).
                \]
                Since $(\r,\Ff_\r)\in \widetilde{B_k(\Kk)}$ it holds $g(x_\r)< b^k$. So, there exists $m_2\in \Nn$ with 
                \[
                    f(x_\r)^2 < b^{m_1 + m_2}.
                \]
                Hence $\r(\Kk[V])$ is Archimedean over~\Kk . Thus, also $\Ff_\r$ is Archimedean over~$\Kk$ and
                \[
                    \specarch{V(\Kk)} = \bigcup_{k\in \Nn} \widetilde{A_k(\Kk)} \cap \speccl{V(\Kk)} = \bigcup_{k\in \Nn} \widetilde{B_k(\Kk)} \cap \speccl{V(\Kk)}.
                \]          
            \end{claimproof}          
            In particular, the Archimedean spectrum of an algebraic set is the intersection of a countable union of open constructible sets with \speccl{V(\Kk)}. Consequently, it forms an open subset of a compact Hausdorff space and, therefore, is locally compact \cite[Corollary 29.2]{Mtopology}.
        \end{proof}
        
        \begin{Cor}\label{Prop: natural compactification of the Archimedean spectrum}
            Let \Ll \ be a real closed field, $\Ll \subset \Kk$ a real closed field with a big element $b$, and $V\subset \Ll^n$ an algebraic set. The topological space \specarch{V(\Kk)} is an open and dense subset of \speccl{V(\Kk)}. 
            In particular, \speccl{V(\Kk)} is a natural compactification of $\specarch{V(\Kk)}$.
        \end{Cor}

        \begin{proof}
            By Theorem \ref{Thm: The Archimedean spectrum is a countable union of open compact sets}, $\specarch{V(\Kk)}$ is open in \speccl{V(\Kk)}. 
            Moreover, the evaluation map on elements of $V(\Kk)$ provides a topological embedding of $V(\Kk)$ in~\specarch{V(\Kk)} as a dense subset (Theorem \ref{Thm: map from the algebraic set to the archimedean spectrum}). Thus \specarch{V(\Kk)} is also dense in \speccl{V(\Kk)}. 
        \end{proof}        

        \begin{Prop} \label{Prop: map from the algebraic set to the archimedean spectrum}
            Let $\Ll,\Kk$ be real closed fields such that $\Ll \subset \Kk$, \Kk \ is Archimedean over \Zz , and $V\subset \Ll^n$ an algebraic set.
            The evaluation map
            \[
                \defmap{ev}{V(\Rr)}{\specarch{V(\Kk)}}{(v_1,\ldots,v_n)}{(ev(v_1,\ldots,v_n),\Rr)}
            \]
            is an homeomorphism from $V(\Rr)$, with its Euclidean topology, to \specarch{V(\Kk)} with its spectral topology.
            Thus, \speccl{V(\Kk)} is a natural compactification of $V(\Rr)$.
        \end{Prop}

        \begin{proof}
            First, Notice that \Rr \ is Archimedean over \Zz \ so that $ev$ is well defined.
            By \cite[Proposition 2.33 (1)]{BIPPthereal}, $ev$ is a continous and open injection from $V(\Rr)$ with its Euclidean topology to \speccl{V(\Kk)} with its spectral topology, see Remark~\ref{Rem: Euclidean topology is equivalent to the spectral topology}. So it remains to show that $ev$ is a surjection onto the Archimedean spectrum.
            
            Let $(\r,\Ff_\r)\in \specarch{V(\Kk)}$ and consider $\func{\sigma}{\Ff_\r}{\Rr}$ an ordered field monomorphism as in \cite[Theorem 3.5]{Hcompletenessoforderedfields}. Then by the fourth item of Proposition \ref{Proposition: various defintion of the real spectrum}
                \[
                    (\r,\Ff_\r) = (\sigma \r, \Rr) \in \specarch{V(\Kk)}.
                \]
                Let $\Kk[V]=\Kk[x_1,\ldots,x_n]/I(V)$ be the coordinate ring of $V(\Kk)$ as in Definition \ref{Def: coordinate ring}. Using Notation \ref{Notation element of VF}, it holds $(\sigma \r(x_1),\ldots,\sigma \r(x_n)) \in V(\Rr)$ and 
                \[ 
                    \left(ev(\sigma \r(x_1),\ldots,\sigma \r(x_n)),\Rr \right) = \left(\sigma \r, \Rr \right) \in \specarch{V(\Kk)}.
                \]
                Thus $ev$ is surjective as wanted, hence a homeomorphism.             
        \end{proof}

        \begin{Rem}
            One might wish to generalize this result to the setting of complete Archimedean fields, see \cite[Section 3]{CSonnonArchimedeanvaluedfields}. However, the notion of Archimedicity used to define generalized Hahn fields is stronger than the definition of Archimedicity employed in this text, so that the uniqueness of a complete Archimedean field fails.
        \end{Rem}

        As an example of computation, we describe the Archimedean spectrum of $\mathbb{A}^1$ in terms of Dedekind cuts, where a \emph{Dedekind cut} of a real field \Ll \ is a partition of \Ll \ into two nonempty subsets $D_-$ and $D_+$ such that $D_-$ is closed downwards and does not contain a greatest element.
    
        \begin{Prop}\label{Proposition: Archimedean spectrum of the affine line}
            For a real closed field \Kk 
            \[
                \specarch{\mathbb{A}^1(\Kk)} = \Kk \cup \setrelfrac{\Kk = D_- \cup D_+}
                {\begin{tabular}{@{}l@{}}
                $D_-\neq \emptyset$ \text{ has not lowest upper bound}, \\
                $D_+ \neq \emptyset$ \text{ has no greatest lower bound}
                \end{tabular}}.
            \]
        \end{Prop}

        \begin{proof}
            With the notation of the third item of Proposition \ref{Proposition: various defintion of the real spectrum}, let $(\r,\Ff_\r)$ be an element of $\spec{\mathbb{A}^1(\Kk)}$. The coordinate ring $\Kk [x]$ of $\mathbb{A}^1(\Kk)$ is principal so that the kernel of~$\r$ is either trivial or a principal ideal generated by an irreducible polynomial~$f$. On the one hand, suppose $\mathrm{ker}(\r) = (f)$. If~$f$ is linear, then $\Ff_\r = \Kk$ and $\r$ is the evaluation on an element in~\Kk . Otherwise, since~\Kk \ is real closed, $f$ is quadratic. However, in this case, $\Ff_\r = \Kk (\sqrt{-1})$ which is not orderable. This is a contradiction with $\Ff_\r$ is real closed. 
            On the other hand, $\r$ is injective and gives an ordering of~$\Kk (x) = \mathrm{Frac}(\Kk[x])$. Thus 
            \begin{align*}
                \spec{\mathbb{A}^1(\Kk)} = & \, \Kk \cup \set{\text{orderings on } \Kk(x)}, \\
                \specarch{\mathbb{A}^1(\Kk)} = & \, \Kk \cup \set{\text{orderings on } \Kk(x) \text{ Archimedean over } \Kk}.
            \end{align*}
            \begin{claim}
                There is a one-to-one correspondence between the orderings of $\Kk(x)$ which are Archimedean over \Kk \ and the Dedekind cuts of \Kk \ with no lowest upper bound and no greatest lower bound.
            \end{claim}
            
            \begin{claimproof}  
                On the one hand, consider an ordering $\leq$ on $\Kk(x)$ Archimedean over \Kk \ and the sets
                \begin{align*}
                    D_- := \setrel{a\in \Kk}{a \leq x}, \\
                    D_+ := \setrel{a\in \Kk}{x \leq a}.
                \end{align*}
                Since the ordering is Archimedean over \Kk, $D_+$ is nonempty and $x^{-1}$ is bounded from above. So the set $D_-$ is nonempty. Suppose $D_-$ has a lowest upper bound $u$ in \Kk . If $u$ is an element in~$D_-$, then $u \leq x \leq u+\e$ for every~$0 \leq \e \in \Kk$. In particular, 
                \[
                    \e^{-1} \leq (x-u)^{-1} \text{ for every } 0 \leq \e \in \Kk,
                \]
                which is a contradiction with the ordering on $\Kk(x)$ is Archimedean over \Kk. 
                If $u$ is an element in $D_+$, then $u-\e \leq x \leq u$ \fev \ $0 \leq \e \in \Kk$ so that 
                \[
                    \e^{-1} \leq (u-x)^{-1} \ \fev \ 0 \leq \e \in \Kk.
                \]
                This is also a contradiction with: the ordering on $\Kk(x)$ is Archimedean over \Kk. 
                Hence $D_-$ does not have a lowest upper bound, and by a similar reasoning, $D_+$ does not have a greatest lower bound. Therefore, 
                an Archimedean ordering on $\Kk(x)$ defines a Dedekind cut of \Kk \ with no lowest upper bound and greatest lower bound.
            
                On the other hand, consider a partition $\Kk = D_- \cup D_+$ where $D_-$ has not lowest upper bound and $D_+$ has no greatest lower bound. Define an ordering on~$\Kk (x)$ generated by
                \begin{align*}
                    a < x \ \fev \ a\in D_-, \\
                    x   < b \ \fev \ b\in D_+.
                \end{align*}
                Since \Kk \ is real closed, any $g \in \Kk (x) \backslash \set{0}$ has a decomposition
                \[
                    g(x)=\prod_{c_i \in D_-}(x-c_i)^{m_i}\prod_{c_j \in D_+}(x-c_j)^{m_j}\prod_{k=1}^{\ell}h_k(x),
                \]
                where the $c_i$'s are the distinct roots of $g$, and $h_k$ are irreducible quadratic polynomials. In particular, $0<h_k$ \fev \ $1\leq k \leq \ell$ so that 
                \[
                    0<g \text{ if and only if } \sum_{c_j \in D_+}m_j \in 2\Zz.
                \]
                Equivalently $0<g$ \iff \ \tes \ $a\in D_-$ such that the restriction $g|_{\Kk_{> a}\cap D_-}$ is strictly positive. We show that this ordering is Archimedean. 
                By definition of the ordering, the linear terms of the decomposition of~$g$ are bounded by some element in the ordered field \Kk . Therefore it remains to prove that the quadratic terms are also bounded in \Kk . 
                Since all $h_k$ are irreducible, \tes \ $0<\e$ with $\e <\prod_k h_k(a)$ \fev \ $a\in \Kk$.
                In particular, \tes \ $\e ' \in \Kk$ so that $g^{-1}<\e '$. This is true for every~$g$ in the real field $\Kk(x)$ so that the ordering is Archimedean. 
            \end{claimproof}
            \vspace{-1em}
        \end{proof}  

        \begin{Rem}
            For every non-Archimedean real closed field \Kk,  \specarch{\mathbb{A}^1(\Kk)} does not have the structure of a field. Indeed, the natural addition of the Dedekind completion of a non-Archimedean field is not invertible, see \cite[Lemma 3.10]{Hcompletenessoforderedfields}.
        \end{Rem}

        We introduced the Archimedean spectrum of the \Kk-extension of an algebraic set $V$ and studied one of its natural compactifications given by the closed points of the real spectrum.  
        Before further exploring the Archimedean spectrum using the Berkovich analytification of~$V(\Kk)$ in Section \ref{Section: The Archimedean spectrum and the real analytification of an algebraic set},
        we first study universal geometric spaces over the real spectrum compactification of $\mathcal{M}_\G$.
        This allows us to show strong convergence properties within the real spectrum in Subsection~\ref{Subsection: The projection map and the Archimedean spectrum of the symmetric space}. Consequently, we gain a deeper understanding of the dynamical properties of $\mathcal{M}_\G$ by leveraging the local compactness of the Archimedean spectrum.

%% file: The_universal_projective_space.tex
\subsection{The universal projective space}\label{Subsection: The universal projective space}

    We examine the lift of the projection map $\func{\pi}{(V\times W)(\Rr)}{V(\Rr)}$ when $V\subset \Rr^n$ is an algebraic set and $W\subset \Rr^m$ is a compact algebraic set. 
    We analyze the fibers of $\func{\speccl{\pi}}{\speccl{(V\times W)(\Rr)}}{\speccl{V(\Rr)}}$ using the Archimedean spectrum and show that the fibers are well organized over \speccl{V(\Rr)}. 
    As an application, we construct a universal projective space over \rsp{\mathcal{M}_\G(\Rr)} that encodes degeneracies of \G -actions on \projspnr \ induced by representations in $\mathcal{M}_\G(\Rr)$.    

    \begin{Rem} \label{Rem: decomposition of spectral morphism}
        Let $\Ll \subset \Kk$ be real closed fields and $V\subset \Ll^n$, $W \subset \Ll^m$ algebraic sets. From \cite[Theorem 2.8.3. (iii)]{BCRrea}, there exists a natural \Kk -algebra isomorphism
        \[
            \beth: \Kk[V]\otimes_\Kk \Kk[W] \cong \Kk[V \times W]
        \]
        given by $\beth(\sum f_i \otimes g_i)(v,w)=\sum f_i(v)g_i(w)$ for every $(v,w) \in (V\times W)(\Kk)$, see \cite[Chapter 1, Subsection 2.2, Example 4, page 17]{Sbasicalgebraicgeometry}. In particular, by the universal property of the tensor product, for every $(\r,\Ff)\in \spec{(V\times W)(\Kk)}$, there exists $(\rho^1,\Ff) \in \spec{V(\Kk)}, (\rho^2,\Ff) \in \spec{W(\Kk)}$ such that the following diagram commutes
        \[\begin{tikzcd}
            {\mathbb{K}[V]} \\
            {\mathbb{K}[V] \otimes_{\mathbb{K}} \mathbb{K}[W]} & {\mathbb{K}[V \times   W]} & {\mathbb{F},} \\
            {\mathbb{K}[W]}
            \arrow[from=1-1, to=2-1]
            \arrow["\rho^1", bend left=20,swap, from=1-1, to=2-3]
            \arrow["\beth", from=2-1, to=2-2]
            \arrow["\rho", from=2-2, to=2-3]
            \arrow[from=3-1, to=2-1]
            \arrow["\rho^2", bend right=20, from=3-1, to=2-3]
        \end{tikzcd}\]
        where the vertical arrows are the inclusions.
        Hence, for every $h \in \mathbb{K}[V \times   W]$
        \[
            \r(h)=\sum\r^1\left(h_i^1\right)\r^2\left(h_i^2\right) \quad \text{where } \beth^{-1}(h) = \sum h_i^1\otimes h_i^2 \in \Kk[V]\otimes_\Kk \Kk\left[W\right].
        \]
        We denote by $(\r^1,\r^2,\Ff)$ this decomposition of the morphism $(\r,\Ff)$.

    \end{Rem}

    \begin{Prop}[{\cite[Proposition 4.3]{CRlatopologieduspectrereel}}]\label{Prop: fiber of spectral map}
        Let \func{\pi}{A}{B} be a ring morphism, and $(\rho,\Ff_\r) \in \spec{A}$. Then $\left(\spec{\pi}\right)^{-1}(\r,\Ff_\r)$ is homeomorphic to $\specf{\Ff_\r \otimes_A B}$.
    \end{Prop}

    We refer to \cite[Proposition 2.5]{CRlatopologieduspectrereel} for a description of the homeomorphism.
    
    \begin{Thm} \label{Thm: universal compact space over the real spectrum}
        Let $V\subset \Rr^n$ be an algebraic set and $W \subset \Rr^m$ a compact algebraic set. If $\func{\pi}{(V \times W)(\Rr)}{V(\Rr)}$ is the projection map, then the restriction to the closed points of the induced map $\spec{\pi}$ takes its values in $\speccl{W(\Rr)}$, that is 
        \[
            \speccl{\pi}: \speccl{(V\times W)(\Rr)}\rightarrow \speccl{V(\Rr)}.
        \]
        Moreover, it is continuous, surjective and the fiber of $(\r,\Kk_\r)\in \speccl{V(\Rr)}$ is homeomorphic to $\specarch{W(\Kk_\r)}$.
    \end{Thm}

    \begin{proof}
        Since $W(\Rr)$ is compact, the projection map is algebraic and proper. Thus, by Theorem \ref{Thm: proper algebraic map and closed points}, the map $\speccl{\pi}: \speccl{(V \times W)(\Rr)}\rightarrow \speccl{V(\Rr)}$ is well-defined, continuous and surjective.
        
        We show that the fiber of an element $(\r^1 ,\Kk_{\r^1}) \in \speccl{V(\Rr)}$ is homeomorphic to \specarch{W(\Kk_{\r^1})}. Consider $(\spec{\pi})^{-1}(\r^1,\Kk_{\r^1})$ endowed with the subspace topology induced by the spectral topology on \spec{(V \times W)(\Rr)}. By Proposition \ref{Prop: fiber of spectral map}, there exists a canonical homeomorphism
        \begin{align*}
            \left(\spec{\pi}\right)^{-1}(\r^1,\Kk_{\r^1}) & \cong \specf{\Kk_{\r^1} \otimes_{\Rr[V]} \Rr\left[V\times W\right]} \\
            &\cong \specf{\Kk_{\r^1} \otimes_{\Rr[V]} \Rr\left[V\right] \otimes_\Rr \Rr\left[W\right]} \\
            &\cong \spec{\Kk_{\r^1}\left[W\right]} = \spec{W\left(\Kk_{\r^1}\right)},
        \end{align*}
        where the second homeomorphism comes from \cite[Theorem 2.8.3. (iii)]{BCRrea}. We study the intersection of this preimage with \speccl{(V\times W)(\Rr)}.
        Consider $(\r,\Ff_\r) \in (\spec{\pi})^{-1}(\r^1,\Kk_{\r^1}) \cap \speccl{(V\times W)(\Rr)}$ and its decomposition $(\r,\Ff_\r) = (\r^1,\r^2,\Ff_\r)$ as in Remark \ref{Rem: decomposition of spectral morphism}.
        Consider the coordinate ring
        \[
            \Rr\left[V\times W\right] := \Rr[x_{1},\ldots, x_n]/I\left(V\times W\right),
        \]
        as in Definition \ref{Def: coordinate ring} and the element $v_\rho := (\r(x_1),\ldots, \r(x_n)) \in V(\Ff_\r)$ defined in Notation \ref{Notation element of VF}.
        By Lemma \ref{Lem: polynomial bound for proper maps}, there exist constants $c,d \in \Nn$ such that  
        \begin{align*}
          \fev \ v \in V(\Rr): \left|x_i(v)\right|&\leq c\left(1+N(\pi(v))^2\right)^d.
        \end{align*}
        So by the Transfer principle (Theorem \ref{Thm transfer principle}) 
        \begin{align*}
            \left|x_{i}(v_\r)\right| &\leq c(1+N\left(\pi_{\Ff_\r}(v_\r))^2\right)^d, 
        \end{align*}
        such that $\r(\Rr[V \times W])$ is Archimedean over $\r^1(\Rr[V]) \subset \Kk_{\r^1}$. In particular, $\Ff_\r$ is Archimedean over $\Kk_{\r^1}$ so that $(\r^2,\Ff_\r) \in \specarch{\Kk_{\r^1}[W]}$. 
        Hence 
        \[
            \left(\speccl{\pi}\right)^{-1}(\r^1,\Kk_{\r^1}) \subset \specarch{W(\Kk_{\r^1})}.
        \]
        
        We prove the remaining inclusion. Consider $(\r^2,\Ff) \in \specarch{W(\Kk_\r^1)}$. Then $\func{\r^2|_{\Kk_{\r^1}}}{\Kk_{\r^1}}{\Ff}$ is an ordered field morphism (Remark \ref{Rem: K is a subfield of F}). Since \Ff \ is Archimedean over $\Kk_{\r^1}$ which is Archimedean over $\r^1(\Rr[V])$
        \[
            \left(\r^2|_{\Kk_{\r^1}} \circ \r^1,\r^2,\Ff\right)\in \speccl{(V\times W)(\Rr)}.
        \]
        Moreover, $\spec{\pi}(\r^2|_{\Kk_{\r^1}} \circ \r^1,\r^2,\Ff) = (\r^2|_{\Kk_{\r^1}} \circ \r^1,\Ff) = (\r^1,\Kk_{\r^1})$ by Proposition \ref{Proposition: various defintion of the real spectrum} so that
        \[
            \left(\speccl{\pi}\right)^{-1}(\r^1,\Kk_{\r^1}) \cong \specarch{W(\Kk_{\r^1})}.
        \]
    \end{proof}
    In the setting of Theorem \ref{Thm: universal compact space over the real spectrum}, the fibers over $\speccl{V(\Rr)}$ exhibit good behavior under projection.
    Informally, we show that if a sequence of elements in $\speccl{V(\Rr)}$ converges to $(\r,\Kk_\r)$, then the corresponding fibers also converge to the fiber of $(\r,\Kk_\r)$. To describe this behavior, we use that \spec{\pi} is an open map, which is described in \cite[Theorem 6.3]{CRlatopologieduspectrereel}, though our setting provides a natural proof.

    \begin{Lem}[{\cite[Theorem 6.3]{CRlatopologieduspectrereel}}] \label{Lem: spec(f) is an open map}
        Let $\Ll \subset \Kk$ be real closed fields and $V\subset \Ll^n$, $W\subset \Ll^m$ algebraic sets.
        If \func{\pi}{(V\times W)(\Kk)}{V(\Kk)} is the projection map, then the lift $\func{\spec{\pi}}{\spec{(V\times W)(\Kk)}}{\spec{V(\Kk)}}$ is open. 
    \end{Lem}

    \begin{proof}
        Since for two open sets $U_1,U_2 \subset \spec{(V\times W)(\Kk)}$, the equality 
        \[
            \spec{\pi}(U_1 \cup U_2)=\spec{\pi}(U_1) \cup \spec{\pi}(U_2)
        \]
        holds, it is enough to show that the image of a basis element 
        \[
            \tilde{U}(f_1,\ldots,f_p) \subset \spec{(V\times W)(\Kk)} 
        \]
        is open for any $f_i \in \Kk[V \times W]$. 
        Consider the open subset $U(f_1,\ldots,f_p) = \tilde{U}(f_1,\ldots,f_p) \cap (V \times W)(\Kk)$, see the fifth item of Proposition \ref{Prop: constructible sets}. Since the projection map is open, $\pi(U(f_1,\ldots,f_p))$ is an open subset of $V(\Kk)$. In particular, by the Finiteness Theorem \cite[Theorem 2.7.2]{BCRrea}, there exists $g_1^i, \ldots, g_{q_i}^i \in \Kk[V]$ with 
        \[
            \pi(U(f_1,\ldots,f_p))= \bigcup_i U\left(g_1^i, \ldots, g_{q_i}^i\right).
        \]
        We show that $\spec{\pi}(\tilde{U}(f_1,\ldots,f_p))= \cup_i \widetilde{U(g_1^i, \ldots, g_{q_i}^i)}$ which is open by Proposition \ref{Prop: constructible sets}.
        As in Definition \ref{Def: coordinate ring}, consider the coordinate rings 
        \[
            \Kk\left[V\right] = \Kk[x_{1},\ldots, x_n]/I\left(V\right) \text{ and } \Kk[W] = \Kk[y_1,\ldots,y_m]/I(W).
        \]
        Let $(\r,\Ff_\r) \in \tilde{U}(f_1,\ldots,f_p)$ and, as in Notation \ref{Notation element of VF}, its associated element \[
            z_\r = (\r(x_1),\ldots,\r(x_n),\r(y_1),\ldots,\r(y_m)) \subset (V\times W)(\Ff_\r).
        \]
        By definition $z_\r \in U(f_1,\ldots,f_p)(\Ff_\r):= \setrel{v\in V(\Ff_\r)}{f_i(v)>0 \quad \forall i}$.
        By \cite[Proposition 5.2.1]{BCRrea} 
        \[
            \pi_{\Ff_{\r}}(z_\r) \in \left(\bigcup_i U\left(g_1^i, \ldots, g_{q_i}^i\right)\right)(\Ff_\r) = \bigcup_i U\left(g_1^i, \ldots, g_{q_i}^i\right)(\Ff_\r),
        \]
        and by the third item of Proposition \ref{Prop: constructible sets}
        \[
            \bigcup_i U\left(g_1^i, \ldots, g_{q_i}^i\right)(\Ff_\r) \subset \bigcup_i \widetilde{U\left(g_1^i, \ldots, g_{q_i}^i\right)}.
        \]
        Hence the following inclusion holds
        \[
            \spec{\pi}\left(\tilde{U}(f_1,\ldots,f_p)\right) \subset \bigcup_i \widetilde{U\left(g_1^i, \ldots, g_{q_i}^i\right)}.
        \]
        Moreover, for every $w\in \cup_i U(g_1^i, \ldots, g_{q_i}^i)$ there exists $v \in U(f_1,\ldots,f_p)$ such that $\pi(v)=w$. So, by \cite[Proposition 5.2.1]{BCRrea}, for every $w_\Ff \in \cup_i U(g_1^i, \ldots, g_{q_i}^i)(\Ff)$ there exists $v_\Ff \in U(f_1,\ldots,f_p)(\Ff)$ with $\pi(v_\Ff)=w_\Ff$. So
        \[
            \spec{\pi}\left(\tilde{U}(f_1,\ldots,f_p)\right) = \bigcup_i \widetilde{U\left(g_1^i, \ldots, g_{q_i}^i\right)},
        \]
        which is an open subset of \spec{V(\Kk)}.
    \end{proof}

    In our context, this lemma has an analogue at the level of closed points.

    \begin{Thm}\label{Thm: the projection map is open}
        Let $\Ll \subset \Kk$ be real closed fields, $V \subset \Ll^n$, $W\subset \Ll^m$ algebraic sets, and \func{\pi}{(V\times W)(\Kk)}{V(\Kk)} the projection map. If there exist open sets $U_k(\Kk)\subset (V\times W)(\Kk)$ for $k\in \Nn$ such that the restriction of \spec{\pi} to 
        \[
            E:=\bigcup_{k\in \Nn} \widetilde{U_k(\Kk)}\cap \speccl{(V\times W)(\Kk)}
        \] 
        has value in the closed points and is surjective, that is
        \[
            \func{\spec{\pi}|_E}{E}{\speccl{V(\Kk)}},
        \]
        then $\spec{\pi}|_E$ is also open.
    \end{Thm}
    
    \begin{proof}
        Consider an open set $U^{E} \subset E$, such that there exists an open set $U \subset \spec{(V\times W)(\Kk)}$ with $U^{E} = U\cap E$. 
        Since $\cup_{k\in\Nn}\widetilde{U_{k}(\Kk)} \subset \spec{(V\times W)(\Kk)}$ is open by Proposition \ref{Prop: constructible sets}, we suppose $U \subset \cup_{k\in\Nn}\widetilde{U_{k}(\Kk)}$.
        It holds
        \[
            \speccl{\pi}\left(U^E\right) \subset \spec{\pi}(U) \cap \speccl{V(\Kk)},
        \]
        and we prove the reverse inclusion.
        Consider an element 
        \[
            \left(\r^1,\Kk_{\r^1}\right)\in \spec{\pi}(U) \cap \speccl{V(\Kk)}
        \]
        so that $(\spec{\pi})^{-1}(\r^1,\Kk_{\r^1}) \cap U$ is a non-empty open subset of $(\spec{\pi})^{-1}(\r^1,\Kk_{\r^1})$ for the subspace topology. 
        By Proposition \ref{Prop: fiber of spectral map}, there exists a canonical homeomorphism
        \begin{align*}
            \left(\spec{\pi}\right)^{-1}(\r^1,\Kk_{\r^1}) & \cong \specf{\Kk_{\r^1} \otimes_{\Kk[V]} \Kk\left[V\times W\right]} \\
            &\cong \specf{\Kk_{\r^1} \otimes_{\Kk[V]} \Kk\left[V\right] \otimes_\Kk \Kk\left[W\right]} \\
            &\cong \specff{\Kk_{\r^1}\left[W\right]} = \spec{W\left(\Kk_{\r^1}\right)},
        \end{align*}
        where the second homeomorphism comes from \cite[Theorem 2.8.3.(iii)]{BCRrea}. Denote by $\func{\varphi}{(\spec{\pi})^{-1}(\r^1,\Kk_{\r^1})}{\spec{W(\Kk_{\r^1})}}$ this homeomorphism. Then
        \[
            \varphi\left(\left(\spec{\pi}\right)^{-1}\left(\r^1,\Kk_{\r^1}\right) \cap U\right)
        \]
        is a non-empty open subset of $\spec{W(\Kk_{\r^1})}$. By Theorem \ref{Thm: map from the algebraic set to the archimedean spectrum}, this open set contains a closed point 
        \[
            (\r^2,\Kk_{\r^2}) \in \varphi\left(\left(\spec{\pi}\right)^{-1}\left(\r^1,\Kk_{\r^1}\right) \cap U\right)\cap \speccl{W\left(\Kk_{\r^1}\right)}.
        \]
        Since $\varphi$ is continuous, $\varphi^{-1}(\r^2,\Kk_{\r^2})=(\psi,\Ff)$ is a closed subset of $(\spec{\pi})^{-1}(\r^1,\Kk_{\r^1})$ which is in $U$. In particular, there exists a closed set $A\subset \spec{(V\times W)(\Kk)}$ with 
        \[
            (\psi,\Ff)=A \cap \left(\spec{\pi}\right)^{-1}(\r^1,\Kk_{\r^1}).
        \]
        Since $(\r^1,\Kk_{\r^1})\in\speccl{V(\Kk)}$ and $\spec{\pi}$ is continuous, the fiber $(\spec{\pi})^{-1}(\r^1,\Kk_{\r^1})$ is closed in $\spec{(V\times W)(\Kk)}$. Hence $(\psi,\Ff) = A\cap (\spec{\pi})^{-1}(\r^1,\Kk_{\r^1})$ is closed and
        \begin{align*}
            (\psi,\Ff) &\in  \speccl{(V\times W)(\Kk)} \cap U \cap \left(\spec{\pi}\right)^{-1}(\r^1,\Kk_{\r^1}) \\
            & =  \speccl{(V\times W)(\Kk)} \cap \bigcup_{k\in \Nn} \widetilde{U_k(\Kk)} \cap U \cap \left(\spec{\pi}\right)^{-1}(\r^1,\Kk_{\r^1}) \\
            & = E \cap U \cap \left(\spec{\pi}\right)^{-1}(\r^1,\Kk_{\r^1}).
        \end{align*}
        Hence
        \[
            \speccl{\pi}\left(U^E\right) = \spec{\pi}(U) \cap \speccl{V(\Kk)},
        \]
        which is open by Lemma \ref{Lem: spec(f) is an open map}.
    \end{proof}

    We now present the convergence of the fibers in \speccl{(V \times W)(\Rr)} in a fully constructive manner.
    
    \begin{Thm} \label{Thm: convergence of the fibers in E general}
        Consider the same setting as in Theorem \ref{Thm: the projection map is open}. Let $(\r^1_n,\Kk_{\r^1_n}) \subset \speccl{V(\Rr)}$ be a sequence converging to an element $(\r^1,\Kk_{\r^1}) \in \speccl{V(\Rr)}$. For every element of the fiber $(\r^1 , \r^2,\Ff) \in (\speccl{\pi})^{-1}(\r^1,\Kk_{\r^1})$ and every open set $U$ in $E$ around this element, \tes \ $M\in \Nn$ and 
        \[
            (\r^1_M,\r^2_M,\Ff_M) \in \left(\speccl{\pi}\right)^{-1}(\r^1_M,\Kk_{\r^1_M})\cap U.
        \]
    \end{Thm}
    
    \begin{proof}
        As in the statement, consider $(\r^1 , \r^2,\Ff)\in (\speccl{\pi})^{-1}(\r^1,\Kk_{\r^1})$ and $U \subset E$ an open set around it. 
        By Theorem \ref{Thm: the projection map is open}, the map $\spec{\pi}|_E$ is open. Thus 
        \[
            U' :=\spec{\pi}|_E(U)
        \]
        is open in $\speccl{\mathcal{M}_\G(\Rr)}$ and contains~$(\r^1,\Kk_{\r^1})$.
        Since the sequence $(\r^1_n,\Kk_{\r^1_n})$ converges towards $(\r^1,\Kk_{\r^1})$, there exists $M\in \Nn$ so that $(\r^1_M,\Kk_{\r^1_M}) \in U'$. 
        Moreover, since $U' = \spec{\pi}|_E(U)$, \tes \ $(\r,\Ff_\r) \in U$ \st  
        \[
            \spec{\pi}|_E\left((\r,\Ff_\r)\right)=\left(\r^1_M,\Kk_{\r^1_M}\right).
        \]
        Hence
        \[
            (\r,\Ff_\r) \in \left(\spec{\pi}\right)^{-1}(\r^1_M,\Kk_{\r^1_M})\cap U
        \]
        is the desired element in the fiber.      
    \end{proof}

        We apply these results to the case $V(\Kk) = \mathcal{M}_\G(\Rr)$ and $W(\Kk) = \projspn(\Rr)$.
        There exists a natural continuous action of \G \ on $\mathcal{M}_\G(\Rr) \times \projspn(\Rr)$ given by 
        \[
            \g.\left(\p,x\right)=\left(\p, \p(\g)x \right).
        \]

        \begin{Prop}\label{Rem: Gamma action on the spectrum of projective space}
            The \G-action on $\mathcal{M}_\G(\Rr) \times \projspn(\Rr)$ extends to an action by homeomorphisms on $\specf{\mathcal{M}_\G(\Rr) \times \projspn(\Rr)}$ preserving $\specclf{\mathcal{M}_\G(\Rr) \times \projspn(\Rr)}$. 
        \end{Prop}

        \begin{proof}
            Since \slnr \ acts on \projspnr \ semialgebraically, the graph of the action of one element of \G \ on $\mathcal{M}_\G(\Rr) \times \projspn(\Rr)$ is semialgebraic. Thus by \cite[Proposition 7.2.8]{BCRrea}, the action extends canonically to a homeomorphism of $\specf{\mathcal{M}_\G(\Rr) \times \projspn(\Rr)}$ preserving $\specclf{\mathcal{M}_\G(\Rr) \times \projspn(\Rr)}$. 
        \end{proof}

    We now prove that \specclf{\mathcal{M}_\G(\Rr) \times \projspn(\Rr)} forms a universal geometric space over \rsp{\mathcal{M}_\G(\Rr)}. Informally, we show that \specclf{\mathcal{M}_\G(\Rr) \times \projspn(\Rr)} behaves similarly to the product $\speccl{\mathcal{M}_\G(\Rr)}\times \speccl{\projspn(\Rr)}$, even though these spaces are not homeomorphic when endowed with the spectral topology.
    
    \begin{Thm} \label{Thm: universal projective space over the real spectrum}
        If $\func{\pi}{\mathcal{M}_\G(\Rr) \times \projspn(\Rr)}{ \mathcal{M}_\G(\Rr)}$ is the projection map, then the restriction to the closed points of the induced map $\spec{\pi}$ takes its values in $\speccl{\mathcal{M}_\G(\Rr)}$, that is 
        \[
            \speccl{\pi}: \specclff{\mathcal{M}_\G(\Rr) \times \projspn(\Rr)}\rightarrow \speccl{\mathcal{M}_\G(\Rr)}.
        \]
        Moreover, it is continuous, surjective, \G -equivariant, and the fiber of $(\r,\Kk_\r)\in \speccl{\mathcal{M}_\G(\Rr)}$ is homeomorphic to $\specarch{\projspn(\Kk_\r)}$.
    \end{Thm}
    
    \begin{proof}
        The projective space is described as the orthogonale projections onto subspaces of dimension $1$ of $\Rr^n$ \cite[Theorem 3.4.4]{BCRrea}
        \[
            \projspnr := \setrel{A \in M_{n\times n}(\Rr)}{A^T = A, A^2 = A, \mathrm{tr}(A) = 1}.
        \] 
        It is a compact algebraic set so that by Theorem \ref{Thm: universal compact space over the real spectrum}, the map 
        \[
            \speccl{\pi}: \specclff{\mathcal{M}_\G(\Rr) \times \projspn(\Rr)}\rightarrow \speccl{\mathcal{M}_\G(\Rr)}
        \]
        is continuous, surjective, and the fiber of $(\r,\Kk_\r)\in \speccl{\mathcal{M}_\G(\Rr)}$ is homeomorphic to $\specarch{\projspn(\Kk_\r)}$.

        We show that \speccl{\pi} is \G-equivariant for the \G -action on \specclf{\mathcal{M}_\G(\Rr)\times \projspnr} described in Proposition \ref{Rem: Gamma action on the spectrum of projective space}, and the trivial action of \G \ on  \speccl{\mathcal{M}_\G(\Rr)}. Let $(\r,\Ff_\r)\in \specclf{\mathcal{M}_\G(\Rr)\times \projspnr}$. 
        By Remark \ref{Rem: decomposition of spectral morphism}, there exist $(\r^1,\Ff_\r)\in \spec{\mathcal{M}_\G(\Rr)}$ and $(\r^2,\Ff_\r)\in \spec{\projspnr}$ such that 
        \[
            \forall h \in \mathbb{R}\left[\mathcal{M}_\G \times \projspn\right] \quad \r(h)=\sum\r^1(h_i^1)\r^2(h_i^2), 
        \]
        where $ \beth^{-1}(h) = \sum h_i^1\otimes h_i^2 \in \Rr[\mathcal{M}_\G]\otimes_\Rr \Rr[\projspn].$
        For $\g\in \G$ and $h^1 \in \Rr[\mathcal{M}_\G]$, the following equalities hold:
        \begin{align*}
            \speccl{\pi}(\g\r)\left(h^1\right) & =\speccl{\pi}\left(\r^1,\r^2 \circ \g\right)\left(h^1\right) \\
             & = \rho^1\left(h^1\right) \\
             & =\g \r^1\left(h^1\right) = \g \speccl{\pi}(\r)\left(h^1\right).
        \end{align*}
        Thus, \speccl{\pi} is \G-equivariant, as desired.
    \end{proof}

    \begin{Rem}
        The map \speccl{\pi} is \G-equivariant, so its fibers are \G-invariant. 
        Thus every $(\r^1,\Kk_{\r^1})\in \speccl{\mathcal{M}_\G(\Rr)}$ induces a \G-action on \specarch{\projspn(\Kk_{\r^1})} such that, for $\g \in \G$, $(\r^2,\Kk_{\r^2})\in (\speccl{\pi})^{-1}(\r^1,\Kk_{\r^1})$, and $h\in \Kk_{\r^1}[\projspn]$ we have:
        \[
            \g \r^2(h) = \r^2(\r^1(\g)h).
        \]
    \end{Rem}

    Finally, we can explicitly characterize the well organization of the fibers of \speccl{\pi}.

    \begin{Cor}\label{Cor: convergence of the fibers in the universal projective space}
        Let $(\r^1_n,\Kk_{\r^1_n}) \subset \speccl{\mathcal{M}_\G(\Rr)}$ be a sequence converging to an element $(\r^1,\Kk_{\r^1}) \in \speccl{\mathcal{M}_\G(\Rr)}$. For every element of the fiber $(\r^1 , \r^2,\Ff) \in (\speccl{\pi})^{-1}(\r^1,\Kk_{\r^1})$ and every open set in \specclf{\mathcal{M}_\G(\Rr)\times \projspn(\Rr)} around this element, \tes \ $M\in \Nn$ and $(\r^1_M,\r^2_M,\Ff_M)$ in $(\speccl{\pi})^{-1}(\r^1_M,\Kk_{\r^1_M})\cap U$. In particular, $(\r^2_M,\Ff_M)$ is in \specarch{\projspnk}.
    \end{Cor}

    \begin{proof}
        Consider the open subsets $U_k(\Rr):=\mathcal{M}_\G(\Rr)\times \projspnr$ of $\mathcal{M}_\G(\Rr)\times \projspnr$ for every $k\in\Nn$ such that with the notation from Theorem \ref{Thm: the projection map is open}
        \[
            E = \specclff{\mathcal{M}_\G(\Rr)\times \projspnr}. 
        \]
        Then, $\spec{\pi}|_E = \speccl{\pi}$ and by Theorem \ref{Thm: universal projective space over the real spectrum}, the map $\speccl{\pi}$ verifies all the conditions from Theorem \ref{Thm: convergence of the fibers in E general}. Thus, the statement is a consequence of the more general Theorem \ref{Thm: convergence of the fibers in E general}.
    \end{proof}

    \begin{Rem}
        By combining $\specarch{\projspr}\cong\projspr$ with the accessibility theorem from \cite[Corollary 3.9]{BIPPthereal}, we also obtain an accessibility theorem stating that all \G-actions on \specarch{\projspf} arise from \G-actions on \projspr.
    \end{Rem}

    We constructed a universal projective space over \rsp{\mathcal{M}_\G(\Rr)}. This provides a framework for studying the actions induced by elements of \rsp{\mathcal{M}_\G(\Rr)} on \projspnk \ within \specarch{\projspnk}. In the following subsections, we extend this approach to analyze the induced actions on the symmetric space \sym. Specifically, we construct a universal symmetric space over \rsp{\mathcal{M}_\G(\Rr)}.

%% file: displacement_function.tex
\subsection{Displacement of an action and its associated geometric space} \label{Subsection: Displacement of an action and its associated geometric space}
    
        In this subsection, we construct a \G -invariant open subset $E$ of \specclf{\mathcal{M}_\G(\Rr)\times \widehat{\symr}} which is a countable union of compact subsets of \specclf{\mathcal{M}_\G(\Rr)\times \widehat{\symr}} and that contains $\mathcal{M}_\G(\Rr)\times \widehat{\symr}$. Moreover, as we see in the next subsection, $E$ surjects continuously onto \speccl{\mathcal{M}_\G(\Rr)}. 

        \begin{Rem}
            There exists a natural action of \G \ on $\mathcal{M}_\G(\Rr)\times \widehat{\symr}$ given by 
            \[
                \g.\left(\p,(A,t_1,\ldots ,t_{n-1})\right)=\left(\p, \p(\g).(A,t_1,\ldots ,t_{n-1}) \right).
            \]
        \end{Rem}

        To encode the \G-action on $\widehat{\symr}$ induced by an element in $\mathcal{M}_\G(\Rr)$ and define a universal space over \rsp{\mathcal{M}_\G(\Rr)}, we look at elements of the symmetric space within specified balls defined by the displacement function of the representation. Consider $F=\set{\g_1,\ldots,\g_s}$ a finite generating set of \G \ and recall that, as in Subsection \ref{Subsection minimal vectors and character varities}, $M_{n\times n}(\Rr)^F$ is equipped with a scalar product 
        \[
            \psc{(A_1,\ldots, A_s)}{(B_1,\ldots, B_s)} := \sum_{i=1}^{s}\mathrm{tr}\left(A_i^{T} B_i\right),
        \]
        for every $A_1,\ldots, A_s, B_1,\ldots, B_s \in M_{n\times n}(\Rr)$. Then we define the norm
        \[
            \DefMap{\eta}{\Homt{\G}{\slnr}}{\Rr_{\geq 0}}{\r}{\psc{(\r(\g_1),\ldots, \r(\g_s))}{(\r(\g_1),\ldots, \r(\g_s))}}
        \]
        In particular, $\eta$ is semialgebraically continuous. 
        Let $d_\d$ be the semialgebraically continuous Cartan multiplicative distance on \symr \ given in Proposition~\ref{Properties Cartan distance}, and $\hat{d_\d}$ its extension to $\widehat{\symr}$.
        Consider \fev \ $k\in \Nn$
        \[
            \mathcal{U}_{k}(\Rr):=\setrelfrac{(\r , (A , t))\in \mathcal{M}_\G(\Rr)\times \widehat{\symr}}{\hat{d_\d}(\mathrm{Id},(A ,t))< \eta(\r)^{k}},
        \] 
        which is open and semialgebraic in $\mathcal{M}_\G(\Rr)\times \widehat{\symr}$ as both $\hat{d_\d}$ and $\eta$ are, see Proposition \ref{Properties Cartan distance}. Define the open subset of $\specclf{\mathcal{M}_\G(\Rr)\times \widehat{\symr}}$ given by
        \[
            E:= \bigcup_{k\in\Nn}\widetilde{\mathcal{U}_{k}(\Rr)}\cap \specclff{\mathcal{M}_\G(\Rr)\times \widehat{\symr}},
        \]
        where $\widetilde{\mathcal{U}_{k}(\Rr)}$ is the constructible set associated to $\mathcal{U}_{k}(\Rr)$, see Definition \ref{Def:constructible sets}. Denote by $\norm{\cdot}_{F}$ the word length with respect to the generating set~$F$ of \G.
    
        \begin{Lem}
            \Fev \ $(\r,A) \in \mathcal{M}_\G(\Rr)\times \symr$ and $\g \in \G$, it holds
            \[
                d_\d(\mathrm{Id},\r(\g).A) \leq \eta(\r)^{\frac{n}{2} \norm{\g}_F} d_\d(\mathrm{Id},A).
            \]
        \end{Lem}
    
        \begin{proof}
            Using the first item of Proposition \ref{Properties Cartan distance} and the \slnr -invariance of the Cartan's multiplicative distance we obtain for every $\g \in \G$
            \begin{equation}\label{crforwordelement}
                d_\d(\mathrm{Id},\r(\g).A)=d_\d(\mathrm{Id},\r(\g _1\cdots \g _\ell).A) \leq \left(\prod_{j=1}^{\ell} d_\d(\mathrm{Id},\r(\g _i). \mathrm{Id})\right)d_\d(\mathrm{Id},A),
            \end{equation}
            where $\norm{\g}_{F}=\ell$, $\g_i \in {F}$ \fev \ $1\leq i \leq \ell$ and $\g = \g_1\cdots \g_\ell$.
            We now bound the value of $d_\d(\mathrm{Id},\r(\g _i). \mathrm{Id})$ by powers of $\eta(\r)$. 
            As seen in Subsection \ref{Subsection: Symmetric space associated to slnr}, any $g\in \slnr$ has a Cartan decomposition $g=kc(g)k'$ where $k,k' \in \mathrm{SO}(n,\Rr)$ and 
            \[
                c(g) = 
                \begin{pmatrix}
                \l_{1} & & \\
                & \ddots & \\
                & & \l_{n}    
                \end{pmatrix}
                \text{ where the } \l_i \text{ are ordered so that } \l_n\leq \cdots \leq \l_1.
            \]
            Then $g^Tg=k'^{-1}c(g)^2k'$ so that $\mathrm{tr}(g^Tg)=\l_1^2 + \cdots + \l_n^2$ and $d_\d(\mathrm{Id},g.\mathrm{Id})=\l_1 \l_n^{-1}$. 
            Since~$\mathrm{det}(g)$ is the product of the $\l_i$ and is equal to one, we obtain $\l_1 \leq d_\d(\mathrm{Id},g.\mathrm{Id})\leq \l_1^n$ and $\l_1^2\leq \mathrm{tr}(g^Tg)\leq n\l_1^2$. Hence
            \begin{equation}\label{relation between trace and cartan distance}
                d_\d(\mathrm{Id},g.\mathrm{Id})\leq \mathrm{tr}\left(g^Tg\right)^\frac{n}{2}.
            \end{equation}
            Furthermore, from the inequalities (\ref{crforwordelement}) and (\ref{relation between trace and cartan distance}), we obtain 
            \begin{align*}
                d_\d(\mathrm{Id},\r(\g).A) &\leq \left(\prod_{j=1}^{\ell} d_\d(\mathrm{Id},\r(\g _i). \mathrm{Id})\right)d_\d(\mathrm{Id},A) \\
                & \leq \left(\prod_{j=1}^{\ell} \mathrm{tr}\left(\r(\g_i)^T\r(\g_i)\right)\right)^\frac{n}{2}d_\d(\mathrm{Id},A) \\
                & \leq  \eta(\r)^{\frac{n}{2}\norm{\g}_{F}}d_\d(\mathrm{Id},A). 
            \end{align*}
            Hence we obtain the desired inequality.
        \end{proof}
    
        \begin{Cor}\label{Cor: Invariance of E}
            \Fev \ $\g \in \G$ and $k \in \Nn$, we have the inclusion $\g . \, \mathcal{U}_k(\Rr) \subset \mathcal{U}_{k+\frac{n}{2}\norm{\g}_{F}}(\Rr)$.
        \end{Cor}

        \begin{Cor}\label{Cor: Gamma action on E}
            The \G-action on $\mathcal{M}_\G(\Rr) \times \widehat{\symr}$ extends to an action by homeomorphisms on $\specf{\mathcal{M}_\G(\Rr) \times \widehat{\symr}}$ preserving $E$. 
        \end{Cor}

        \begin{proof}
            Since \slnr \ acts on $\widehat{\symr}$ semialgebraically, the graph of the action of one element of \G \ on $\mathcal{M}_\G(\Rr) \times \widehat{\symr}$ is semialgebraic. Therefore, by \cite[Proposition 7.2.8]{BCRrea}, the action extends canonically to a homeomorphism of $\specf{\mathcal{M}_\G(\Rr) \times \widehat{\symr}}$, preserving $\specclf{\mathcal{M}_\G(\Rr) \times \widehat{\symr}}$. Additionally, by Corollary \ref{Cor: Invariance of E}, this homeomorphism preserves $E$. 
        \end{proof}

        Since the sets $\widetilde{\mathcal{U}_k(\Rr)}$ are open subsets of \specf{\mathcal{M}_\G(\Rr)\times \widehat{\symr}} for any $k\in \Nn$, the above defined set $E$ is a locally compact open subset of $\specclf{\mathcal{M}_\G(\Rr)\times \widehat{\symr}}$ which contains $\mathcal{M}_\G(\Rr)\times \widehat{\symr}$ and on which \G \ acts by homeomorphisms.
        
        \begin{Rem}
            As in Theorem \ref{Thm: The Archimedean spectrum is a countable union of open compact sets}, if we replace $\mathcal{U}_{k}(\Rr)$ in the definition of $E$ by
            \[
                \mathcal{V}_{k}:=\setrelfrac{(\r , (A , t))\in \mathcal{M}_\G(\Rr)\times \widehat{\symr}}{\hat{d_\d}(\mathrm{Id},(A,t))\leq \eta(\r)^{k}},
            \] 
            which are closed and semialgebraic in $\mathcal{M}_\G(\Rr)\times \widehat{\symr}$. Then 
            \[
                E:= \bigcup_{k\in\Nn}\widetilde{\mathcal{U}_{k}(\Rr)}\cap \specclff{\mathcal{M}_\G(\Rr)\times \widehat{\symr}} = \bigcup_{k\in\Nn}\widetilde{\mathcal{V}_{k}(\Rr)}\cap \specclff{\mathcal{M}_\G(\Rr)\times \widehat{\symr}}.
            \]
            so that $E$ is also a countable union of compact subsets of \specclf{\mathcal{M}_\G(\Rr)\times \widehat{\symr}}. Thus it is $\s$-compact.
        \end{Rem}

        Using constructible subsets of the real spectrum, we constructed a topological space $E$ on which \G \ acts by homeomorphisms. 
        In the following subsection, we study the lift of the projection map $\pi : \mathcal{M}_\G(\Rr) \times \widehat{\symr} \rightarrow \mathcal{M}_\G(\Rr)$ to the real spectrum, particularly focusing on its restriction to $E$. This results in a universal symmetric space over \rsp{\mathcal{M}_\G(\Rr)}.

%% file: The_universal_symmetric_space.tex
\subsection{The universal symmetric space}\label{Subsection: The projection map and the Archimedean spectrum of the symmetric space}
    In this subsection, we construct a universal projective space over \rsp{\mathcal{M}_\G(\Rr)} that encodes the degeneracies of \G -actions on $ \widehat{\symr}$ induced by representations in $\mathcal{M}_\G(\Rr)$.
    In particular, the Archimedean spectrum allows us to describe the fibers of the projection $\pi: \mathcal{M}_\G \times \widehat{\sym} \rightarrow \mathcal{M}_\G$ at the level of the real spectrum, especially its restriction to $E$. Moreover, we prove that the fibers are well organized over the target space and establish convergence results about them. 
    To do this, we need the following two linear algebra results.

    \begin{Lem} \label{Lem: Linear algebra to bound matrix values}
        Let $A\in M_{n\times n}(\Rr)$ be a symmetric matrix with eigenvalues $\l_n \leq \cdots \leq \l_1$.
        \begin{enumerate}
            \item \label{item: bound of the entries of the matrix} If $A$ is positive semidefinite, then for every $1\leq i,j\leq n$ it holds
            \[
                |A_{i,j}|=|\psc{Ae_i}{e_j}|\leq |\l_1|.
            \]
            \item \label{item: bound of the principal minor of the matrix} For every $1\leq \ell \leq n$, we have $\l_n^\ell \leq \mathrm{det}(A[\ell]) $, where $\mathrm{det}(A[\ell])$ is the $\ell$-principal minor of $A$.
        \end{enumerate}
    \end{Lem}

    \begin{proof}
        For the first item, let $1 \leq i,j \leq n$ such that
        \[    
            |A_{i,j}| = |\psc{Ae_i}{e_j}| \leq \max_{\norm{x}=1,\norm{y}=1} |\psc{Ax}{y}|.
        \]
        Since $A$ is positive semidefinite, by the Rayleigh--Ritz Theorem (see for example \cite[Theorem 4.2.2]{HJmatrix})
        \[
            0 < \max_{\norm{x}=1,\norm{y}=1} |\psc{Ax}{y}| \leq \l_1.
        \]
        Hence $|A_{i,j}|\leq \l_1$.
        The second item is a consequence of the Cauchy's interlacing theorem \cite[Theorem 4.3.17]{HJmatrix}. Indeed, if the eigenvalues of $A[\ell]$ are $\mu_\ell\leq \cdots \leq \mu_1$, then by the Cauchy's interlacing theorem $\l_n \leq \mu_i$ for every $1\leq i\leq n$. Hence 
        \[
             \l_n^\ell \leq \prod_{i=1}^\ell \mu_i = \mathrm{det}(A[\ell]).
        \]
    \end{proof}

    Recall that there exists a multiplicative Cartan pseudo-distance $\hat{d_\d}$ defined on $\widehat{\sym}$ (Proposition~\ref{Properties Cartan distance}) and that $\eta: \p \mapsto \sum_{i=1}^{s}\mathrm{tr}(\p(\g_i)^{T} \p(\g_i))$ defines a scalar product on the representation space $\mathrm{Hom}(\G ,\slnr)$.
    Moreover, these two maps are semialgebraically continuous, which allows us to extend them using the Transfer principle (Theorem \ref{Thm transfer principle}) to a \Kk -multiplicative pseudo-distance on $\widehat{\symk}$ and a scalar product on $\mathrm{Hom}(\G ,\slnk)$ for any real closed field \Kk. 
    This allows us to prove the following theorem using similar techniques as in Subsection \ref{Subsection: The universal projective space}

    \begin{Thm}\label{Thm: universal symmetric space over the real spectrum}
        If $\func{\pi}{\mathcal{M}_\G(\Rr) \times \widehat{\symr}}{ \mathcal{M}_\G(\Rr)}$ is the projection map, then the restriction to $E$ of the induced map $\spec{\pi}$ takes its values in $\speccl{\mathcal{M}_\G(\Rr)}$, that is
        \[
            \spec{\pi}|_E: E \rightarrow \specclf{\mathcal{M}_\G(\Rr)}.
        \]
        Moreover, it is continuous, surjective, \G -equivariant, and the fiber of $(\r,\Kk_\r)\in \speccl{\mathcal{M}_\G(\Rr)}$ is homeomorphic to $\specarch{\widehat{\symkr}}$.
    \end{Thm}
    
    \begin{proof}
        As in the notation of the third item of Proposition \ref{Proposition: various defintion of the real spectrum}, consider $(\r,\Ff_\r) \in  \widetilde{\mathcal{U}_k(\Rr)} \cap \specclf{\mathcal{M}_\G(\Rr)\times \widehat{\symr}}$ for some~$k\in \Nn$, where $\Ff_\r$ is the real closure of the field of fractions of $\r(\Rr[\mathcal{M}_\G\times \widehat{\sym}])$. 
        We first prove that $\Ff_\r$ is Archimedean over $\r(\Rr[\mathcal{M}_\G])$. 
        For every $1\leq i,j \leq n$, $1\leq \ell \leq n-1$ and~$1\leq m\leq sn^2$, consider the coordinates $x_{i,j}, y_\ell$ and $z_m$ and the coordinate rings as in Definition \ref{Def: coordinate ring}
        \[
            \Rr\left[\widehat{\sym}\right] = \Rr[x_{i,j},y_\ell]/I\left(\widehat{\sym}\right) \text{ and } \Rr[\mathcal{M}_\G] = \Rr[z_m]/I(\mathcal{M}_\G).
        \] 
        Consider also the element $w_\rho = (\r(z_m),\rho(x_{i,j}),\rho(y_\ell)) \in \mathcal{M}_\G(\Ff_\r)\times \widehat{\symfr}$ defined in Notation \ref{Notation element of VF}.
        With the notations from Proposition~\ref{Properties Cartan distance}, for any $(A,t_1,\ldots,t_{n-1}) \in \widehat{\symr}$ where $0 < \l_n \leq \cdots \leq \l_1$ are the eigenvalues of $A$ 
        \[
            d_\d(\mathrm{Id},A)=N(\mathrm{diag}(\l_1,\ldots,\l_n))=\l_1/\l_n.
        \]
        Thus, using Lemma~\ref{Lem: Linear algebra to bound matrix values}, we obtain
        \begin{align*}
          \left|x_{i,j}(A)\right|&=|\psc{Ae_i}{e_j}|\leq \l_1 \leq \frac{\l_1}{\l_n} =d_\d(\mathrm{Id},A), \\
          t_\ell^2 &=\frac{1}{\mathrm{det}\left(A[\ell]\right)}\leq \l_n^{-\ell}\leq \left(\frac{\l_1}{\l_n}\right)^\ell =d_\d(\mathrm{Id},A)^\ell.
        \end{align*}     
        Using that $(\r,\Ff_\r)$ is an element of~$\widetilde{\mathcal{U}_k(\Rr)}$, the $\Ff_\r$-extensions of $\hat{d_\d}$ and $\eta$ verify 
        \[
            \hat{d_\d}\left(\mathrm{Id},x_{i,j}(w_\r),y_\ell(w_\r)\right) < \eta\left(z_m(w_\r)\right)^k.
        \]        
        Hence, by the Transfer principle (Theorem \ref{Thm transfer principle})
        \begin{align}
            \left|x_{i,j}(w_\r)\right| &\leq \hat{d_\d}\left(\mathrm{Id},x_{i,j}(w_\r),y_\ell(w_\r)\right) < \eta\left(z_m(w_\r)\right)^k, \label{Equation: Archimedean equation for xij}\\ 
            y_\ell(w_\r)^2 & \leq \hat{d_\d}\left(\mathrm{Id},x_{i,j}(w_\r),y_\ell(w_\r)\right)^\ell < \eta\left(z_m(w_\r)\right)^{k\ell}, \label{Equation: Archimedean equation for yell}
        \end{align}
        so that $x_{i,j}(w_\r)$ and $y_\ell(w_\r)$ are bounded by $\eta(z_m(w_\r))^{k\ell}$ for every $1\leq i,j\leq n$ and every $1\leq \ell \leq n-1$. 
        Hence~$\r(\Rr[\mathcal{M}_\G \times \widehat{\sym}])$ is Archimedean over $\r (\Rr[\mathcal{M}_\G])$.
        Since $(\r,\Ff_\r)$ is a closed point of the real spectrum, $\Ff_\r$ is Archimedean over $\r (\Rr[\mathcal{M}_\G \times \widehat{\sym}])$ by Proposition~\ref{Prop: closed points and Archimedicity}.
        Hence $\Ff_\r$ is Archimedean over $\r(\Rr[\mathcal{M}_\G])$ so that
        \[
            \spec{\pi}(E)\subset \speccl{\mathcal{M}_\G(\Rr)}.
        \]        
        
        Second, we prove that $\spec{\pi}|_E$ is surjective onto \speccl{\mathcal{M}_\G(\Rr)}.
        Let $(\r^1 , \Kk_{\r^1})\in \speccl{\mathcal{M}_\G(\Rr)}$.
        Using $\Rr[\mathcal{M}_\G\times \widehat{\sym}] \cong \Rr[\mathcal{M}_\G] \otimes \Rr[\widehat{\sym}]$ \cite[Theorem 2.3.8. iii]{BCRrea}, we extend $\r^1$ trivially on $\Rr[\mathcal{M}_\G] \otimes \Rr[\widehat{\sym}]$ as
        \[
        \DefMap{\r}{\Rr[\mathcal{M}_\G ] \otimes \Rr[\widehat{\sym}]}{\Kk_{\r^1}}{\sum h_i^1 \otimes h_i^2}{\sum \rho^1(h_i^1)}
        \]
        Since $\eta(z_m(w_\r))$ is a big element of $\Kk_{\r^1}$, there exists a natural number~$k$ with $\hat{d_\d}(\mathrm{Id},x_{i,j}(w_\r),y_\ell(w_\r))<\eta(z_m(w_\r))^k$ so that $(\r,\Kk_{\r^1}) \in E$. 
        Moreover, for every $f\in \Rr[\mathcal{M}_\G ]$, as in Remark \ref{Rem: decomposition of spectral morphism}
        \[
            \speccl{\pi}\left(\r,\Kk_{\r^1}\right)(f)=f \circ \pi (\rho^1(z_m),1) = f(\rho^1(z_m)) = \rho^1(f),
        \]
        so that $\speccl{\pi}|_E$ is surjective. 

        Third, we show that $\spec{\pi}|_E$ is \G-equivariant for the trivial action of \G \ on  \speccl{\mathcal{M}_\G(\Rr)} and the \G -action on \specclf{\mathcal{M}_\G(\Rr)\times \widehat{\symr}} described in Corollary \ref{Cor: Gamma action on E}. Let $(\r,\Ff_\r)\in \specclf{\mathcal{M}_\G(\Rr)\times \widehat{\symr}}$. By Remark \ref{Rem: decomposition of spectral morphism}, there exist elements $(\r^1,\Ff_\r)\in \spec{\mathcal{M}_\G(\Rr)}$ and $(\r^2,\Ff_\r)\in \spec{\widehat{\symr}}$ such that, for every $h \in \Rr[\mathcal{M}_\G\times \widehat{\symr}]$
        \[
            \r(h)=\sum\r^1(h_i^1)\r^2(h_i^2) \quad \text{where } \beth^{-1}(h) = \sum h_i^1\otimes h_i^2 \in \Rr[\mathcal{M}_\G]\otimes_\Rr \Rr\left[\widehat{\sym}\right].
        \]
        Thus, for $\g\in \G$ and $h^1 \in \Rr[\mathcal{M}_\G]$, the following equalities hold:
        \begin{align*}
            \speccl{\pi}(\g\r)\left(h^1\right) & =\speccl{\pi}\left(\r^1,\r^2 \circ \g\right)\left(h^1\right) \\
             & = \r_1\left(h^1\right) \\
             & =\g \r^1\left(h^1\right) = \g \speccl{\pi}(\r)\left(h^1\right).
        \end{align*}
        Hence, \speccl{\pi} is \G-equivariant, as desired.
        
        Finally, we show that the fiber in $E$ of an element $(\r^1 ,\Kk_{\r^1}) \in \speccl{\mathcal{M}_\G(\Rr)}$ is homeomorphic to \specarch{\widehat{\mathcal{P}^1(n,\Kk_{\r^1})}}. Consider $(\spec{\pi})^{-1}(\r^1,\Kk_{\r^1})$ endowed with the subspace topology induced by the spectral topology on \specf{\mathcal{M}_\G(\Rr) \times \widehat{\symr}}. By Proposition \ref{Prop: fiber of spectral map}, there exists a canonical homeomorphism 
        \begin{align*}
            \left(\spec{\pi}\right)^{-1}(\r^1,\Kk_{\r^1}) & \cong \specff{\Kk_{\r^1} \otimes_{\Rr[\mathcal{M}_\G]} \Rr\left[\mathcal{M}_\G\times \widehat{\sym}\right]} \\
            &\cong \specff{\Kk_{\r^1} \otimes_{\Rr[\mathcal{M}_\G]} \Rr\left[\mathcal{M}_\G\right] \otimes_\Rr \Rr\left[\widehat{\sym}\right]} \\
            &\cong \specff{\Kk_{\r^1}\left[\widehat{\sym}\right]} = \spec{\widehat{\mathcal{P}^1(n,\Kk_{\r^1})}},
        \end{align*}
        where the second homeomorphism comes from \cite[Theorem 2.8.3. (iii)]{BCRrea}.
        We examine the intersection of this preimage with $E$. 
        Consider 
        \[
            (\r,\Ff_\r) \in \left(\spec{\pi}\right)^{-1}(\r^1,\Kk_{\r^1}) \cap E
        \] 
        and its decomposition $(\r,\Ff_\r) = (\r^1,\r^2,\Ff_\r)$, as in Remark \ref{Rem: decomposition of spectral morphism}.  
        By Equations \ref{Equation: Archimedean equation for xij} and \ref{Equation: Archimedean equation for yell} above, $\r(\Rr[\mathcal{M}_\G \times \widehat{\sym}])$ is Archimedean over $\r^1(\Rr[\mathcal{M}_\G]) \subset \Kk_{\r^1}$. In particular, $\Ff_\r$ is Archimedean over $\Kk_{\r^1}$ so that $(\r^2,\Ff_\r) \in \specarch{\Kk_{\r^1}[\widehat{\sym}]}$. 
        Hence
        \[
            \left(\spec{\pi}\right)|_E^{-1}(\r,\Ff_\r) \subset \specarch{\widehat{\mathcal{P}^1(n,\Kk_{\r^1})}}. 
        \]
        We prove the remaining inclusion. Consider $(\r^2,\Kk_{\r^2}) \in \specarch{\widehat{\mathcal{P}^1(n,\Kk_{\r^1})}}$. Then $\func{\r^2|_{\Kk_{\r^1}}}{\Kk_{\r^1}}{\Kk_{\r^2}}$ is an ordered field morphism by Remark \ref{Rem: K is a subfield of F}. Then
        \[
            \left(\r^2|_{\Kk_{\r^1}} \circ \r^1,\r^2,\Kk_{\r^2}\right)\in \specclff{\mathcal{M}_\G(\Rr) \times \widehat{\symr}}
        \]
        as $\Kk_{\r^2}$ is Archimedean over $\Kk_{\r^1}$ which is Archimedean over $\r^1(\Rr[\mathcal{M}_\G])$. Moreover, $\spec{\pi}(\r^2|_{\Kk_{\r^1}} \circ \r^1,\r^2,\Kk_{\r^2}) = (\r^2|_{\Kk_{\r^1}} \circ \r^1,\Kk_{\r^2}) = (\r^1,\Kk_{\r^1})$ so
        \[
            \left(\speccl{\pi}\right)^{-1}(\r^1,\Kk_{\r^1}) \cong \specarch{\widehat{\mathcal{P}^1(n,\Kk_{\r^1})}}.
        \]
        
    \end{proof}

    \begin{Rem}
        The map \speccl{\pi} is \G-equivariant, so its fibers are \G-invariant. For $(\r^1,\Kk_{\r^1})\in \speccl{\mathcal{M}_\G(\Rr)}$, this induces a \G-action on \specarch{\widehat{\mathcal{P}^1(n,\Kk_{\r^1})}} such that, for $\g \in \G$, $(\r^2,\Kk_{\r^2})\in (\speccl{\pi})^{-1}(\r^1,\Kk_{\r^1})$, and $h\in \Kk_{\r^1}[\widehat{\sym}]$ we have:
        \[
            \g \r^2(h) = \r^2(\r^1(\g)h).
        \]
    \end{Rem}

    As in Subsection \ref{Subsection: The universal projective space}, fibers over points in the real spectrum compactification of $\mathcal{M}_\G(\Rr)$ behave well with respect to the projection. 
    Informally, we show that for a sequence of elements in $\speccl{\mathcal{M}_\G(\Rr)}$ that converges to some element~$(\r,\Kk_\r)$, the fibers of the elements in the sequence converge to the fiber of the element $(\r,\Kk_\r)$. This highlights some similar behavior between \specclf{\mathcal{M}_\G(\Rr) \times \widehat{\symr}} and $\speccl{\mathcal{M}_\G(\Rr)} \times \speccl{\widehat{\symr}}$. 
    
    \begin{Cor} \label{Cor: convergence of the fibers in symmetric space}
        Let $(\r^1_n,\Kk_{\r^1_n}) \subset \speccl{\mathcal{M}_\G(\Rr)}$ be a sequence converging to an element $(\r^1,\Kk_{\r^1}) \in \speccl{\mathcal{M}_\G(\Rr)}$. For every element of the fiber $(\r^1 , \r^2,\Ff) \in (\spec{\pi}|_{E})^{-1}(\r^1,\Kk_{\r^1})$ and every open set in~$E$ around this element, \tes \ $M\in \Nn$ and $(\r^1_M,\r^2_M,\Ff_M)$ in the intersection $(\spec{\pi}|_{E})^{-1}(\r^1_M,\Kk_{\r^1_M})\cap U$.
    \end{Cor}
    
    \begin{proof}
        By Theorem \ref{Thm: universal symmetric space over the real spectrum}$\spec{\pi}|_E$ verifies all the conditions from Theorem \ref{Thm: convergence of the fibers in E general}. Thus, the statement is a consequence of the more general Theorem \ref{Thm: convergence of the fibers in E general}.
    \end{proof}

    We constructed a universal symmetric space over \rsp{\mathcal{M}_\G(\Rr)}. This provides a framework for studying the actions induced by elements of \rsp{\mathcal{M}_\G(\Rr)} on $\widehat{\symk}$ within \specarch{\widehat{\symk}}. 
    Furthermore, an element $(\r,\Kk) \in \speccl{\mathcal{M}_\G(\Rr)}$ induces an action of \G \ on the nonstandard symmetric space~$\symk$ and its associated building $B_\Kk$, see \cite[Section 5]{BIPPthereal} and \cite[Theorem 8.1]{Ageneralizedaffinebuildings}. This action of \G \ on both spaces is well-described by the actions induced by sequences $(\r_n,\Rr) \in \speccl{\mathcal{M}_\G(\Rr)}$ that converge to $(\r,\Kk)$. However, unlike the Archimedean spectrum, \symk \ and $B_\Kk$ are not locally compact.
    In contrast, $\specarch{\widehat{\symk}}$ provides a locally compact space containing $\widehat{\symk}$, offering a suitable framework to study the geometry and dynamics of elements in $\speccl{\mathcal{M}_\G(\Rr)}$. Consequently, the next section focuses on exploring the relationship between the Archimedean spectrum of an algebraic set and its Berkovich analytification, using real analytification as an intermediary step.

%% file: Berkovich_and_real_analytification.tex
\subsection{Berkovich and real analytifications of algebraic sets}\label{subsection: Real analytification}

    We present the material necessary to study the real and Berkovich analytifications applied to our context. For a more detailed study of real and Berkovich analytifications, we refer to the texts \cite{Bspe}, \cite{BRpot}, and \cite{JSYrea}.
    
    In this subsection, \Kk \ is a real closed field non-Archimedean over the reals with a non-trivial absolute value $|\cdot|_\Kk$ which is compatible with the order. That is, $\func{|\cdot|_\Kk}{\Kk}{\Rr}$ is a non-Archimedean \emph{absolute value} if for every $h,k\in \Kk$:
    \begin{itemize}
        \item $|k|_\Kk\geq 0$ with equality \iff \ $k=0$,
        \item $|hk|_\Kk = |h|_\Kk|k|_\Kk$,
        \item $|h+k|_\Kk \leq \max\set{|h|_\Kk,|k|_\Kk}.$
    \end{itemize}
    Also,~$A$ is a \Kk -algebra and $V = \specn{A}$ is an affine \Kk -variety. 
    
    \begin{Rem}
        In the context of subsection \ref{Section relation between the real spectrum of minimal vectors and the Archimedean spectrum of their associated symmetric space}, $\Ll \subset \Kk$ are real closed fields where \Kk \ has a big element $b$, and $V\subset \Ll^n$ is an algebraic set. Then $A = \Kk[V]$ is a \Kk-algebra and $V(\Kk) = \specn{\Kk[V]}$ is an affine \Kk -variety.
        Moreover, for every $h \in \Kk$ the two subsets of $\Qq$
        \[
           \setrelfrac{\frac{m}{n}}{b^m\leq h^n, n\in \Nn_{\geq 0}, m\in \Zz} \text{ and } \setrelfrac{\frac{m'}{n'}}{b^{m'}\geq h^{n'}, n'\in \Nn_{\geq 0}, m'\in \Zz}
        \]
        define a Dedekind cut of \Qq. Hence the two subsets above define a real number denoted $\log_b(h)$. Then 
        \[
            \nu (h) := e^{- \log_b |h|}, \ \text{where }  |h|:=\max\set{h,-h}
        \]
        is a non-trivial \emph{order compatible absolute value} on \Kk \ \cite[Section 5]{Bthe}.
    \end{Rem}
    
    A \emph{multiplicative seminorm} on a \Kk -algebra~$A$ is a map \func{\eta}{A}{\Rr_{\geq 0}} \st 
    \begin{itemize}
        \item $\eta(k)=|k|_{\Kk}$ \fev \ $k\in \Kk$,
        \item  $\eta(fg)=\eta(f)\eta(g)$ \fev \ $f,g\in A$, 
        \item $\eta(f+g)\leq \max{\set{\eta(f),\eta(g)}}$ \fev \ $f,g\in A$. 
    \end{itemize}
    A \emph{signed multiplicative seminorm} on a \Kk -algebra~$A$ is a map \func{\eta}{A}{\Rr} \st 
    \begin{itemize}
        \item $\eta(k)=sgn(k)|k|_{\Kk}$ \fev \ $k\in \Kk$,
        \item  $\eta(fg)=\eta(f)\eta(g)$ \fev \ $f,g\in A$, 
        \item $\min{\set{\eta(f),\eta(g)}}\leq\eta(f+g)\leq \max{\set{\eta(f),\eta(g)}}$ \fev \ $f,g\in A$. 
    \end{itemize}

    We now define the Berkovich analytification and the real analytification of an affine \Kk -variety.
    
    \begin{Def}[{\cite[Definition 1.5.1]{Bspe} and \cite[Definition 3.3]{JSYrea}}] \label{Def: Berkovich analytification}
        Let $A$ be a \Kk -algebra. The \emph{Berkovich} \emph{analytification} of $V=\specn{A}$ is the topological space 
        \[
            \mathcal{M}(A):= \setrel{\func{\eta}{A}{\Rr_{\geq 0}}}{\eta \text{ is a multiplicative seminorm}},
        \]
            with the coarsest topology that makes the evaluation map
        \[
            \map{\mathcal{M}(A)}{\Rr_{\geq 0}}{\eta}{\eta(f)}
        \]
        continuous for every element $f$ in $A$.

        Similarly, the \emph{real analytification} of $V$ is 
        \[
            \mathcal{M}_{\mathrm{R}}(A):= \setrel{\func{\eta}{A}{\Rr}}{\eta \text{ is a signed multiplicative seminorm}},
        \]
        with the coarsest topology that makes the evaluation map 
        \[
            \map{\mathcal{M}_{\mathrm{R}}(A)}{\Rr}{\eta}{\eta(f)}
        \]
        continuous for every element $f$ in $A$.
    \end{Def}

    Both the Berkovich and the real analytifications have definitions in terms of ideals in the affine variety and absolute value on the residue field. This gives a relation between the spectrum of an algebraic set and its analytification.

    \begin{Prop}[{\cite[Remark 3.4.2]{Bspe}}] \label{Prop: homeomoprhism between Berkovich analytification and Berkovich spectrum}
        The Berkovich analytification of $V=\specn{A}$ is in bijective correspondence with 
        \[
            \ana{V}:= \left\{
                    (p, |\cdot|_p) \;\middle|\;
                    \begin{array}{l}
                    p \text{ is a prime ideal in } A, \\
                    |\cdot|_p \text{ an absolute value on } \mathrm{Frac}(A/p)
                    \text{ extending } |\cdot|_\Kk
                    \end{array}
                    \right\}.
        \]
        Moreover, if $\ana{V}$ is endowed with the coarsest topology such that
        \[
            \defmap{supp}{\ana{V}}{V}{(p,|\cdot|_p)}{p}
        \]
        is continuous and the map 
        \[
        \map{supp(U)^{-1}\subset \ana{V}}{\Rr_{\geq 0}}{(p,|\cdot|_p)}{|f|_p}
        \]
        is continuous for every open $U \subset V$ and every regular map $f$ on $U$, then the bijection is a homeomorphism.
    \end{Prop}
    
    \begin{Prop}[{\cite[Proposition 3.4]{JSYrea}}]
        The real analytification of $V=\specn{A}$ is in bijective correspondence with the set 
        \[
            \anar{V} := \left\{
                (p, |\cdot|_p, <_p) \;\middle|
                \begin{array}{l}
                p \text{ a prime ideal in } A, \\
                |\cdot|_p \text{ an absolute value on } \mathrm{Frac}(A/p) 
                \text{ extending } |\cdot|_\Kk, \\
                <_p \text{ an order on } \mathrm{Frac}(A/p) 
                \text{ compatible with } |\cdot|_\Kk
                \end{array}
                \hspace{-0.25em}\right\}.
        \]
        Furthermore, if the set of triples is endowed with the coarsest topology such that 
        \[
            \defmap{supp}{\anar{V}}{V}{(p,|\cdot|_p,<_p)}{p}
        \]
        is continuous and the map 
        \[
        \map{supp(U)^{-1}\subset \anar{V}}{\Rr}{(p,|\cdot|_p,<_p)}{\mathrm{sgn}_p(f)|f|_p}
        \]
        is continuous for every open $U \subset V$ and every regular map $f$ on $U$, then the bijection is a homeomorphism.
    \end{Prop}

    These two spaces represent distinct structures of algebraic varieties. However, they are connected by the following proposition.
    
    \begin{Prop}[{\cite[Lemma 3.9]{JSYrea}}] \label{prop: map real analytification to analytification} 
        The map 
        \[
            \map{\anar{V}}{\ana{V}}{(p,|\cdot|_p,<_p)}{(p,|\cdot|_p)}
        \]
        is a proper continuous map of topological spaces.
    \end{Prop}
    
    Both spaces are crucial in the study of analytic properties of ultrametric spaces. Although ultrametric spaces are totally disconnected, their Berkovich analytification provides a method to embed them within a uniquely path connected space.
    
    \begin{Prop}[{\cite[Theorem 3.4.8]{Bspe}}] \label{Prop: embedding in the Berkovich analytification}
        The Berkovich analytification of $V=\specn{A}$ is
        \begin{itemize}
            \item a locally compact Hausdorff space which is locally contractible.
            \item uniquely path connected if and only if~$V$, with the Zariski topology, is connected.
        \end{itemize} 
        Moreover, there exists a canonical embedding from $V$ to~$\ana{V}$ and $V$ is dense in its Berkovich analytification. 
    \end{Prop}

    The real analytification also exhibits intriguing characteristics similar to those of the Berkovich analytification.

    \begin{Prop}[{\cite[Proposition 3.6, Corollary 3.18]{JSYrea}}]\label{Prop: the real analytification is Hausdorff}
        If $V$ is an affine \Kk -variety, then \anar{V} is a Hausdorff space. Moreover, \tes \ a canonical embedding from $V$ to~$\ana{V}$ and $V$ is dense in its real
        analytification. 
    \end{Prop}
    
    While the real analytification provides a topological space that retains the real structure of the studied algebraic set, it does not preserve all the favorable topological properties of the Berkovich analytification. As a result, the analytic study of real algebraic sets becomes more challenging.
    
    \begin{Prop}[{\cite[Theorem 3.20]{JSYrea}}]\label{Prop: the real analytification is totally discontinuous}
        Let $F:[0,1] \rightarrow V_R^{an}$ be a continuous map, where $V$ is an affine \Kk-variety. Then $F$ is constant. In particular, the real analytification of an affine \Kk -variety is totally path disconnected.
    \end{Prop}

    Moreover, using hyperfield theory \cite{Jgeometry}, one can show that both the Berkovich and real analytifications are functors from \Kk-algebras to topological spaces.
    
    \begin{Prop}[{\cite[Proposition 5.5]{Jgeometry}}]\label{Prop: functoriality of the real analytification}
        The Berkovich and the real analytifications are contravariant functors $\Hom{\Kk}{\cdot}{\mathbb{T}}$ and $\Hom{\Kk}{\cdot}{\Rr \mathbb{T}}$ from \Kk -algebras to topological spaces. 
    \end{Prop}

    We introduced two analytifications with desirable topological properties for affine algebraic varieties over non-Archimedean real closed fields. In the next subsection, we study in more detail the relationship between the Archimedean spectrum, the real analytification and the Berkovich analytification. 

%% file: Archimedean_spectrum_and_Berkovich_analytification.tex
\subsection{Archimedean spectrum and real analytification}

    The Archimedean spectrum of an algebraic set is homeomorphic to its real analytification \cite{JSYrea}. We use this homeomorphism to deduce topological properties of the Archimedean spectrum. Moreover, we describe the image of the Archimedean spectrum of a \Kk-algebra in its Berkovich analytification. 
    
    \begin{Thm}[{\cite[Theorem 3.17]{JSYrea}}]\label{Thm: identification Archimedean spectrum and real analytification}
        Let \Kk \ be a real closed field with a non-trivial order compatible absolute value, $A$ a \Kk -algebra and $V = \specn{A}$ an affine \Kk -variety. There exists a canonical map from~\specarch{A} to $V_R^{an}$ which is a homeomorphism. 
    \end{Thm}

    A direct consequence of Theorem \ref{Thm: identification Archimedean spectrum and real analytification} and Theorem \ref{Thm: The Archimedean spectrum is a countable union of open compact sets} is the local compactness of the real analytification.
    
    \begin{Cor}
        The real analytification of an affine \Kk -variety is a countable union of compact sets. Thus, it is \s -compact and locally compact.   
    \end{Cor}

    Furthermore, the real spectrum of an affine \Kk -variety is metrizable under a countability condition on~\Kk \ \cite[Proposition 2.33]{BIPPthereal}. This gives a real counterpart to the metrizability theorem, stated in \cite[Remark 1.5]{HLPberkovichspacesembedineuclideanspaces}  for the Berkovich analytification, which is a direct consequence of Theorem \ref{Thm: identification Archimedean spectrum and real analytification}.

    \begin{Cor}
        Let \Kk \ be a real closed field with a non-trivial absolute value and $V$ an affine \Kk-variety. If \Kk \ is countable, then the real analytification $\anar{V}$ is metrizable.
    \end{Cor}
    
    The following result is a direct consequence of Proposition \ref{Prop: the real analytification is totally discontinuous}.
    
    \begin{Cor}
        The Archimedean spectrum is a functor from \Kk -algebras to Hausdorff and totally path disconnected topological spaces.
    \end{Cor}
    
    In particular, the Archimedean spectrum of the \Kk -extension of an algebraic set is totally disconnected.
    Finally, we describe the image of the Archimedean spectrum of a \Kk -algebra in its Berkovich analytification using the maps from Theorem \ref{Thm: identification Archimedean spectrum and real analytification} and Proposition \ref{prop: map real analytification to analytification}.

    \begin{Thm}\label{Thm: image of the archimedean spectrum in the anlaytification}
        The image of $\func{\psi}{\specarch{A}}{\mathcal{M}(A)}$ is 
        \[
            \setrelfrac{\eta \in \mathcal{M}(A)}{ \eta \left(f^2_1 + \cdots + f^2_q\right) = \max_i \eta\left(f^2_i\right) \quad \forall f_1,\ldots,f_q\in A}.
        \]
    \end{Thm}
    
    \begin{proof}
        We consider two maps. The first one
        \[
        \begin{matrix}    
             \psi_1: &\specarch{A} & \rightarrow & \anar{\specn{A}} & \rightarrow & \ana{\specn{A}}, \\
              & \a & \mapsto & \left(supp\left(\a\right), \leq_{\a}, |\cdot|_{\a}\right) & \mapsto & \left(supp(\a), |\cdot|_{\a}\right)
        \end{matrix}
        \]
        where, as in Proposition \ref{Proposition: various defintion of the real spectrum}, $supp(\a)$ is the prime ideal $\a\cap -\a$ associated to the prime cone $\a$, $\leq_{\a}$ is the order on $\mathrm{Frac}(A/supp(\a))$ \st \ $0\leq_\a \overline{f}/\overline{g}$ if $fg\in \a$ and 
        \[
        \left|\overline{f}\right|_{\a}:= \sup\setrelfrac{|k|_\Kk}{k\in \Kk_{\geq 0}, \, k\leq_{\a} sgn\left(\overline{f}\right)\cdot \overline{f}}
        \]
        is an absolute value on the fraction field \cite[Lemma 3.13]{JSYrea}. This absolute value is well defined because the real closure of the fraction field of $\mathrm{Frac}(A/supp(\a))$ is Archimedean over \Kk . The second map
        \[
        \begin{matrix}    
            \psi_2: & \ana{\specn{A}} & \rightarrow & \mathcal{M}(A), \\
            & \left(supp\left(\a\right), |\cdot|_{\a}\right) & \mapsto & |\cdot|^{s}_{\a}
        \end{matrix}
        \]
        where $|\cdot|^{s}_{\a}:f\mapsto |\overline{f}|_\a$ defines the homeomorphism from Proposition \ref{Prop: homeomoprhism between Berkovich analytification and Berkovich spectrum}.
        We study the image of $\specarch{A}$ in $\mathcal{M}(A)$ using the composition $\psi:=\psi_2 \circ \psi_1$.
        For $f_1, \ldots , f_q \in A$, $\a\in \specarch{A}$ a prime cone, and by positivity of a sum of squares, it holds
        \begin{align*}
            \psi(\a)\left(f^2_1 + \cdots + f^2_q\right) &= \left|f^2_1 + \cdots + f^2_q\right|^s_{\a} \\
            &= \sup \setrelfrac{|k_1 + \cdots + k_q|_\Kk}{k_i\in \Kk_{\geq 0}, \ k_i \leq_{\a} \overline{f_i}^2, \ 1\leq i \leq q}.
        \end{align*}
        The absolute values on~\Kk \ is non-Archimedean. Thus, the following inequality holds
        \[
            |k_1 + \cdots + k_q|_\Kk \leq \max_{1\leq i \leq q} |k_i|_\Kk.
        \]
        Moreover, the inequality is an equality. Indeed, all the $k_i$ are positive and the non-Archimedean absolute value on \Kk \ is compatible with the order. Hence 
        \[
            |k_1 + \cdots + k_q|_\Kk \geq |k_i|_\Kk
        \]
        for every $1\leq i \leq q$.
        Hence $\psi(\a)\left(f^2_1 + \cdots + f^2_q\right)= \max_i |f_i^2|^s_{\a}$ so that 
        \[
        \mathrm{Im}(\psi) \subset \setrelfrac{\eta \in \mathcal{M}(A)}{ \eta \left(f^2_1 + \cdots + f^2_q\right) = \max_i \eta\left(f^2_i\right)}.
        \]
        We prove the inverse inclusion.
        Let $\eta:A \rightarrow \Rr$ be a \Kk -multiplicative seminorm so that $\eta(f^2_1 + \cdots + f^2_q)=\max_i \eta(f_i^2)$ \fev \ $f_1,\ldots , f_q \in A$. This seminorm defines a \Kk -absolute value on the real closure of the fraction field of $A/I_\eta$ via
        \[
        \left|\overline{f}\right|=\eta (f),
        \]
        where $I_\eta$ is the support of the seminorm $\eta$. Suppose $\mathrm{Frac}(A/I_\eta)$ is not orderable. Then, \tes \ $f_1,\ldots ,f_n,g_1,\ldots ,g_n \in A$ with $g_1,\ldots ,g_n \not\in I_\eta$ \st 
        \[
            -1 = \sum^n_{i=1} \left(\frac{\overline{f_i}}{\overline{g_i}}\right)^2 \text{ or equivalently } -\left( \prod_{i=1}^n \overline{g_i} \right)^2 = \sum^n_{i=1} \left(\prod_{j\neq i}\overline{f_i}\overline{g_j}\right)^2.
        \]
        Thus, using the seminorm $\eta$, we obtain
        \begin{align*}
            0=\eta(0) &= \eta\left(\sum^n_{i=1} \left(\prod_{j\neq i}f_ig_j\right)^2 + \left( \prod_{i=1}^n g_i\right)^2\right) \\
            & = \max\setfrac{\max_i \setfrac{\eta\left(\prod_{j\neq i}f_ig_j\right)^2}, \eta\left( \prod_{i=1}^n g_i\right)^2} > 0.
        \end{align*}
        The last inequality holds because every $g_i \not\in I_\eta$ so that the seminorm of their product is not zero. Now, consider an order $\leq_{\eta}$ on $\mathrm{Frac}(A/I_\eta)$ and its associated absolute value defined by 
        \[
            \left|\overline{f}\right|_\eta:=\sup \setrelfrac{|k|_\Kk}{k\in \Kk_{\geq 0}, \ k\leq_{\eta} \mathrm{sgn}\left(\overline{f}\right) \cdot \overline{f}}.
        \]
        This absolute value is compatible with the order on the fraction field so that $(I_\eta,\leq_{\eta},|\cdot|_\eta)$ 
        defines an element in $\specarch{A}$. Hence
        \[
        \mathrm{Im}(\psi) = \setrelfrac{\eta \in \mathcal{M}(A)}{ \eta \left(f^2_1 + \cdots + f^2_q\right) = \max_i \eta\left(f^2_i\right) \quad \forall f_1,\ldots,f_q\in A}.
        \]
    \end{proof}

    \begin{Def}\label{Def: real element of the analytification}
        An element in the image of $\func{\psi}{\specarch{A}}{\mathcal{M}(A)}$ is called a \emph{real element} of $\mathcal{M}(A)$.
    \end{Def}

    \begin{Cor} \label{Cor: Closed image in the analytification}
        The image $T$ of $\specarch{A}$ inside $\mathcal{M}(A)$ is closed.
    \end{Cor}

    \begin{proof}
        For a fixed finite collection $f_1,\ldots,f_q\in A$, define the map
        \[
           \DefMap{\varphi_{f_1,\ldots,f_q}}{\mathcal{M}(A)}{\Rr}{\eta}{ \eta\left(f_1^2+\cdots+f_q^2\right) - \max\setfrac{\eta\left(f_1^2\right),\ldots,\eta\left(f_q^2\right)}}
        \]
        By Definition \ref{Def: Berkovich analytification}, this map is continuous, so that 
        \[
            Z_{f_1,\ldots,f_q} := \setrelfrac{\eta \in \mathcal{M}(A)}{\varphi_{f_1,\ldots,f_q}(\eta) = 0}
        \]
        is closed for every $f_1,\ldots,f_q\in A$. Hence, by Theorem \ref{Thm: image of the archimedean spectrum in the anlaytification}
        \[
            T = \bigcap_{q \geq 1} \bigcap_{f_1,\dots,f_q \in A} Z_{f_1,\ldots,f_q},
        \]
        which is a closed subset of $\mathcal{M}(A)$.
    \end{proof}
    
    We studied the connections between the Archimedean spectrum, the real analytification, and the Berkovich analytification. In particular, we provided a description of the image of the Archimedean spectrum in the Berkovich analytification. In the next subsection, we apply this connection to the projective line to derive actions on real trees from these results.

%% file: Application_to_the_universal_projective_space.tex
\subsection{Application to the universal projective line}\label{Subsection:applications to the universal projective space}

    We prove that the image of $\specarch{\projspk}$ in $\ana{\projspk}$ is a \pslk -invariant closed uniquely path connected subset, where \Kk \ is a non-Archimedean real closed field with a big element. In particular, the image of $\specarch{\projspk}\backslash \projspk$ in $\ana{\projspk}\backslash \projspk$ is a \Rr-tree. To show this, we first describe the \Rr-tree structure of $\ana{\projspf}$ for any non-Archimedean field \Ff \ with an absolute value. For further details, we refer the reader to \cite[Subsection 4.2]{Bspe}, \cite[Section 2]{BRpot} and \cite[Section 1]{DFdegeneration}.

    Let $\Ff$ be an algebraically closed non-Archimedean complete field with an absolute value $|\cdot|_\Ff$.
    Following Definition \ref{Def: Berkovich analytification}, the space $\ana{\mathbb{A}^{1}(\Ff)}$ has a natural partial order defined by
    \[
        \eta_1 \leq \eta_2 \ \iff \ \eta_1(f) \leq \eta_2(f) \ \fev \ f \in \Ff[x].
    \]
    With this order, any two elements $\eta_1, \eta_2 \in \ana{\mathbb{A}^{1}(\Ff)}$ admit a unique least upper bound $\eta_1 \vee \eta_2 \in \ana{\mathbb{A}^{1}(\Ff)}$, see Example \ref{Ex: path in analytification using closed balls}.
    By viewing $\ana{\projspf}$ as the one point compactification $\ana{\mathbb{A}^{1}(\Ff)} \cup \set{\infty}$, we extend the partial order from $\ana{\mathbb{A}^{1}(\Ff)}$ to $\ana{\projspf}$ by setting $\eta \leq \infty$ for every $ \eta \in \ana{\mathbb{A}^{1}(\Ff)}$ \cite[Section 2]{BRpot}. 
    For any two elements $\eta_1,\eta_2 \in \ana{\projspf}$, define the unique path joining them by
    \begin{align*}
        \ell(\eta_1, \eta_2) := & \setrelfrac{\eta_3 \in \ana{\projspf}}{\eta_1 \leq \eta_3 \leq \eta_1 \vee \eta_2} \\
        & \cup \setrelfrac{\eta_3 \in \ana{\projspf}}{\eta_2 \leq \eta_3 \leq \eta_1 \vee \eta_2}.
    \end{align*}

    \begin{Rem}
        We refer to \cite[Subsection 2.2]{BRpot} for a description of \ana{\projspf} in terms of \Ff-multiplicative seminorms on the homogeneous coordinate ring of \projspf, which provides a more natural approach for studying the action of $\mathrm{PGL}_2(\Ff)$ on \ana{\projspf}.
    \end{Rem}

    For $v \in \Ff$ and $r \in \Rr$, set $D(v,r) := \setrel{w \in \mathbb{A}^1(\Ff)}{|w - v|_{\Ff} \leq r}$.

    \begin{Thm}[{\cite[Example 1.4.4. page 18]{Bspe}}] \label{Thm: Berkovich classification}
        For every $\eta \in \ana{\mathbb{A}^{1}(\Ff)}$, there exists a nested sequence of closed disks
        \[ 
            \mathbb{A}^{1}(\Ff) \supseteq D(v_1, r_1) \supseteq D(v_2, r_2) \supseteq \cdots,
        \]
        where $v_i \in \mathbb{A}^1(\Ff)$ and $r_i \in \Rr_{\geq 0}$ such that for every $f\in \Ff[x]$
        \[
            \eta(f) = \lim_{i \to \infty} \sup_{w\in D(v_i, r_i)}|f(w)|_{\Ff}.
        \]
        Two nested sequences define the same element if and only if
        \begin{itemize}
            \item each has a nonempty intersection, and their intersections are the same,
            \item both have empty intersection, and the sequences are cofinal.
        \end{itemize}
    \end{Thm}
    
    \begin{Rem}[{\cite[Example 1.4.4. page 18]{Bspe}}] 
        If an element $\eta \in \ana{\mathbb{A}^{1}(\Ff)}$ corresponds to a nested sequence of closed disks
        \[ 
            \mathbb{A}^{1}(\Ff) \supseteq D(v_1, r_1) \supseteq D(v_2, r_2) \supseteq \cdots,
        \]
        then we denote $\eta$ by $(D(v_i,r_i))_i$. 
        In the following, we use the notation
        \[
            \mathcal{H}\left(\mathbb{A}^1(\Ff)\right):=\setrelfrac{D(v,r)\in \ana{\mathbb{A}^1(\Ff)}}{v\in \Ff, \ r\in \Rr_{>0}}.
        \]
        Moreover, for every $f\in \Ff[x]$
        \[
            \sup_{w\in D(v, r)}|f(w)|_{\Ff}= \max_n |a_n|_\Ff r^n,
        \]
        where $f(x)=\sum_n a_n(x^n-v)$ is the expansion of $f$ with center $v$.  
    \end{Rem}

    \begin{Ex}\label{Ex: path in analytification using closed balls}
        With the notation from Theorem \ref{Thm: Berkovich classification}, the unique least upper bound of $D(v_1, r_1)$ and $ D(v_2, r_2) \in \ana{\mathbb{A}^{1}(\Ff)}$ is 
        \[
            D(v_1,\max \set{r_1,r_2,|v_2-v_1|_\Ff}).
        \]
        So the unique path between $D(v_1, r_1)$ and $ D(v_2, r_2) \in \ana{\mathbb{A}^{1}(\Ff)}$ corresponds to 
        \begin{align*}
            \ell(D(v_1, r_1), D(v_2, r_2)) = & \setrel{D(v_1, r)}{r_1\leq r\leq \max \set{|v_2-v_1|_\Ff, r_2}} \\
            & \cup \setrel{D(v_2, s)}{r_2\leq s\leq \max \set{|v_2-v_1|_\Ff , r_1}},
        \end{align*}
        where we use the strong triangular inequality to argue that every point in $D(v,r)$ is the center of the closed disk $D(v,r)$.
    \end{Ex}

    Through the identification $\ana{\projspf} = \ana{\mathbb{A}^1(\Ff)}\cup \set{\infty}$, Theorem \ref{Thm: Berkovich classification} allows us to further study the topology on $\ana{\projspf}$

    \begin{Def}
        Identifying \ana{\projspf} with $\ana{\mathbb{A}^1(\Ff)}\cup \set{\infty}$, define the \emph{open and closed Berkovich discs}
        \begin{align*}
            \mathcal{D}(v, r)^- &:= \setrel{\eta \in \mathbb{A}^1(\Ff)^{an}}{\eta(x-v) < r}, \\
            \mathcal{D}(v, r) &:= \setrel{\eta \in \mathbb{A}^1(\Ff)^{an}}{\eta(x-v) \leq r},
        \end{align*}
        and we view them as subsets of \ana{\projspf}.
    \end{Def}

    \begin{Prop}[{\cite[Proposition 2.7]{BRpot}}]\label{Prop: basis of topology for the berkovich line}
        A basis for the open sets of \ana{\projspf} is given by the sets of the form
        \[ 
            \mathcal{D}(v, r)^-, \quad \mathcal{D}(v, r)^- \backslash \bigcup_{i=1}^{M} \mathcal{D}(v_i, r_i), \quad \text{and} \quad \ana{\projspf} \backslash \bigcup_{i=1}^{M} \mathcal{D}(v_i, r_i),
        \]
        where \( v, v_i \in \Ff \) and \( r, r_i \in \Rr_{>0} \). 
    \end{Prop}

    \begin{Rem}[{\cite[Lemma 2.9]{BRpot}}]\label{Rem: Convergence of type 2 points to type 4}
        Let $(D(v_i,r_i))_i$ be an element in $\ana{\projspf}$ associated to the \Ff -multiplicative seminorm $\eta$ and 
        \[
            U:=\mathcal{D}(v, r)^- \text{ or } U:=\mathcal{D}(v, r)^- \backslash \bigcup_{i=1}^{M} \mathcal{D}(w_i, s_i)
        \] 
        a basis open set around $(D(v_i,r_i))_i$.
        By Theorem \ref{Thm: Berkovich classification}, for every $1\leq i \leq M$
        \[
            \eta(x-v)=\lim_n \sup_{z\in D(v_n,r_n)} |z-v|_\Ff \text{ and }
            \eta(x-w_i)=\lim_n \sup_{z\in D(v_n,r_n)} |z-w_i|_\Ff
        \]
        such that there exists $N\in \Nn$ with $D(v_n,r_n) \in U$ for every $n\geq N$.
        Hence, the sequence of discs $D(v_i,r_i)$ converges in the Berkovich topology to the element $(D(v_i,r_i))_i$.
        If $U$ is of the form $\ana{\projspf} \backslash \cup_{i=1}^{M} \mathcal{D}(w_i, s_i)$, a similar argument proves the convergence so that 
        $\mathcal{H}(\mathbb{A}^1(\Ff))$ is dense in $\ana{\projspf}$.
    \end{Rem}

    \begin{Lem}[{\cite[Lemma 2.10]{BRpot}}]
        With the definition of paths from above, $\ana{\projspf}$ is uniquely path connected.
    \end{Lem}
    
    This property can be studied further as follows. 
    Using the notation from Theorem \ref{Thm: Berkovich classification}, define the diameter function
    \[
        \DefMap{\mathrm{diam}}{\ana{\mathbb{A}^{1}(\Ff)}}{\Rr_{\geq 0}}{(D(v_i, r_i))_i}{\lim_i r_i}
    \]
    This definition is independent of the choice of the nested sequence. As in \cite[Subsection 2.7]{BRpot}, we obtain a well defined distance 
    \[
        \DefMap{\ana{d}}{\left(\ana{\projspf}\backslash \projspf \right)^2}{\Rr_{\geq 0}}{\left((D(v_i, r_i)),(D(w_j, s_j))\right)}{\log\left(\frac{\mathrm{diam}(D(v_i, r_i)\vee D(w_j, s_j) )^2}{\mathrm{diam}(D(v_i, r_i))\mathrm{diam}(D(w_j, s_j))}\right)}
    \]
    We extend $\ana{d}$ to a pseudo-metric on $\ana{\projspf}$ by: for every $D(v_i, r_i), D(w_i, s_i) \in \ana{\projspf}$
    \begin{align*}
        \ana{d}(D(v_i, r_i),D(w_j, s_j)) &= \infty \quad \text{if } D(v_i, r_i) \neq D(w_j, s_j) \\ 
        \ana{d}(D(v_i, r_i),D(w_j, s_j)) &= 0 \quad \text{otherwise.}
    \end{align*}

    \begin{Rem}
        The topology on $\ana{\projspf}$ induced by $\ana{d}$ is stronger than the Berkovich topology.
    \end{Rem}

    \begin{Thm}[{\cite[Lemma 1.3]{DFdegeneration}}]
        The metric $\ana{d}$ turns $\ana{\projspf}\backslash \projspf$ into a complete \Rr-tree on which \pslf \ acts by isometries.
    \end{Thm}
    Finally, by applying Galois theory, this construction extends to non-algebraically closed fields, see \cite[Subsection 1.6]{DFdegeneration}. 
    Let \Kk \ be a non-Archimedean real closed field with an absolute value and $\Ff = \Kk(\sqrt{-1})$ its algebraic closure. The Galois group $\mathrm{Gal}(\Ff/\Kk)$ acts on \Ff \ by complex conjugation. Thus, $\mathrm{Gal}(\Ff/\Kk)$ acts by isometries on $\ana{\projspf}$ as
    \[
        (D(v_i,r_i))_i \mapsto (D(\overline{v_i},r_i))_i.
    \]
    By \cite[page 525]{DFdegeneration}, $\ana{\projspk}$ is then homeomorphic to $\ana{\projspf}/\mathrm{Gal}(\Ff/\Kk)$.

    \begin{Prop}[{\cite[Lemma 1.6]{DFdegeneration}}]
        The set of fixed points of the action of $\mathrm{Gal}(\Ff/\Kk)$ on $\ana{\projspf}$ is the closure of the convex hull of $\projspk$.
    \end{Prop}

    \begin{Lem}[{\cite[Proposition 1.7]{DFdegeneration}}] \label{Lem: the analytification of the projective line is a real tree}
        The space $\ana{\projspk} \backslash \projspk$, with the metric induced from $\ana{\projspf}\backslash \projspf$, is a complete \Rr-tree on which \pslk \ acts by isometries. 
    \end{Lem}

    We show that the image $T$ of \func{\psi}{\specarch{\projspk} \backslash \projspk}{\ana{\projspk}\backslash\projspk} is a \Rr-subtree which is \pslk -invariant. 
    This will follow from the description of the image of $\anar{\mathbb{A}^1(\Kk)}$ in $\ana{\mathbb{A}^1(\Kk)}$.

    \begin{Ex}[{\cite[Example 3.12]{JSYrea}}]\label{Ex: real points in the analytification of the line}
        An element $\eta \in \ana{\mathbb{A}^1(\Kk)}$ is real, see Definition \ref{Def: real element of the analytification}, if it can be represented by a sequence $(D(v_i,r_i))_i \subset \mathbb{A}^1(\Ff)^{an}$ such that $D(v_i,r_i) \cap \mathbb{A}^1(\Kk) \neq \varnothing$ for every $i$.
    \end{Ex}

    \begin{Thm} \label{Thm: real points of the line is uniquely path connected}
        The image $T_1$ of $\func{\psi_A}{\anar{\mathbb{A}^1(\Kk)}}{\ana{\mathbb{A}^1(\Kk)}}$ is uniquely path connected.
    \end{Thm}

    \begin{proof}
        Let $D(v,r), D(w,s)$ be elements of $ T_1$, where $v,w \in \Kk(\sqrt{-1})$,  $r,s\in \Rr_{\geq 0}$, and denote by $|\cdot|_{\Kk(\sqrt{-1})}$the absolute value on ${\Kk(\sqrt{-1})}$ defined by
        \[
            |k+h\sqrt{-1}|_{\Kk(\sqrt{-1})} = |\sqrt{k^2+h^2}|_\Kk.
        \]
        By Example \ref{Ex: real points in the analytification of the line} 
        \[
            D(v,r)\cap \mathbb{A}^1(\Kk) \neq \varnothing \neq D(w,s) \cap \mathbb{A}^1(\Kk).
        \]
        In particular, any $D(v,r')$ such that $r\leq r' \leq \max \set{r, |v-w|_{\Kk(\sqrt{-1})}}$ verifies $D(v,r')\cap \mathbb{A}^1(\Kk) \neq \varnothing$ and any $D(w,s')$ such that $s\leq s' \leq \max \set{s, |v-w|_{\Kk(\sqrt{-1})}}$ verifies $D(s,s')\cap \mathbb{A}^1(\Kk) \neq \varnothing$. 
        Hence the unique path $\ell(D(v,r),D(w,s))$ in \ana{\mathbb{A}^1(\Kk)} is contained in $T_1$. 
        Moreover, by Example \ref{Ex: real points in the analytification of the line}, any element $(D(v_i,r_i))_i \in T_1$ verifies 
        \[
            D(v_i,r_i) \cap \mathbb{A}^1(\Kk) \neq \varnothing \quad \fev \ i\in \Nn.
        \]
        So, by Proposition \ref{Prop: basis of topology for the berkovich line} and Remark \ref{Rem: Convergence of type 2 points to type 4}, the real point $(D(v_i,r_i))_i$ is the limit of the disks $D(v_i,r_i)$, which are real. 
        In particular, every element of $T_1$ is contained in the closure of $\mathcal{H}(\mathbb{A}^1(\Ff))$. Thus, $T_1$ is a connected subset of a uniquely path connected set. Hence, $T_1$ is a uniquely path connected subspace of $\ana{\mathbb{A}^1(\Kk)}$.
    \end{proof}

    \begin{Cor}\label{Cor: image of projective space is uniquely path connected}
        The image of $\func{\psi}{\specarch{\projspk}}{\ana{\projspk}}$ is uniquely path connected.
    \end{Cor}

    \begin{proof}
        By the Chinese remainder Theorem, the coordinate ring of $\projspk = \mathbb{A}^1(\Kk) \cup \set{\infty}$ is the product $\Kk[\mathbb{A}^1]\times \Kk$. By \cite[Subsection 13.4.1]{DSTspe}
        \[
            \specarch{\projspk}=\specarchff{\Kk\left[\mathbb{A}^1\right]\times \Kk}=\specarch{\mathbb{A}^1(\Kk)}\cup \set{\infty},
        \]
        so we extend continuously $\func{\psi_A}{\anar{\mathbb{A}^1(\Kk)}}{\ana{\mathbb{A}^1(\Kk)}}$ from Theorem \ref{Thm: real points of the line is uniquely path connected} to 
        \[
            \func{\psi}{\specarch{\projspk}}{\ana{\projspk}}
        \] 
        by setting $\psi(\set{\infty})=\set{\infty}$. 
        Since $\set{\infty}$ is in the closure of the real elements of $\ana{\mathbb{A}^1(\Kk)}$, the image of $\specarch{\projspk}$ in $\ana{\projspk}$ is a connected subset of a uniquely path connected space. 
        Hence, it is uniquely path connected.
    \end{proof}

    \begin{Cor}\label{Cor: image of the projective space is a real tree}
        Let \Kk \ be a non-Archimedean real closed field with a big element. 
        The image of $\specarch{\projspk}\backslash \projspk$ inside $\ana{\projspk} \backslash \projspk$ is a closed, \pslk -invariant \Rr-subtree of the space $\ana{\projspk}$. 
    \end{Cor}

    \begin{proof}
        As in Corollary \ref{Cor: image of projective space is uniquely path connected}, consider $\func{\psi}{\specarch{\projspk}}{\ana{\projspk}}$.
        The map $\psi$ sends element of $\projspk\subset \specarch{\projspk}$ to elements of $\projspk \subset \ana{\projspk}$ \cite[Example 3.12]{JSYrea}. Hence, it induces a continuous map 
        \[
            \func{\Psi}{\specarch{\projspk}\backslash \projspk}{\ana{\projspk} \backslash \projspk}.
        \]
        Since $\set{\infty}$ is the limit of real elements and the image of $\psi$ is uniquely path connected, the image of $\Psi$ is uniquely path connected. 
        In particular, 
        \[
            T:=\Psi\left(\specarch{\projspk}\backslash \projspk\right)
        \] 
        is a \Rr-subtree of $\ana{\projspk} \backslash \projspk$, which is closed by Corollary \ref{Cor: Closed image in the analytification}.
        We now show that $T$ is \pslk -invariant. Let 
        \[
            A = \begin{pmatrix}
            a & b \\
            c & d
            \end{pmatrix}\in \pslk
        \]
        and $f\in \Kk[x]$. We homogenize $f(x)^2=\sum_{i=0}^{2n} \a_i x^i$ by setting 
        \[
            g(x,y)=\sum_{i=0}^{2n} \a_i x^i y^{2n-i} \in \Kk[x,y].
        \]
        The induced action of $A$ on the homogenous polynomial is $A g(x,y)=g(ax+b,cy+d)$ such that 
        \[
            A f(x)^2 = A g(x,x) (cx+d)^{-2n}.
        \]
        In particular, $A f(x)^2 = g(ax+b,cx+d) (cx+d)^{-2n} = h(x)^2$ for some $h\in \Kk[x]$.
        Hence, for $\eta \in T$ and $f_1,\dots,f_k\in \Kk[x]$
        \begin{align*}
            A \eta \left(f_1^2+\cdots+f_k^2\right) &= \eta \left(Af_1^2+\cdots+Af_k^2\right) \\
            & =\eta\left(h_1^2+\cdots+h_k^2\right) \\ 
            & =\max_{1\le i\le k} \eta\left(h_i^2\right)= \max_{1\le i\le k} A \eta\left(f_i^2\right).
        \end{align*}
        Hence, $A\eta \in T$ so that $T$ is \pslk-invariant.
    \end{proof}